\documentclass[11pt]{article}

\usepackage[a4paper,margin=1in]{geometry}
\usepackage{amsmath,amssymb,amsthm,mathtools,mathrsfs}
\usepackage{enumitem}
\usepackage{microtype}
\usepackage{graphicx}
\usepackage[percent]{overpic}
\usepackage{xcolor}
\usepackage{placeins}
\usepackage[hyperfootnotes=false]{hyperref}

\allowdisplaybreaks

\hypersetup{
  pdftitle={Connective Constants on Nested Fractal Graphs},
  pdfauthor={Hua Qiu and Yifan Wang},
  colorlinks=true,
  linkcolor=blue!55!black,
  citecolor=blue!55!black,
  urlcolor=blue!55!black
}
\setlist{leftmargin=2em,itemsep=2pt,topsep=4pt}

\newtheorem{theorem}{Theorem}
\newtheorem{proposition}{Proposition}[section]
\newtheorem{lemma}[proposition]{Lemma}
\theoremstyle{definition}
\newtheorem{definition}[proposition]{Definition}
\theoremstyle{remark}
\newtheorem{remark}[proposition]{Remark}

\newcommand{\R}{\mathbb R}
\newcommand{\cI}{\mathcal I}
\newcommand{\cT}{\mathcal T}
\newcommand{\cW}{\mathcal W}
\newcommand{\cS}{\mathcal S}
\newcommand{\cC}{\mathcal C}
\newcommand{\cB}{\mathcal B}
\newcommand{\bZ}{\mathbf Z}
\newcommand{\bA}{\mathbf A}
\newcommand{\diam}{\operatorname{diam}}

\title{Connective Constants on Nested Fractal Graphs}
\author{HUA QIU AND YIFAN WANG}
\date{}

\begin{document}
\maketitle
\begingroup
\renewcommand{\thefootnote}{}
\makeatletter
\renewcommand{\@makefntext}[1]{\noindent #1}
\makeatother
\footnotetext{
The research of Hua Qiu was supported by the National Natural Science
Foundation of China (grants 12471087 and 12531004).}
\endgroup

\begin{abstract}
We study self-avoiding walks on the canonical one-sided graphs of
Lindstr\o m nested fractals.  We prove that the connective
constant \(\mu\) exists and identify \(\log\mu\) with the critical
inverse temperature of a finite-dimensional boundary-state
renormalization.  If the boundary-state partition vectors are bounded at
criticality, then the fixed-length counts \(c_n\) satisfy two-sided polynomial
bounds around \(\mu^n\).  We also prove that \(h\)-flexibility implies
\(c_{n+h}/c_n\to\mu^h\).  For regular polygonal \(N\)-gaskets, we derive
exact crossing recursions, determine the smallest flexibility step \(h\), and
obtain explicit algebraic connective constants for the \(6\)- and
\(9\)-gaskets.  The Vicsek graph has no flexibility step, and its
successive ratios do not converge.
\end{abstract}

\tableofcontents

\section{Introduction}

\begin{figure}[!t]
\centering
\begin{minipage}[t]{0.23\textwidth}
  \centering
\includegraphics[width=\linewidth]{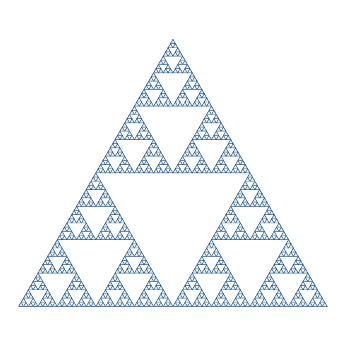}
  \par\vspace{2pt}
  \footnotesize (a) Sierpi\'nski gasket
\end{minipage}\hfill
\begin{minipage}[t]{0.23\textwidth}
  \centering
\includegraphics[width=\linewidth]{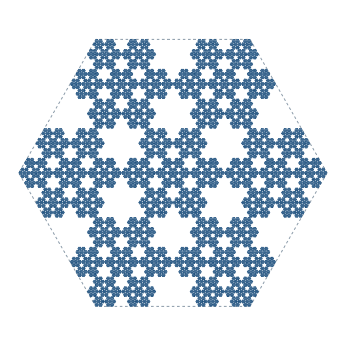}
  \par\vspace{2pt}
  \footnotesize (b) Lindstr\o m snowflake
\end{minipage}\hfill
\begin{minipage}[t]{0.23\textwidth}
  \centering
\includegraphics[width=\linewidth]{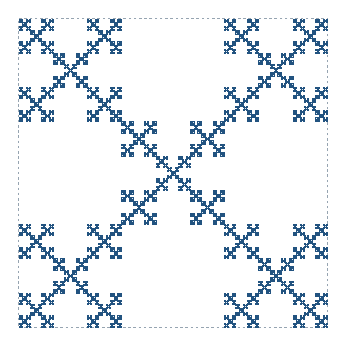}
  \par\vspace{2pt}
  \footnotesize (c) Vicsek set
\end{minipage}\hfill
\begin{minipage}[t]{0.23\textwidth}
  \centering
\includegraphics[width=\linewidth]{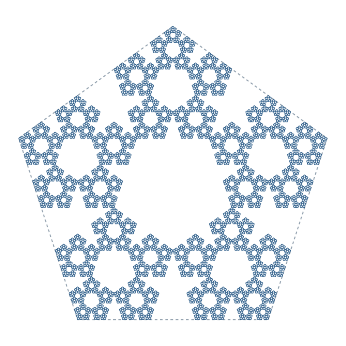}
  \par\vspace{2pt}
  \footnotesize (d) Pentagasket
\end{minipage}
\caption{Four standard nested fractals.}
\label{fig:nested-fractal-examples}
\end{figure}

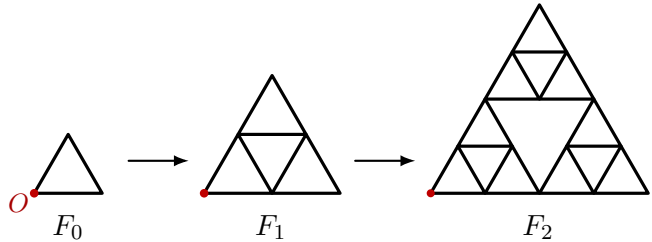
\begin{figure}[!t]
\centering
\hspace*{.6cm}%
\begin{overpic}[width=.54\textwidth]{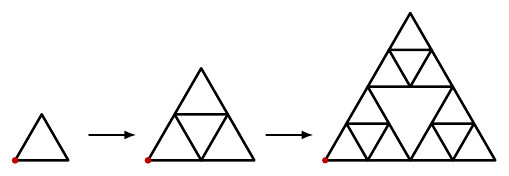}
  \put(2.2,1.5){\makebox(0,0)[r]{\small\color{red!65!black}\(O\)}}
  \put(8.2,-2.2){\makebox(0,0){\(F_0\)}}
  \put(39.4,-2.2){\makebox(0,0){\(F_1\)}}
  \put(80.5,-2.2){\makebox(0,0){\(F_2\)}}
\end{overpic}
\vspace{8pt}
\caption{The one-sided blow-up \(F_0\subset F_1\subset F_2\) for the
Sierpi\'nski gasket.  The red vertex is the fixed root \(O\).}
\label{fig:sg-blowup}
\end{figure}

Self-avoiding walks arose as lattice models for long-chain polymers.
Early systematic mathematical study and numerical estimation began with
Hammersley and Morton~\cite{HammersleyMorton1954}, and Hammersley later isolated their
exponential growth rate, now called the \emph{connective
constant}~\cite{Hammersley1957}.  It has since become one of the basic
quantities in the subject; see Madras and Slade~\cite{MadrasSlade1993}.
Let \(c_n\) denote the number of \(n\)-step self-avoiding
walks from the origin on \(\mathbb Z^d\), with \(d\ge2\).
Submultiplicativity shows that
\(\mu=\lim_{n\to\infty}c_n^{1/n}\) exists and that
\(\mu^n\le c_n\).  Hammersley and Welsh proved~\cite{HammersleyWelsh1962}
\[
  c_n
  \le
  \mu^n\exp\{O(\sqrt n)\}
\]
and Hutchcroft later replaced
\(O(\sqrt n)\) by \(o(\sqrt n)\)~\cite{Hutchcroft2018}.  The precise
subexponential correction remains unknown in dimensions \(2,3,4\);
for \(d\ge5\), the lace expansion gives
\(c_n\sim A\mu^n\)~\cite{HaraSlade1992,HaraSlade1992Lace}.
The convergence of fixed-step counting ratios is a separate and more
delicate problem.  On \(\mathbb Z^d\), Kesten's two-step ratio limit
\(c_{n+2}/c_n\to\mu^2\) is closely tied to local pattern
replacement~\cite{Kesten1963}.
Exact values are exceptionally rare: among the standard two-dimensional
lattices, the only rigorously known one is the honeycomb value
\(\mu=\sqrt{2+\sqrt2}\), proved by Duminil-Copin and
Smirnov~\cite{DuminilCopinSmirnov2012}.

Nested fractal graphs are self-similar but not transitive.  A walk rooted
at one vertex therefore cannot usually be split into two walks rooted in
the same way.  Finite ramification provides a different tool: the
connections made by a walk at the boundary of a finite graph cell can be
recorded by finitely many states.  Ben-Avraham and Havlin used this idea
to study self-avoiding walks on finitely ramified
fractals~\cite{BenAvrahamHavlin1984}.  Hattori, Hattori and Kusuoka developed an
exact renormalization for self-avoiding paths on the pre-Sierpi\'nski
gasket~\cite{HHK1990}.  Building on this special two-variable crossing
recursion, Hattori and Kusuoka established the existence of the
connective constant for the Sierpi\'nski gasket~\cite{HattoriKusuoka1992}.
See also the construction of a self-avoiding process as a renormalization
limit on the Sierpi\'nski gasket~\cite{HattoriHattori1991}, and its
higher-dimensional extensions~\cite{HHK1993,HattoriTsuda2002}.

\begin{figure}[!t]
\centering
\begin{tabular}{@{}c@{\hspace{1.95cm}}c@{\hspace{1.05cm}}c@{}}
  \raisebox{30pt}[0pt][0pt]{%
    \begin{overpic}{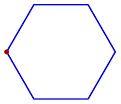}
      \put(4,52){\makebox(0,0)[r]{\small\color{red!65!black}\(O\)}}
    \end{overpic}}
  &
  \raisebox{10pt}[0pt][0pt]{%
    \begin{overpic}{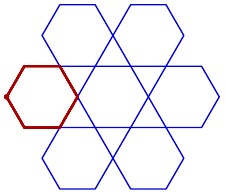}
      \put(4,52){\makebox(0,0)[r]{\small\color{red!65!black}\(O\)}}
    \end{overpic}}
  &
  \begin{overpic}{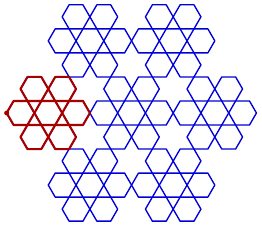}
    \put(3,52){\makebox(0,0)[r]{\small\color{red!65!black}\(O\)}}
  \end{overpic}
  \\[-1pt]
  \raisebox{22pt}[0pt][0pt]{\(F_0\)}
  &\raisebox{8pt}[0pt][0pt]{\(F_1\)}
  &\raisebox{6pt}[0pt][0pt]{\(F_2\)}
\end{tabular}
\caption{The first three finite graphs of the hexagonal snowflake.  The
red vertex is the fixed root \(O\), and the red subgraphs in \(F_1\) and
\(F_2\) identify the copy of the preceding graph.  The one-sided
construction therefore expands away from \(O\).}
\label{fig:snowflake-levels}
\end{figure}

On a general
nested fractal, there may be more boundary vertices, and the part of a
walk inside one graph cell may have several path components.  The
Sierpi\'nski gasket recursion therefore does not extend directly.
In this paper, our first result is a geometric proof that the
connective constant exists on nested fractal graphs.  The argument
combines graph cells and the Lindstr\o m reflection symmetries with a
finite convolution inequality, which bounds walks from arbitrary
starting vertices by finite products of walk counts from the
distinguished root.  In particular, the proof requires no dynamical
analysis.

For the Sierpi\'nski gasket, Hattori and Kusuoka obtained polynomial
upper and lower corrections to the exponential growth of the fixed-length
walk counts~\cite{HattoriKusuoka1992}.  Their upper bound localizes a walk
by the smallest graph cell containing it and estimates the corresponding
critical partition functions.  The rooted counts on the gasket are not
submultiplicative, so the direct lower bound available on
\(\mathbb Z^d\) has no immediate analogue.  Hattori and Kusuoka instead
derived it from a substantially more detailed analysis of the special
two-variable renormalization dynamics, including the limiting distribution
of the critical path length and the regularity of its density.  We retain
the scale-localization idea in the upper bound, but derive the lower bound
directly from the finite convolution estimate used in our existence proof,
without a separate analysis of the critical dynamics.  This is our second
result: the boundary-state
renormalization identifies its critical inverse temperature with the
logarithm of the connective constant and, whenever the partition vectors
remain bounded at criticality, gives corresponding two-sided polynomial
bounds for the fixed-length walk counts.

Our third result is a fixed-step ratio theorem.  We show that
\(h\)-flexibility forces the ratio of the walk counts
at lengths \(n+h\) and \(n\) to converge to the \(h\)-th power of the
connective constant.

For regular polygonal \(N\)-gaskets, which include the Sierpi\'nski
gasket as the case \(N=3\), we can say considerably more.  We derive
exact crossing recursions, identify the crossing critical inverse
temperature, and determine the smallest flexibility step \(h\).  As on
the honeycomb lattice~\cite{DuminilCopinSmirnov2012}, these
recursions yield explicit algebraic connective constants for \(N=6,9\).

We begin with the homogeneous iterated function system used throughout.
Fix \(d\ge1\), let
\(\cI=\{0,\ldots,M-1\}\), where \(M\ge2\), and consider contractive
similitudes
\begin{equation}
  \Psi_i(x)=\lambda^{-1}\mathsf U_i x+\nu_i,
  \qquad i\in\cI,
  \qquad \lambda>1,
  \label{eq:ifs}
\end{equation}
where each \(\mathsf U_i\) is orthogonal.  Their attractor is the unique nonempty
compact set \(K\) satisfying
\[
  K=\bigcup_{i=0}^{M-1}\Psi_i(K).
\]
The compact set \(K\) is the self-similar fractal generated by the
iterated function system.
Let \(p_i\) be the fixed point of \(\Psi_i\).  A fixed point \(p_i\) is
\emph{essential} if there exist a fixed point \(p_j\) and distinct
indices \(k,\ell\in\cI\) such that
\(\Psi_k(p_i)=\Psi_\ell(p_j)\).  Write
\[
  V_0=\{q_0,\ldots,q_{b-1}\}
\]
for the set of essential fixed points, where \(b=\#V_0\ge2\).

For a word \(w=w_1\cdots w_r\in\cI^r\) with \(r\ge1\), write
\[
  \Psi_w=\Psi_{w_1}\circ\cdots\circ\Psi_{w_r},
  \qquad
  K_w=\Psi_w(K),
  \qquad
  \partial K_w=\Psi_w(V_0).
\]
For \(r=0\), let \(\emptyset\) be the empty word, set
\(\cI^0=\{\emptyset\}\), and put
\(\Psi_\emptyset=\mathrm{id}\) and \(K_\emptyset=K\).
The set \(K_w\) is a \emph{level-\(r\) fractal cell}, and
\(\partial K_w\) is its \emph{boundary}.  The term \emph{fractal cell}
will always refer to a subset of \(K\); finite cells of the graph will
be called graph cells.  For distinct \(x,y\in V_0\),
let \(\mathscr R_{x,y}\) be reflection in the affine hyperplane
\(\{z\in\mathbb R^d:|z-x|=|z-y|\}\).

\begin{definition}[Lindstr\o m nested fractal~\cite{Lindstrom1990}]
\label{def:nested_fractal}
The attractor \(K\) of the similitudes in \eqref{eq:ifs} is a
\emph{nested fractal} if the following conditions hold.
\begin{enumerate}[label=\textup{(NF\arabic*)}]
\item There is a nonempty bounded open set \(\mathcal U\) such that
  \(\Psi_i(\mathcal U)\cap\Psi_j(\mathcal U)=\varnothing\) for \(i\ne j\) and
  \(\bigcup_i\Psi_i(\mathcal U)\subset\mathcal U\).
\item The set \(K\) is connected.
\item For every distinct \(x,y\in V_0\), the reflection
  \(\mathscr R_{x,y}\) preserves \(V_0\) and, for each \(r\ge1\),
  permutes the level-\(r\) fractal cells together with their boundaries.
\item For distinct \(i,j\in\cI\),
  \[
    \Psi_i(K)\cap\Psi_j(K)
    =
    \Psi_i(V_0)\cap\Psi_j(V_0).
  \]
\end{enumerate}
\end{definition}

Four standard examples are shown in
Figure~\ref{fig:nested-fractal-examples}.

Definition~\ref{def:nested_fractal} is the homogeneous form of
Lindstr\o m's definition.  Iterating \textup{(NF4)} shows that the same
boundary-intersection identity holds for any two distinct fractal cells
of the same level.  Much of the work on nested fractals adopts the
\emph{one-point intersection condition} as a standing axiom:
\begin{equation}
  \#\bigl(\Psi_u(V_0)\cap\Psi_v(V_0)\bigr)\le1
  \qquad
  (u\ne v,\ |u|=|v|).
  \label{eq:one-point}
\end{equation}
This condition appears in
\cite[Proposition~2.2(4)]{FitzsimmonsHamblyKumagai1994} and as
assumption \textup{(A5)} in
\cite{Kumagai1997}.  Together with \textup{(NF4)}, it says that two
distinct fractal cells of the same level are either disjoint or meet at
one common boundary point.  Whether it follows from the Lindstr\o m
axioms is unknown; see~\cite[Remark~5.25]{Barlow1998}.

On \(V_0\), join the pairs at the smallest positive Euclidean distance.
Thus, with
\[
  \delta_0=\min\{|x-y|:x,y\in V_0,\ x\ne y\},
\]
let
\[
  F_0=(V_0,E_0),
  \qquad
  E_0=\bigl\{\{x,y\}\subset V_0:|x-y|=\delta_0\bigr\}.
\]
Choose \(O\in V_0\), relabel the maps if necessary, and assume that
\(\Psi_0(O)=O\).  The finite approximating graphs and the canonical
one-sided infinite graph are
\begin{equation}
  F_m
  =
  \bigcup_{w\in\cI^m}
    \bigl(\Psi_0^{-m}\circ\Psi_w\bigr)(F_0),
  \qquad
  F_\infty=\bigcup_{m\ge0}F_m.
  \label{eq:finite-graphs}
\end{equation}
Here powers of \(\Psi_0\) denote its iterates or inverse iterates.  Put
\(V_m=V(F_m)\), \(V_\infty=V(F_\infty)\), \(v_1=\#V_1\), and
\[
  \partial F_m=\Psi_0^{-m}(V_0).
\]
This one-sided blow-up is the graph-theoretic analogue of zooming into
the fixed point \(O\).

Figures~\ref{fig:sg-blowup} and~\ref{fig:snowflake-levels} show how the
construction produces, respectively, the familiar one-sided
Sierpi\'nski gasket graph and the hexagonal snowflake graph.  In both
examples the distinguished root remains fixed while a new collection of
isometric graph pieces is added at every step.

A walk of length \(n\) is a sequence
\(w=(w(0),\ldots,w(n))\) of vertices such that
\(\{w(j-1),w(j)\}\in E(F_\infty),\ 1\le j\le n\).  We write
\(L(w)=n\), and call \(w\) \emph{self-avoiding} if its vertices are
distinct.  For \(a\in V_\infty\), define
\begin{equation}
  \begin{aligned}
  \cW_{a,n}
    &=\{w:w(0)=a,\ L(w)=n,\ w\text{ is self-avoiding}\},\\
  c_{a,n}&=\#\cW_{a,n},
  \qquad
  c_n=c_{O,n},
  \qquad
  u_n=\sup_{a\in V_\infty}c_{a,n}.
  \end{aligned}
  \label{eq:walk-counts}
\end{equation}
We set \(c_{a,0}=c_0=u_0=1\), and write
\[
  D=\sup_{x\in V_\infty}\deg_{F_\infty}(x).
\]

The two counting sequences \(c_n\) and \(u_n\) have the same exponential growth rate,
even though \(F_\infty\) need not be transitive.

\begin{theorem}
\label{thm:connective}
The limits of the rooted and uniform self-avoiding walk counts exist and
are equal:
\[
  \mu
  :=
  \lim_{n\to\infty}c_n^{1/n}
  =
  \lim_{n\to\infty}u_n^{1/n}
  =
  \inf_{n\ge1}u_n^{1/n}.
\]
Moreover,
\[
  1\le\mu\le D-1<\infty.
\]
Thus \(\mu\) is the connective constant of \(F_\infty\).
\end{theorem}

At scale \(m\), finite ramification leaves only finitely many possible
connection patterns at \(\partial F_m\).  For \(\beta\ge0\), let
\(\bZ_m(\beta)\) be the vector of their partition functions, where a
connection using \(\ell\) edges has weight \(e^{-\beta\ell}\).  Define
\[
  \cB
  =
  \left\{
    \beta\ge0:
    \sup_{m\ge0}\|\bZ_m(\beta)\|_\infty<\infty
  \right\},
  \qquad
  \beta_{\mathrm c}=\inf\cB.
\]
The precise boundary states are defined in Section~4.  Their partition
functions characterize the connective constant and, when they remain
bounded at \(\beta=\beta_{\mathrm c}\), control the remaining subexponential
factor in the walk counts.

\begin{theorem}
\phantomsection
\label{thm:critical-identity}
\begin{enumerate}[label=\textup{(\roman*)}]
\item The critical inverse temperature \(\beta_{\mathrm c}\) is finite,
and \(\mu=e^{\beta_{\mathrm c}}\).
\item If \(\beta_{\mathrm c}=\min\cB\), then there are constants
  \(A\ge1\) and \(p\ge0\) such that
  \[
  A^{-1}\mu^n n^{-p}
  \le c_n
  \le A\mu^n n^p,
  \qquad n\ge1.
  \]
\end{enumerate}
\end{theorem}

We also study regular polygonal \(N\)-gaskets.  Starting from a regular
\(N\)-gon, we place a homothetic copy of the polygon at each vertex and
choose the common
contraction ratio at which neighboring copies first meet.  When
\(N\ge3\) and \(4\nmid N\), the resulting attractor is a nested fractal.
Let \(F_\infty^{(N)}\) be its one-sided graph.  Let \(\mu_N\) and
\(\beta_{\mathrm c,N}\) be its connective constant and boundary-state
critical inverse temperature, respectively.
At each scale, the boundary-to-boundary crossings fall into finitely many
types according to the boundary vertices they visit.  We collect their
partition functions in the crossing vector
\(\mathbf C_m^{(N)}(\beta)\) (see the exact definition in Section~5).

These crossing functions identify the crossing critical inverse
temperature directly and give explicit connective constants when
\(N=6\) or \(9\).

\begin{theorem}
\label{thm:ngasket}
Let \(N\ge3\) and \(4\nmid N\).  The crossing vector of the regular
polygonal \(N\)-gasket satisfies a closed finite-dimensional polynomial
recursion.
Define its \emph{crossing critical inverse temperature} by
\[
  \beta_{\mathrm c,N}^{\mathrm{cr}}
  :=
  \inf\left\{
    \beta\ge0:
    \sup_{m\ge0}\|\mathbf C_m^{(N)}(\beta)\|_\infty<\infty
  \right\}.
\]
Then \(\mu_N=e^{\beta_{\mathrm c,N}^{\mathrm{cr}}}\).

For \(N\in\{6,9\}\), put
\(\tau=(\sqrt5-1)/2\) and \(d_N=N/3\).  Then
\begin{equation}
  \mu_N
  =
  \left(
    \frac{\sqrt{1+4\tau^{1/d_N}}-1}{2}
  \right)^{-1/d_N}.
  \label{eq:ngasket-algebraic}
\end{equation}
In particular, \(\mu_6\) and \(\mu_9\) are algebraic numbers.
\end{theorem}

The crossing recursions contain more detailed information
than is needed for Theorem~\ref{thm:ngasket}.
Appendix~\ref{app:ngasket-dynamics} gives the complete behavior of these
\(N\)-gasket recursions above, below, and at the critical value, and
Appendix~\ref{app:ngasket-rates} determines their off-critical rates.
These results are stated there as a \emph{complete phase diagram
theorem} (Theorem~\ref{thm:ngasket-dynamics}) and an
\emph{off-critical rate theorem}
(Theorem~\ref{thm:ngasket-rates}).  For the Sierpi\'nski gasket,
this dynamical viewpoint originates in the work of Hattori, Hattori and
Kusuoka~\cite{HHK1990}; among the remaining polygonal recursions, the
\(N=5\) case requires the most detailed argument in our treatment, owing
in part to its four-variable crossing recursion.

\begin{remark}
For every regular polygonal \(N\)-gasket,
Theorem~\ref{thm:ngasket-dynamics} and
Proposition~\ref{prop:ngasket-completeness} show that the boundary-state
partition vectors are bounded at \(\beta_{\mathrm c,N}\).  By
Definition~\ref{def:critical-point}, the infimum defining
\(\beta_{\mathrm c,N}\) is therefore a minimum; hence
Theorem~\ref{thm:critical-identity}(ii) applies to these
graphs.
\end{remark}

We call the graph \emph{\(h\)-flexible} if, at some finite scale, every
realizable boundary connection has two internal forest realizations whose
numbers of edges differ by \(h\).  The formal definition is given in
Section~6.

This condition yields a fixed-step ratio limit.  For regular polygonal
\(N\)-gaskets, its smallest possible step can be determined explicitly,
while the Vicsek graph shows what can happen without flexibility.

\begin{theorem}
\phantomsection
\label{thm:ratio-limit}
\begin{enumerate}[label=\textup{(\roman*)}]
\item If the nested fractal graph is \(h\)-flexible,
then
\begin{equation}
  \lim_{n\to\infty}\frac{c_{n+h}}{c_n}=\mu^h.
  \label{eq:h-step-ratio}
\end{equation}
In particular, \(1\)-flexibility implies
\[
  \lim_{n\to\infty}\frac{c_{n+1}}{c_n}=\mu.
\]
\item For the regular polygonal \(N\)-gasket with \(N\ge3\) and
\(4\nmid N\), write
\[
  N=4k_N+\ell_N,\qquad
  \ell_N\in\{1,2,3\},\qquad
  s_N=k_N+1,
\]
and put
\[
  h_N
  =
  \gcd\bigl(N-2s_N,\,2s_N-(N-2s_N)\bigr).
\]
Then \(h_N\) is the smallest flexibility step.  In fact, the graph is
\(h_N\)-flexible at scale \(2\), and
\[
  h_N=
  \begin{cases}
    2,&N\equiv2\pmod4,\\
    3,&N\equiv9\pmod{12},\\
    1,&\text{otherwise}.
  \end{cases}
\]
Consequently,
\[
  \lim_{n\to\infty}\frac{c_{n+h_N}}{c_n}=\mu_N^{h_N}.
\]
\item The standard Vicsek graph is not \(h\)-flexible for any
\(h\ge1\).  Its connective constant is
\(\mu=\sqrt2\), but, for every \(k\ge2\),
\begin{equation}
  \begin{aligned}
  c_{3^k-2}&=2^{(3^k+1)/2},\\
  c_{3^k-1}&=3\cdot2^{(3^k-3)/2},\\
  c_{3^k}&=3\cdot2^{(3^k-1)/2}.
  \end{aligned}
  \label{eq:vicsek-special-counts}
\end{equation}
Consequently,
\[
  \frac{c_{3^k-1}}{c_{3^k-2}}=\frac34,
  \qquad
  \frac{c_{3^k}}{c_{3^k-1}}=2,
\]
and the successive ratios do not converge.
\end{enumerate}
\end{theorem}

The rest of the paper is organized as follows.  Section~2 develops the
geometry of graph cells and establishes the basic properties of
\(F_\infty\).  Section~3 proves the existence of the connective constant
stated in Theorem~\ref{thm:connective}.  Section~4
introduces the boundary-state recursion and
proves both the critical identity and the two-sided
fixed-length bounds in Theorem~\ref{thm:critical-identity}.
Section~5 studies regular polygonal
\(N\)-gaskets and proves Theorem~\ref{thm:ngasket}.  Section~6 proves
the ratio-limit results, determines the smallest flexibility step of
each regular polygonal \(N\)-gasket, and gives the Vicsek counterexample in
Theorem~\ref{thm:ratio-limit}.
Appendix~\ref{app:ngasket-dynamics} gives the complete phase
diagram of the crossing recursions for regular polygonal \(N\)-gaskets,
and Appendix~\ref{app:ngasket-rates} gives their off-critical rates.

\section{Nested fractal graphs}

The reflections in (NF3) preserve Euclidean distance and \(V_0\), and
therefore act by automorphisms of the nearest-neighbor graph \(F_0\).
In this section, we record the graph-theoretic consequences
of the nested fractal axioms.

\begin{proposition}
\label{prop:basic-graph}
The graph \(F_0\) is connected.  Moreover,
\(F_m\subset F_{m+1}\) for every \(m\), and \(F_\infty\) is an infinite,
connected graph of uniformly bounded degree.
\end{proposition}

\begin{proof}
Choose \(u,v\in V_0\) with \(|u-v|=\delta_0\).  Every
\(x\in V_0\) has a neighbor in \(F_0\): this is \(v\) when \(x=u\), and
otherwise it is \(\mathscr R_{x,u}(v)\), since
\[
  |\mathscr R_{x,u}(v)-x|
    =|\mathscr R_{x,u}(v)-\mathscr R_{x,u}(u)|
    =|v-u|
    =\delta_0.
\]

If \(F_0\) were disconnected, choose \(x,y\) in different components with
\(\delta_*=|x-y|\) minimal, and let \(z\) be a neighbor of \(x\).  Then
\(|x-z|=\delta_0<\delta_*\).  Let \(H\) be the perpendicular bisector of
\(z\) and \(y\), and set
\[
  d_y=\operatorname{dist}(y,H),
  \qquad
  d_x=\operatorname{dist}(x,H).
\]
Since \(|x-z|<|x-y|\), the point \(x\) lies on the same side of \(H\)
as \(z\).  Projection onto the line perpendicular to \(H\) therefore
gives \(\delta_*\ge d_y+d_x\).  Since \(z\) and \(y\) are also in
different components,
\[
  2d_y=|z-y|\ge\delta_*.
\]
It follows that \(2d_x\le\delta_*\).  Equality would force \(x\) and \(z\)
to be the same point on the line perpendicular to \(H\), which is
impossible.  Thus \(2d_x<\delta_*\).  Since \(\mathscr R_{z,y}\) is an automorphism
of \(F_0\) and sends \(z\) to \(y\), it maps the component containing
\(x\) onto the component containing \(y\).  Therefore
\(x'=\mathscr R_{z,y}(x)\) is in the component of \(y\), whereas
\[
  |x-x'|=2d_x<\delta_*.
\]
This contradicts the minimal choice of \(x,y\).  Hence \(F_0\) is
connected.

For \(i\in\cI\), set
\[
  S_{m,i}=\Psi_0^{-(m+1)}\Psi_i\Psi_0^m.
\]
Then \(S_{m,i}\) is an isometry, \(S_{m,0}\) is the identity, and
\eqref{eq:finite-graphs} gives
\[
  F_{m+1}=\bigcup_{i=0}^{M-1} S_{m,i}(F_m).
\]
Since \(S_{m,0}\) is the identity, \(F_m\subset F_{m+1}\).
Since \(K\) is connected by (NF2), the incidence graph of the level-\(m\)
fractal cells is connected; here its vertices are the cells, and two
vertices are adjacent when the corresponding cells intersect.  Otherwise,
the union of the cells would split into two disjoint nonempty compact
sets.  By (NF4) and \eqref{eq:one-point}, two adjacent cells meet at a
common boundary vertex.  Since
each copy of \(F_0\) is connected, their union \(F_m\) is connected.  Hence
\(F_\infty\) is connected.  Moreover,
\[
  \diam(\partial F_m)=\lambda^m\diam(V_0),
\]
so \(F_\infty\) is infinite.

Finally, we prove that the degrees are uniformly bounded.  For \(r\ge0\)
and \(\alpha\in\cI^r\), call the subgraph
\(\Psi_0^{-r}\Psi_\alpha(F_0)\) an \emph{elementary copy}; it is an
isometric copy of \(F_0\) in \(F_\infty\).  Choose a closed ball
\(\overline B(x_{\mathcal U},\rho)\subset\mathcal U\).  To each elementary
copy \(\Psi_0^{-r}\Psi_\alpha(F_0)\), associate the ball
\(\Psi_0^{-r}\Psi_\alpha(B(x_{\mathcal U},\rho))\).  Because the two similarities have
reciprocal ratios, all these associated balls have radius \(\rho\).

Fix \(x\in V_\infty\), and consider any finite family of distinct
elementary copies containing \(x\).  If \(s\ge r\), then
\[
  \Psi_0^{-r}\Psi_\alpha(F_0)
  =
  \Psi_0^{-s}\Psi_0^{s-r}\Psi_\alpha(F_0).
\]
The composition \(\Psi_0^{s-r}\Psi_\alpha\) equals \(\Psi_\beta\) for a
word \(\beta\in\cI^s\), so the right-hand side is an elementary copy.
Choose \(s\) no smaller than the parameters \(r\) of all the copies in
the family.  The above identity then represents every member of the
family with the same parameter \(s\).  Their
associated balls are then pairwise disjoint by (NF1).  If
\[
  R=\max_{y\in K}|x_{\mathcal U}-y|,
\]
then the center of every associated ball is at distance at most
\(R\)
from \(x\).  Thus all these radius-\(\rho\) balls lie in
\(B(x,R+\rho)\).  Comparing \(d\)-dimensional volumes shows that the
family contains at most
\[
  C
  =
  \left(\frac{R+\rho}{\rho}\right)^d<\infty.
\]
This bound is independent of \(x\) and of the chosen family.  Every edge
incident with \(x\) belongs to an elementary copy, and one copy contributes
at most \(\max_{y\in V_0}\deg_{F_0}(y)\) incident edges.  Hence
\[
  \deg_{F_\infty}(x)
    \le C\max_{y\in V_0}\deg_{F_0}(y),
\]
uniformly in \(x\).
\end{proof}

We shall repeatedly use finite copies of \(F_m\) inside \(F_\infty\).
For \(m\ge0\), \(r\ge m\), and
\(\alpha\in\cI^{r-m}\), define
\begin{equation}
  S_{r,\alpha}^{(m)}
  =
  \Psi_0^{-r}\Psi_\alpha\Psi_0^m,
  \qquad
  \Delta_{m;r,\alpha}
  =
  S_{r,\alpha}^{(m)}(F_m),
  \qquad
  \partial\Delta_{m;r,\alpha}
  =
  \Psi_0^{-r}\Psi_\alpha(V_0).
  \label{eq:m-cell}
\end{equation}
The map \(S_{r,\alpha}^{(m)}\) has similarity ratio one and is an
isometry.  Let
\[
  \cT_m
  =
  \{\Delta_{m;r,\alpha}:
    r\ge m,\ \alpha\in\cI^{r-m}\},
\]
where repeated copies of the same subgraph are identified.  An element
of \(\cT_m\) is a \emph{graph \(m\)-cell}.  When no confusion with a
fractal cell can arise, we simply call it an \emph{\(m\)-cell}.  Thus a
graph \(m\)-cell is a finite subgraph of \(F_\infty\) isometric to
\(F_m\), including all vertices and edges of that copy, and its boundary
consists of the \(\#V_0\) vertices specified in \eqref{eq:m-cell}.

\begin{proposition}
\label{prop:cell-hierarchy}
For every \(m\ge0\), \(F_m\) is an \(m\)-cell.  Two distinct \(m\)-cells
are either disjoint or meet at one common boundary vertex.  Consequently,
every edge of \(F_\infty\) lies in exactly one \(m\)-cell.  If \(m\ge1\),
every \(m\)-cell is the union of \(M\) \((m-1)\)-cells, and the union of
the boundaries of these \(M\) subcells contains \(v_1\) vertices.
\end{proposition}

\begin{proof}
The representation in \eqref{eq:m-cell} shows that every \(m\)-cell is a
copy of \(F_m\).  If two distinct \(m\)-cells meet, the iterated form of
(NF4) says that their intersection is contained in the intersection of
their boundary sets.  By \eqref{eq:one-point}, this intersection has at
most one vertex.  Thus two distinct \(m\)-cells cannot share an edge.
Conversely, every edge of \(F_\infty\) belongs to an elementary copy
\(\Psi_0^{-r}\Psi_\alpha(F_0)\) for some \(r\ge0\).  If \(r\ge m\),
write \(\alpha=\alpha'\alpha''\), where \(\alpha'\) has length \(r-m\).
Since \(\Psi_{\alpha''}(F_0)\subset\Psi_0^m(F_m)\), we have
\[
  \Psi_0^{-r}\Psi_\alpha(F_0)
  \subset
  \Psi_0^{-r}\Psi_{\alpha'}\Psi_0^m(F_m)
  =
  \Delta_{m;r,\alpha'}.
\]
Thus the edge is contained in the \(m\)-cell
\(\Delta_{m;r,\alpha'}\).  If \(r<m\), the elementary copy is one of
the copies in the union defining \(F_r\), hence is contained in
\(F_r\subset F_m=\Delta_{m;m,\emptyset}\).  Thus the edge again lies in
an \(m\)-cell.
Hence every edge belongs to exactly one
\(m\)-cell.

If \(\Delta=\Delta_{m;r,\alpha}\) and \(m\ge1\), the decomposition in
\eqref{eq:m-cell} gives
\[
  \Delta
    =\bigcup_{i=0}^{M-1}\Delta_{m-1;r,\alpha i}.
\]
The graphs on the right are the \(M\) \((m-1)\)-subcells of \(\Delta\).
Their boundary union is the image of
\(\bigcup_i\Psi_i(V_0)\) under the injective similarity
\(\Psi_0^{-r}\Psi_\alpha\), and hence its cardinality is
\[
  \#\left(\Psi_0^{-1}\left(\bigcup_i\Psi_i(V_0)\right)\right)
    =\#V_1
    =v_1.
\]
\end{proof}

Proposition~\ref{prop:basic-graph} shows that
\(2\le D<\infty\).  Since \(F_\infty\) is infinite, connected, and
locally finite, K\"onig's infinity lemma gives an infinite
self-avoiding ray starting at \(O\).  Consequently,
\begin{equation}
  1\le c_n\le u_n\le D(D-1)^{n-1},
  \qquad n\ge1.
  \label{eq:elementary-walk-bound}
\end{equation}

\section{Existence of the connective constant}

We prove Theorem~\ref{thm:connective} in two steps.  First, the reflection
axiom gives a local operation which moves any boundary point of a finite
cell to the distinguished root.  We then combine the resulting uniform
comparison with submultiplicativity.

\subsection{\texorpdfstring{A convolution
inequality}{A convolution inequality}}

\begin{lemma}
\label{lem:boundary-rooting}
For every \(r\ge0\) and every \(x\in\partial F_r\), there is a Euclidean
isometry \(A_{r,x}\) whose restriction to \(F_r\) is a graph automorphism
and which satisfies \(A_{r,x}(x)=O\).
\end{lemma}

\begin{proof}
Write \(x=\Psi_0^{-r}(q)\) with \(q\in V_0\).  If \(q=O\), take the
identity.  Otherwise define
\begin{equation}
  A_{r,x}
    =\Psi_0^{-r}\mathscr R_{q,O}\Psi_0^r.
  \label{eq:rooting-map}
\end{equation}
When \(r=0\), the reflection \(\mathscr R_{q,O}\) preserves \(V_0\) by (NF3) and is
therefore an automorphism of the nearest-neighbor graph \(F_0\).  When
\(r\ge1\), (NF3) in Definition~\ref{def:nested_fractal} says that
\(\mathscr R_{q,O}\) permutes the level-\(r\) fractal cells and their
boundaries.  Since it preserves Euclidean distances, it also transports every copy of the
nearest-neighbor edge set \(E_0\) to the corresponding edge set.  Hence
the map in \eqref{eq:rooting-map} preserves \(F_r\) in either case.
Finally,
\[
  A_{r,x}(x)
    =\Psi_0^{-r}\mathscr R_{q,O}(q)
    =\Psi_0^{-r}(O)
    =O,
\]
because \(\mathscr R_{q,O}\) interchanges \(q\) and \(O\), and \(O\) is fixed by
\(\Psi_0\).
\end{proof}

\begin{lemma}
\label{lem:image-count}
Let
\[
  \eta=(x_0,\ldots,x_\ell),\qquad \ell\ge1,
\]
be a finite polygonal path in \(\R^d\), with consecutive vertices
distinct.  For \(y\in V_\infty\), let \(\mathscr E_y(\eta)\) be the set
of distinct
vertex sequences
\[
  (T(x_0),\ldots,T(x_\ell))
\]
such that \(T\) is a Euclidean isometry of \(\R^d\),
\(T(x_0)=y\), and
\[
  \{T(x_{j-1}),T(x_j)\}\in E(F_\infty),
  \qquad 1\le j\le\ell.
\]
Then
\[
  \#\mathscr E_y(\eta)\le D^d.
\]
The same estimate holds if \(T(x_\ell)=y\) is prescribed in place of
\(T(x_0)=y\).
\end{lemma}

\begin{proof}
Set
\[
  r(\eta)=\dim\operatorname{Aff}(x_0,\ldots,x_\ell),
\]
where \(\operatorname{Aff}(x_0,\ldots,x_\ell)\) is the smallest affine
subspace of \(\R^d\) containing the vertices of \(\eta\).  For
\(0\le j\le\ell\), put
\[
  \mathsf A_j=\operatorname{Aff}(x_0,\ldots,x_j).
\]
We construct a possible image sequence from left to right.  Suppose that
the images
\[
  T(x_0),\ldots,T(x_j)
\]
have already been fixed.  If \(x_{j+1}\in\mathsf A_j\), there are coefficients
\(\lambda_0,\ldots,\lambda_j\) such that
\[
  x_{j+1}=\sum_{i=0}^j\lambda_i x_i,
  \qquad
  \sum_{i=0}^j\lambda_i=1.
\]
Every Euclidean isometry is affine.  Hence every isometry compatible with
these previously fixed images must satisfy
\[
  T(x_{j+1})
    =\sum_{i=0}^j\lambda_iT(x_i).
\]
Thus the next image \(T(x_{j+1})\) is uniquely determined.

If \(x_{j+1}\notin\mathsf A_j\), then \(T(x_{j+1})\) must be a neighbor of
the already fixed vertex \(T(x_j)\).  There are at most \(D\) choices.
At exactly such a step,
\[
  \dim\mathsf A_{j+1}=\dim\mathsf A_j+1.
\]
Since \(\dim\mathsf A_0=0\) and
\(\dim\mathsf A_\ell=r(\eta)\), the dimension increases
exactly \(r(\eta)\) times.  A factor \(D\) occurs only at those times, and
this proves the desired estimate.  If the image of the terminal
vertex is prescribed, apply the result just proved to the reversed path
\((x_\ell,\ldots,x_0)\).
\end{proof}

Set
\[
  Q=v_1+1,
  \qquad
  C=D^d.
\]
Both constants depend only on the fixed nested fractal graph.
The next proposition is the main combinatorial estimate in the proof of
the existence of the connective constant.

\begin{proposition}
\label{prop:bounded-piece}
For every \(a\in V_\infty\) and \(n\ge1\),
\begin{equation}
  c_{a,n}
  \le
  \sum_{q=1}^{\min\{Q,n\}}C^q
  \sum_{\substack{\ell_1+\cdots+\ell_q=n\\ \ell_i\ge1}}
  \prod_{i=1}^q c_{\ell_i}.
  \label{eq:bounded-piece}
\end{equation}
\end{proposition}

The terms in \eqref{eq:bounded-piece} have a direct interpretation.  A
walk is divided into \(q\) pieces of positive lengths
\(\ell_1,\ldots,\ell_q\).  After a suitable isometry, each piece becomes
a self-avoiding walk from \(O\), which accounts for the product
\(\prod_i c_{\ell_i}\).  The number of pieces is at most \(Q\), and at
most \(C^q\) original walks can give the same
\(q\)-tuple of walks from \(O\).

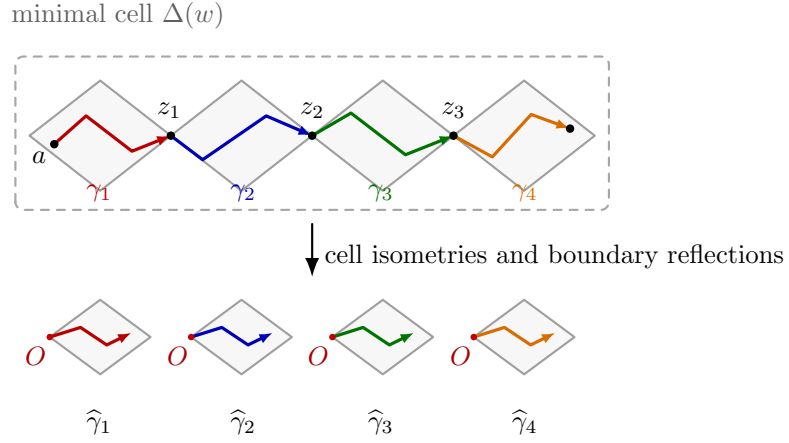
\begin{figure}[ht]
\centering
\vspace{20pt}
\begin{overpic}[width=.52\textwidth]{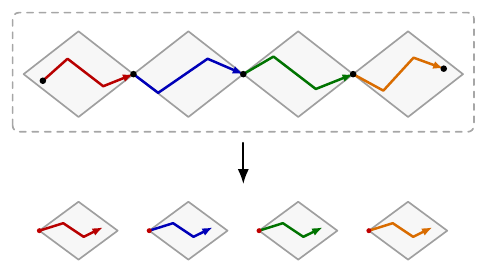}
  \put(2,60){\makebox(0,0)[l]{\small\color{gray!75!black}minimal cell \(\Delta(w)\)}}
  \put(7.5,37.2){\makebox(0,0)[r]{\small \(a\)}}
  \put(27.2,43.5){\makebox(0,0)[b]{\small \(z_1\)}}
  \put(50,43.5){\makebox(0,0)[b]{\small \(z_2\)}}
  \put(72.6,43.5){\makebox(0,0)[b]{\small \(z_3\)}}
  \put(16,31.5){\makebox(0,0){\small\color{red!72!black}\(\gamma_1\)}}
  \put(39,31.5){\makebox(0,0){\small\color{blue!72!black}\(\gamma_2\)}}
  \put(61,31.5){\makebox(0,0){\small\color{green!45!black}\(\gamma_3\)}}
  \put(84,31.5){\makebox(0,0){\small\color{orange!85!black}\(\gamma_4\)}}
  \put(52,21.5){\makebox(0,0)[l]{\small cell isometries and boundary reflections}}
  \put(7.5,5.5){\makebox(0,0)[r]{\small\color{red!65!black}\(O\)}}
  \put(30.2,5.5){\makebox(0,0)[r]{\small\color{red!65!black}\(O\)}}
  \put(52.7,5.5){\makebox(0,0)[r]{\small\color{red!65!black}\(O\)}}
  \put(75.4,5.5){\makebox(0,0)[r]{\small\color{red!65!black}\(O\)}}
  \put(16,-5){\makebox(0,0){\small\(\widehat\gamma_1\)}}
  \put(39,-5){\makebox(0,0){\small\(\widehat\gamma_2\)}}
  \put(61,-5){\makebox(0,0){\small\(\widehat\gamma_3\)}}
  \put(84,-5){\makebox(0,0){\small\(\widehat\gamma_4\)}}
\end{overpic}
\vspace{19pt}
\caption{Decomposition inside the smallest cell containing the walk.  The
upper walk is split whenever two consecutive edges lie in different
children.  The first piece is reversed, and Lemma~\ref{lem:boundary-rooting}
then moves the initial point of every piece to \(O\).}
\label{fig:decomposition}
\end{figure}
\FloatBarrier

Figure~\ref{fig:decomposition} illustrates the decomposition and the
normalization of its pieces.

\begin{proof}
Fix \(a\in V_\infty\) and \(w\in\cW_{a,n}\).  Since \(w\) has only
finitely many edges, it is contained in some \(F_s\).  By
Proposition~\ref{prop:cell-hierarchy}, \(F_s\) is an \(s\)-cell.  Hence
\[
  m(w)
    =\min\{m\ge0:\text{some }\Delta\in\cT_m
       \text{ contains every edge of }w\}
\]
is well defined.  There is only one \(m(w)\)-cell containing all the
edges of \(w\).  Indeed, two such cells would both contain the first edge
of \(w\), which is impossible by Proposition~\ref{prop:cell-hierarchy}.
Denote this unique cell by \(\Delta(w)\).

Suppose first that \(m(w)=0\).  Every vertex of the \(0\)-cell
\(\Delta(w)\) belongs to its boundary.  Choose a representation of
\(\Delta(w)\) in \eqref{eq:m-cell} and use the inverse of its isometry to
move \(w\) to \(F_0\).  Then apply
Lemma~\ref{lem:boundary-rooting} to move its initial point to \(O\).
This produces one walk in \(\cW_{O,n}\).

Now suppose that \(m=m(w)\ge1\).  By
Proposition~\ref{prop:cell-hierarchy}, the cell \(\Delta(w)\) has \(M\)
children of level \(m-1\), and every edge of \(w\) lies in exactly one of
them.  Divide \(w\) into maximal consecutive pieces for which all edges
lie in the same child.  Write
\[
  w=\gamma_1\circ\cdots\circ\gamma_q,
  \qquad \gamma_i\subset\Lambda_i,
  \qquad \Lambda_i\ne\Lambda_{i+1},
  \qquad
  L(\gamma_i)=\ell_i\ge1,
  \qquad
  \sum_{i=1}^q\ell_i=n.
\]
Here \(\Lambda_i\) is the child containing the edges of \(\gamma_i\).
The minimality of \(m\) implies \(q\ge2\), because otherwise all edges of
\(w\) would lie in one \((m-1)\)-cell.

For \(1\le i<q\), let \(z_i\) be the common endpoint of
\(\gamma_i\) and \(\gamma_{i+1}\).  The point \(z_i\) lies in both
\(\Lambda_i\) and \(\Lambda_{i+1}\).  These children are distinct, so
Proposition~\ref{prop:cell-hierarchy} shows that \(z_i\) is a boundary
vertex of both.  Moreover, the vertices \(z_1,\ldots,z_{q-1}\) are
distinct because \(w\) is self-avoiding.  The union of the boundaries of
all children of \(\Delta(w)\) has \(v_1\) vertices, again by
Proposition~\ref{prop:cell-hierarchy}.  It follows that
\[
  q-1\le v_1,
  \qquad
  q\le Q.
\]
This remains true if \(w\) leaves a child and later returns to it: each
change from one child to another occurs at a different vertex of \(w\).

The first piece need not start at a boundary vertex, but it ends at
\(z_1\in\partial\Lambda_1\).  We therefore reverse \(\gamma_1\).  For
\(i\ge2\), the original orientation is retained, and \(\gamma_i\) starts
at \(z_{i-1}\in\partial\Lambda_i\).  Thus every oriented piece now starts
at a boundary vertex of its child.  For each child, choose a representation
in \eqref{eq:m-cell}, move the corresponding piece to \(F_{m-1}\) by the
inverse of its isometry, and then use Lemma~\ref{lem:boundary-rooting}.
This gives a walk
\[
  \widehat\gamma_i\in\cW_{O,\ell_i},
  \qquad 1\le i\le q.
\]
Consequently, \(w\) determines \(q\), the positive composition
\(\ell_1+\cdots+\ell_q=n\), and the tuple
\((\widehat\gamma_1,\ldots,\widehat\gamma_q)\).

It remains to estimate how many original walks can produce a fixed tuple.
Fix \(a\), \(q\), the lengths, and the walks
\(\widehat\gamma_1,\ldots,\widehat\gamma_q\).  Each oriented original
piece is the image of the corresponding \(\widehat\gamma_i\) under a
Euclidean isometry: the required isometry is the composition of the
inverse boundary automorphism from Lemma~\ref{lem:boundary-rooting} and
the corresponding isometry in \eqref{eq:m-cell}.

When \(m\ge1\), the first
normalized walk represents the
reverse of \(\gamma_1\).  Its terminal image must therefore be \(a\).  By
Lemma~\ref{lem:image-count}, there are at most
\(C=D^d\) possible image
sequences for this piece.  Each
choice determines its other endpoint, which is \(z_1\).
The image of \(\widehat\gamma_2\) must start at this prescribed vertex
\(z_1\), so the same lemma gives at most \(C\) choices for the second
piece.  Each of those choices determines \(z_2\).  Repeating the argument
piece by piece gives at most \(C^q\) possible sequences of pieces.  When
\(m=0\), Lemma~\ref{lem:image-count} is applied with the initial image
fixed at \(a\), and gives the same bound \(C\) for the
single piece.

Once the images of all pieces are fixed, their concatenation determines the
entire walk \(w\), after the first piece is returned to its original
orientation.  Some of the image sequences counted above may not lie in
the required children, and some concatenations may fail to be
self-avoiding.  Keeping these extra possibilities only enlarges the
upper bound.  For fixed \(q\) and fixed lengths, there are
\(\prod_{i=1}^q c_{\ell_i}\) possible tuples of walks from \(O\).
Finally, summing over \(1\le q\le\min\{Q,n\}\) and over all positive
compositions of \(n\) proves \eqref{eq:bounded-piece}.
\end{proof}

\begin{remark}
Every product on the right-hand side of \eqref{eq:bounded-piece}
contains at most \(Q\) rooted factors; hence at least one factor has
length at least \(\lceil n/Q\rceil\).  This is essential for the lower bound in
Theorem~\ref{thm:critical-identity}(ii).
\end{remark}

\subsection{\texorpdfstring{Submultiplicativity and the
root limit}{Submultiplicativity and the root limit}}

The numbers \(c_n\) are tied to the distinguished vertex
\(O\), and there is no evident submultiplicative inequality for them.
The uniform counts \(u_n\), however, do satisfy such an inequality.

\begin{lemma}
\label{lem:uniform-submultiplicativity}
For all \(r,s\ge0\),
\begin{equation}
  u_{r+s}\le u_ru_s.
  \label{eq:uniform-submultiplicativity}
\end{equation}
Consequently, the limit
\[
  \gamma
    =\lim_{n\to\infty}\frac{1}{n}\log u_n
    =\inf_{n\ge1}\frac{1}{n}\log u_n
\]
exists and belongs to \([0,\log(D-1)]\).
\end{lemma}

\begin{proof}
Fix \(a\in V_\infty\), and cut a walk in \(\cW_{a,r+s}\) after its first
\(r\) edges.  There are at most \(u_r\) choices for the first part.  Once
that part is fixed, its endpoint is fixed as well.  If we ignore the
condition that the remaining \(s\) edges must avoid the vertices already
visited, there are at most \(u_s\) choices for the second part.  Thus
\(c_{a,r+s}\le u_ru_s\).  Taking the supremum over \(a\) proves
\eqref{eq:uniform-submultiplicativity}.

It follows that \(\log u_n\) is subadditive.  Fekete's lemma therefore
gives the asserted limit and the formula for its infimum.  Finally,
\eqref{eq:walk-counts} and the degree estimate in
Proposition~\ref{prop:basic-graph} give
\(1\le u_n\le D(D-1)^{n-1}\), which yields
\(0\le\gamma\le\log(D-1)\).
\end{proof}

We now return to the walks from \(O\).  Taking the supremum over the
starting vertex \(a\) in Proposition~\ref{prop:bounded-piece},
specifically in \eqref{eq:bounded-piece}, gives
\begin{equation}
  u_n
  \le
  \sum_{q=1}^{\min\{Q,n\}}C^q
  \sum_{\substack{\ell_1+\cdots+\ell_q=n\\ \ell_i\ge1}}
  \prod_{i=1}^q c_{\ell_i}.
  \label{eq:uniform-convolution}
\end{equation}
The important point in \eqref{eq:uniform-convolution} is that the number
of factors is bounded by the fixed constant \(Q\), independently of
\(n\).  We use this to compare the exponential growth of \(c_n\) with
that of \(u_n\).

\begin{lemma}
\label{lem:fixed-convolution}
The convolution estimate \eqref{eq:uniform-convolution}, together with
Lemma~\ref{lem:uniform-submultiplicativity}, implies
\[
  \lim_{n\to\infty}\frac{1}{n}\log c_n=\gamma.
\]
\end{lemma}

\begin{proof}
Fix \(\varepsilon>0\).  By
Lemma~\ref{lem:uniform-submultiplicativity},
\(j^{-1}\log u_j\to\gamma\).  Hence the finitely many small values of
\(j\) can be absorbed into a constant \(A_\varepsilon\ge1\) so that
\[
  u_j\le A_\varepsilon
       \exp\bigl((\gamma+\varepsilon)j\bigr)
  \qquad j\ge0.
\]

For a fixed \(q\), there are
\(\binom{n-1}{q-1}\) positive compositions of \(n\) into \(q\) parts.
Consequently, the sum of all coefficients multiplying the products on
the right-hand side of \eqref{eq:uniform-convolution} is
\begin{equation}
  G_n
    =\sum_{q=1}^{\min\{Q,n\}}
       C^q\binom{n-1}{q-1}
    \le QC^Q n^{Q-1},
  \label{eq:weight-bound}
\end{equation}
where we used \(C\ge1\) and
\(\binom{n-1}{q-1}\le n^{Q-1}\).  By
\eqref{eq:uniform-convolution}, not every product on its right-hand side
can be smaller than \(u_n/G_n\).  Therefore there are
\(q\le Q\) and positive integers
\(\ell_1,\ldots,\ell_q\), with
\(\ell_1+\cdots+\ell_q=n\), such that
\[
  \prod_{i=1}^q c_{\ell_i}\ge\frac{u_n}{G_n}.
\]

Let \(k=\max_i\ell_i\), and choose one factor of this length.  Since there
are at most \(Q\) parts,
\[
  k\ge\left\lceil\frac{n}{q}\right\rceil
  \ge\left\lceil\frac{n}{Q}\right\rceil.
\]
For every other part, use \(c_{\ell_i}\le u_{\ell_i}\) and the preceding
upper bound for \(u_{\ell_i}\).  The sum of those other lengths is
\(n-k\).  We obtain
\begin{equation}
  c_k
  \ge
  \frac{u_n}
       {G_nA_\varepsilon^{q-1}
        \exp\bigl((\gamma+\varepsilon)(n-k)\bigr)}.
  \label{eq:long-factor}
\end{equation}

We next compare this long rooted walk with a rooted walk of any shorter
length.  If \(0\le t\le k\), cutting a walk in \(\cW_{O,k}\) after
\(t\) edges gives
\[
  c_k\le c_tu_{k-t}.
\]

We now choose \(n=Qt\).  This choice guarantees that the length \(k\)
selected above satisfies \(k\ge\lceil n/Q\rceil=t\), so the last
inequality can be applied.  Combining it with \eqref{eq:long-factor} and with
\[
  u_{k-t}\le
  A_\varepsilon
  \exp\bigl((\gamma+\varepsilon)(k-t)\bigr)
\]
gives
\begin{equation}
  c_t
  \ge
  \frac{u_{Qt}}
       {G_{Qt}A_\varepsilon^Q
        \exp\bigl((\gamma+\varepsilon)(Q-1)t\bigr)}.
  \label{eq:rooted-lower-bound}
\end{equation}
Indeed, the two exponential losses combine because
\[
  (Qt-k)+(k-t)=(Q-1)t
\]
and the powers of \(A_\varepsilon\) are bounded by
\(A_\varepsilon^Q\), since \(q\le Q\).

Taking logarithms in \eqref{eq:rooted-lower-bound} and dividing by \(t\)
gives
\[
  \frac{1}{t}\log c_t
  \ge
  \frac{1}{t}\log u_{Qt}
  -\frac{1}{t}\log G_{Qt}
  -\frac{Q}{t}\log A_\varepsilon
  - (Q-1)(\gamma+\varepsilon).
\]
By Lemma~\ref{lem:uniform-submultiplicativity},
\[
  \frac{1}{t}\log u_{Qt}\longrightarrow Q\gamma.
\]
Moreover, \eqref{eq:weight-bound} shows that \(G_{Qt}\) grows
polynomially in \(t\), and hence
\[
  \frac{1}{t}\log G_{Qt}\longrightarrow0,
  \qquad
  \frac{Q}{t}\log A_\varepsilon\longrightarrow0.
\]
Taking the lower limit in the preceding inequality therefore yields
\[
  \liminf_{t\to\infty}\frac{1}{t}\log c_t
    \ge Q\gamma-(Q-1)(\gamma+\varepsilon)
    =\gamma-(Q-1)\varepsilon.
\]
Letting \(\varepsilon\downarrow0\) gives
\[
  \liminf_{t\to\infty}\frac{1}{t}\log c_t\ge\gamma.
\]
On the other hand, \(c_t\le u_t\) by \eqref{eq:walk-counts}, and
Lemma~\ref{lem:uniform-submultiplicativity} therefore gives
\[
  \limsup_{t\to\infty}\frac{1}{t}\log c_t\le\gamma.
\]
The two inequalities prove the lemma.
\end{proof}

\begin{proof}[Proof of Theorem~\ref{thm:connective}]
By Lemmas~\ref{lem:uniform-submultiplicativity} and~
\ref{lem:fixed-convolution},
\[
  \lim_{n\to\infty}c_n^{1/n}
    =\lim_{n\to\infty}u_n^{1/n}
    =e^\gamma
    =\inf_{n\ge1}u_n^{1/n}.
\]
Moreover, Proposition~\ref{prop:basic-graph} and K\"onig's infinity lemma
give an infinite self-avoiding ray from \(O\), so \(c_n\ge1\).  The same
proposition gives the degree bound
\(c_n\le D(D-1)^{n-1}\).  Taking \(n\)-th roots proves
\(1\le\mu\le D-1<\infty\).
\end{proof}

\section{The critical inverse temperature}

In this section, Subsections~4.1 and~4.2 introduce finite
boundary states and compare their partition functions with the
generating function for walks from \(O\), thereby proving
Theorem~\ref{thm:critical-identity}\textup{(i)}.  Subsection~4.3 combines
this comparison with the finite convolution bound from Section~3 to
prove Theorem~\ref{thm:critical-identity}\textup{(ii)}.

For \(m\ge0\), identify the boundary of \(F_m\) with \(V_0\) by
\[
  \iota_m:V_0\longrightarrow\partial F_m,
  \qquad
  \iota_m(q)=\Psi_0^{-m}(q).
\]

\begin{definition}[Boundary states]
\label{def:boundary-state}
A \emph{linear forest} is an acyclic graph whose vertices have degree at
most two.  A \emph{boundary forest} in \(F_m\) is a nonempty edge
subgraph \(\Gamma\subset F_m\) such that
\begin{enumerate}[label=\textup{(\roman*)}]
\item \(\Gamma\) is a linear forest;
\item every vertex of degree one belongs to \(\partial F_m\);
\item every connected component contains at least one edge.
\end{enumerate}
Thus every component is a self-avoiding path with two distinct endpoints
in \(\partial F_m\).

An \emph{abstract boundary forest} is a nonempty linear forest whose
vertex set is contained in \(V_0\) and which has no isolated vertices;
equivalently, each connected component is a path containing at least one
edge.  Let \(\cS\) be the finite set of all such abstract boundary
forests.  Thus an element \(\sigma\in\cS\) may consist of one path or of
several vertex-disjoint paths; it records only a possible finite
combination of boundary connections, not an actual walk in \(F_m\).

To pass from the actual boundary forest \(\Gamma\subset F_m\) to its
abstract boundary pattern, consider each component of \(\Gamma\) in
\(F_m\).  List the boundary vertices in the order in which the path meets
them.  Keep only these listed vertices, replace each portion
between two consecutive listed vertices by one abstract edge, and relabel the
retained vertices by \(\iota_m^{-1}\).  This compression operation is
called taking the \emph{boundary trace}.  The resulting nonempty linear
forest on \(V_0\), hence an element of \(\cS\), is the \emph{boundary
state} of \(\Gamma\), denoted by \(\operatorname{tr}_m(\Gamma)\).  It
records which boundary vertices are visited and how consecutive visits
are connected along each path; it does not record the internal vertices
or the lengths of the individual portions.

For \(\sigma\in\cS\), put
\[
  \cC_m(\sigma)
  =
  \{\Gamma\subset F_m:
    \Gamma\text{ is a boundary forest and }
    \operatorname{tr}_m(\Gamma)=\sigma\},
\]
and, for \(\beta\ge0\), define
\begin{equation}
  Z_m^\sigma(\beta)
  =
  \sum_{\Gamma\in\cC_m(\sigma)}
  e^{-\beta\#E(\Gamma)},
  \qquad
  \bZ_m(\beta)
  =
  \bigl(Z_m^\sigma(\beta)\bigr)_{\sigma\in\cS}.
  \label{eq:boundary-partition}
\end{equation}
We call \(\bZ_m(\beta)\) the \emph{boundary-state partition vector}.
We adjoin the \emph{empty state} \(\varnothing\) and set
\(Z_m^\varnothing(\beta)=1\).
\end{definition}

\begin{definition}[Critical inverse temperature]
\label{def:critical-point}
Let
\begin{equation}
  \cB
  =
  \left\{
    \beta\ge0:
    \sup_{m\ge0}\|\bZ_m(\beta)\|_\infty<\infty
  \right\},
  \qquad
  \beta_{\mathrm c}=\inf\cB,
  \label{eq:critical-point}
\end{equation}
where the infimum is initially understood in the extended half-line
\([0,\infty]\).  The number \(\beta_{\mathrm c}\) is the
\emph{critical inverse temperature} of the boundary-state
renormalization.
\end{definition}

\subsection{\texorpdfstring{Boundary-state
recursions}{Boundary-state recursions}}

The boundary-state construction in Definition~\ref{def:boundary-state}
retains every possible connection pattern at the boundary of a cell.
This is necessary when \(\#V_0\ge4\), because the intersection of one
global self-avoiding walk with a cell may have more than one path
component.

The order retained in Definition~\ref{def:boundary-state} matters.  For
example, a component whose boundary trace is
\(q_0-q_1-q_2\) has degree two at \(q_1\), whereas a component with trace
\(q_0-q_2\) does not use \(q_1\).

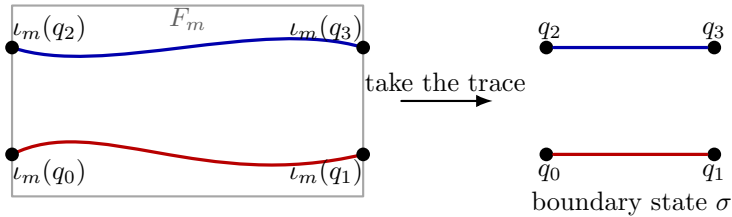
\begin{figure}[ht]
\centering
\begin{overpic}[width=.62\textwidth]{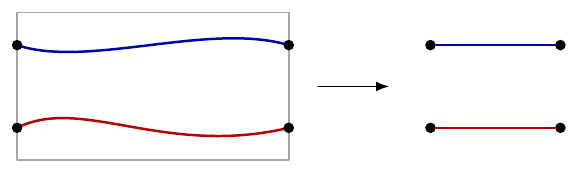}
  \put(26.5,26){\makebox(0,0){\small\color{gray!75!black}\(F_m\)}}
  \put(3,5.4){\makebox(0,0)[l]{\small\(\iota_m(q_0)\)}}
  \put(50,5.4){\makebox(0,0)[r]{\small\(\iota_m(q_1)\)}}
  \put(3,24.7){\makebox(0,0)[l]{\small\(\iota_m(q_2)\)}}
  \put(50,24.7){\makebox(0,0)[r]{\small\(\iota_m(q_3)\)}}
  \put(61,17.5){\makebox(0,0){\small take the trace}}
  \put(75,5.3){\makebox(0,0){\small \(q_0\)}}
  \put(97,5.3){\makebox(0,0){\small \(q_1\)}}
  \put(75,24.4){\makebox(0,0){\small \(q_2\)}}
  \put(97,24.4){\makebox(0,0){\small \(q_3\)}}
  \put(86,1.3){\makebox(0,0){\small boundary state \(\sigma\)}}
\end{overpic}
\vspace{3pt}
\caption{A boundary forest and its boundary state.  Internal vertices and
the shapes of the two paths are suppressed, while their boundary
connections are retained.}
\label{fig:boundary-state}
\end{figure}
\FloatBarrier

Figure~\ref{fig:boundary-state} records precisely the information kept
by a boundary state.

\begin{proposition}
\label{prop:boundary-renormalization}
There is a polynomial map
\[
  \Phi:\mathbb R_+^{\cS}\longrightarrow\mathbb R_+^{\cS}
\]
with nonnegative integer coefficients such that
\begin{equation}
  \bZ_{m+1}(\beta)=\Phi\bigl(\bZ_m(\beta)\bigr),
  \qquad m\ge0,\quad \beta\ge0.
  \label{eq:boundary-renormalization}
\end{equation}
The map depends only on the incidence pattern of the first-level
fractal cells and the identifications of their boundary vertices.
\end{proposition}

\begin{proof}
Let \(\cS_0=\cS\cup\{\varnothing\}\).  For a tuple
\(\boldsymbol\sigma=(\sigma_0,\ldots,\sigma_{M-1})\in\cS_0^M\), transport
\(\sigma_i\) to the boundary \(\Psi_i(V_0)\) of the first-level fractal
cell \(\Psi_i(K)\), and identify vertices which represent the same point.
This produces an abstract graph whose vertex set is contained in
\[
  \bigcup_{i=0}^{M-1}\Psi_i(V_0).
\]
We call the tuple \emph{admissible} if this graph is acyclic, has maximum degree
at most two, and every vertex of degree one belongs to
\(V_0\), the boundary of the parent cell.  Suppressing
all vertices outside \(V_0\) then produces
a well-defined state in \(\cS\).  Let \(\mathfrak C_\sigma\) be the finite
set of admissible tuples whose resulting state is \(\sigma\).

We claim that
\begin{equation}
  Z_{m+1}^{\sigma}(\beta)
  =
  \sum_{\boldsymbol\sigma\in\mathfrak C_\sigma}
  \prod_{i=0}^{M-1} Z_m^{\sigma_i}(\beta),
  \qquad
  Z_m^\varnothing(\beta)=1.
  \label{eq:state-coordinate-recursion}
\end{equation}
Indeed, Proposition~\ref{prop:cell-hierarchy} assigns every edge of
\(F_{m+1}\) to a unique \(m\)-subcell.  Restricting a boundary forest to
these subcells therefore gives a unique tuple in \(\cS_0^M\).  Any
degree-one vertex created by the restriction lies on the boundary of the
corresponding subcell.  After the portions of the forest inside each
child copy of \(F_m\) have been compressed to their child states, the resulting
abstract union records the original forest, including the junctions at
the junction vertices in the interior of the parent cell that arise from
the boundaries of its \(m\)-subcells.  Suppressing these internal
junctions gives exactly the boundary state of the original forest.  Hence
the tuple is admissible.

Conversely, let
\(\boldsymbol\sigma=(\sigma_0,\ldots,\sigma_{M-1})\in\mathfrak C_\sigma\).
For each \(i\in\{0,\ldots,M-1\}\), choose a forest
\(\Gamma_i\in\cC_m(\sigma_i)\) when \(\sigma_i\ne\varnothing\), and set
\(\Gamma_i=\varnothing\) when \(\sigma_i=\varnothing\).  Distinct child
copies of \(F_m\) have disjoint edge sets and meet only at boundary vertices by
Proposition~\ref{prop:cell-hierarchy}.  More precisely, admissibility
means that \(\bigcup_i\Gamma_i\) is acyclic, every vertex of this union
has degree at most two, and every degree-one vertex belongs to the
boundary \(V_0\) of the parent cell.  Therefore
\(\bigcup_i\Gamma_i\)
is a boundary forest with state \(\sigma\).

These two operations are inverse to each other.  Moreover, the edge sets
of distinct child copies are disjoint.  Hence the number of edges in the
union is the sum of the edge counts in the children, while its weight is
the product of their weights.  Summing over all admissible tuples proves
\eqref{eq:state-coordinate-recursion}.

The first-level incidence pattern is the same at every scale, so the finite sets
\(\mathfrak C_\sigma\) do not depend on \(m\).  Collecting
\eqref{eq:state-coordinate-recursion} over \(\sigma\in\cS\) gives
\eqref{eq:boundary-renormalization}.
\end{proof}

We also require states with one free endpoint.  A \emph{marked boundary
forest} in \(F_m\) is a pair \((\Gamma,\xi)\), where \(\Gamma\) satisfies
conditions (i) and (iii) of Definition~\ref{def:boundary-state},
\(\xi\) is a distinguished degree-one vertex, and every degree-one vertex
other than \(\xi\) belongs to \(\partial F_m\).  The point \(\xi\) is
allowed to lie on \(\partial F_m\).  We introduce a new formal symbol
\(\star\), which is not a vertex of \(F_m\) or \(V_0\), to represent a
marked endpoint lying in the interior of \(F_m\).  The \emph{marked
boundary trace} \(\operatorname{tr}_m^\bullet(\Gamma,\xi)\) is defined
by applying the same compression procedure to \((\Gamma,\xi)\), with
the following modification: if \(\xi\in\partial F_m\),
the corresponding vertex \(\iota_m^{-1}(\xi)\in V_0\) is marked; if
\(\xi\notin\partial F_m\), the symbol \(\star\) is marked instead.  In
both cases the trace records which component contains the marked point.
Let \(\cS^\bullet\) be the resulting finite set of marked states, and set
\[
  A_m^{\sigma^\bullet}(\beta)
  =
  \sum_{\substack{(\Gamma,\xi)\text{ marked}\\
                  \operatorname{tr}_m^\bullet(\Gamma,\xi)=\sigma^\bullet}}
  e^{-\beta\#E(\Gamma)},
  \qquad
  \bA_m(\beta)
  =
  \bigl(A_m^{\sigma^\bullet}(\beta)\bigr)_
        {\sigma^\bullet\in\cS^\bullet}.
\]

\begin{proposition}
\label{prop:marked-renormalization}
There is a matrix \(\mathcal L(\mathbf x)\), whose entries are polynomials
in \(\mathbf x=(x_\sigma)_{\sigma\in\cS}\) with nonnegative integer
coefficients, such that
\begin{equation}
  \bA_{m+1}(\beta)
  =
  \mathcal L\bigl(\bZ_m(\beta)\bigr)\bA_m(\beta).
  \label{eq:marked-renormalization}
\end{equation}
Consequently, if
\(\sup_m\|\bZ_m(\beta)\|_\infty<\infty\), then there is
\(C_\beta<\infty\) such that
\begin{equation}
  \|\bA_m(\beta)\|_\infty\le C_\beta^{\,m+1},
  \qquad m\ge0.
  \label{eq:marked-growth}
\end{equation}
\end{proposition}

\begin{proof}
The marked vertex \(\xi\) has degree one in \(\Gamma\), so exactly one
edge of \(\Gamma\) is incident with \(\xi\).  By
Proposition~\ref{prop:cell-hierarchy}, this edge belongs to exactly one
child copy of \(F_m\).  We put the mark on that child restriction.  Every
other nonempty child restriction is an unmarked boundary forest, even if
its copy shares the boundary vertex \(\xi\).

For a fixed marked state of the parent cell, whether a given collection of
child states can be joined to form that parent state is determined solely by
the child states, as in the proof of Proposition~\ref{prop:boundary-renormalization}.
Therefore each parent marked partition function is a finite sum of terms
consisting of exactly one marked factor and at most \(M-1\) unmarked
factors.  This gives the exact linear recursion
\eqref{eq:marked-renormalization}.

If \(\|\bZ_m(\beta)\|_\infty\le K_\beta\) for all \(m\), the finitely many entries
of \(\mathcal L\) are bounded on the cube
\([0,K_\beta]^{\cS}\).  Hence
\[
  \|\bA_{m+1}(\beta)\|_\infty
  \le C_\beta\|\bA_m(\beta)\|_\infty
\]
for some \(C_\beta<\infty\).  Iteration,
with one further increase of \(C_\beta\) to include
\(\|\bA_0(\beta)\|_\infty\), proves
\eqref{eq:marked-growth}.
\end{proof}

Since all terms in \eqref{eq:boundary-partition} are nonnegative,
\(\bZ_m(\beta')\le\bZ_m(\beta)\) coordinatewise whenever
\(\beta'\ge\beta\).  Hence \(\beta\in\cB\) and \(\beta'\ge\beta\)
imply \(\beta'\in\cB\).  Proposition~\ref{prop:boundary-renormalization}
gives the fixed finite-dimensional polynomial recursion for the boundary-state
partition functions, while Proposition~\ref{prop:marked-renormalization}
gives the corresponding linear recursion and growth estimate for marked
partition functions.  We use these two recursions below to analyze bounded
orbits and the critical inverse temperature in
Definition~\ref{def:critical-point}.

\subsection{\texorpdfstring{Comparison with rooted
walks}{Comparison with rooted walks}}

\begin{lemma}
\label{lem:logarithmic-localization}
There are constants \(c_0,c_1<\infty\), depending only on the graph, such
that every self-avoiding walk of length \(n\) starting at \(O\) is
contained in \(F_{m(n)}\), where
\[
  m(n)\le c_0+c_1\log(n+1).
\]
\end{lemma}

\begin{proof}
By the definition of \(\delta_0\), every edge has length \(\delta_0\),
and distinct points of \(V_0\) are at least \(\delta_0\) apart.
The degree of \(O\) is finite and \(F_\infty=\bigcup_mF_m\), so there is
\(m_0\) such that every edge incident with \(O\) belongs to \(F_{m_0}\).
For \(n\ge1\), choose
\[
  m(n)=
  \max\left\{
    m_0,\,
    \left\lceil\log_{\lambda}(n+1)\right\rceil
  \right\}.
\]

Suppose that a self-avoiding walk \(w\) from \(O\) has an edge outside
\(F_{m(n)}\), and consider its first such edge.  Its initial vertex is a
point at which \(F_{m(n)}\) meets a different \(m(n)\)-cell.  By
Proposition~\ref{prop:cell-hierarchy}, this point belongs to
\(\partial F_{m(n)}\).  It cannot be \(O\): an exit through \(O\) on the
first step is excluded by the choice of \(m_0\), and a later exit through
\(O\) would revisit the initial vertex.  The walk must therefore reach
\(\iota_{m(n)}(q)\) for some \(q\in V_0\setminus\{O\}\) before it exits.
The length of the part already traversed is at least
\[
  \frac{|\iota_{m(n)}(q)-O|}{\delta_0}
  =
  \frac{\lambda^{m(n)}|q-O|}{\delta_0}
  \ge
  \lambda^{m(n)}
  \ge
  n+1
  >
  n,
\]
which is impossible.  Thus \(w\subset F_{m(n)}\).  The asserted
logarithmic bound follows immediately from the definition of \(m(n)\).
\end{proof}

For \(\beta\ge0\), let
\begin{equation}
  \chi(\beta)=\sum_{n=0}^{\infty}c_ne^{-\beta n},
  \label{eq:susceptibility}
\end{equation}
be the rooted self-avoiding walk \emph{susceptibility}, or equivalently its
length-generating function.

\begin{proposition}
\label{prop:critical-comparison}
For every \(\beta\ge0\), the following implications hold.
\begin{enumerate}[label=\textup{(\roman*)}]
\item If \(\beta\in\cB\), then
  \(\chi(\beta')<\infty\) for every \(\beta'>\beta\).
\item If \(\chi(\beta)<\infty\), then \(\beta\in\cB\).
\end{enumerate}
\end{proposition}

\begin{proof}
First suppose that \(\beta\in\cB\).  By
Proposition~\ref{prop:marked-renormalization}, there is
\(C_\beta\ge1\) such that \eqref{eq:marked-growth} holds.  Fix
\(n\ge1\), and use Lemma~\ref{lem:logarithmic-localization} to choose
\(m=m(n)\).  Every \(w\in\cW_{O,n}\), marked at its terminal vertex, is
a marked boundary forest in \(F_m\), because its other endpoint is
\(O\in\partial F_m\).  Consequently,
\[
  c_ne^{-\beta n}
  \le
  \sum_{\sigma^\bullet\in\cS^\bullet}
    A_m^{\sigma^\bullet}(\beta)
  \le
  \#\cS^\bullet\,C_\beta^{\,m+1}
  \le
  C(n+1)^{\kappa_\beta}
\]
for constants \(C,\kappa_\beta<\infty\) which may depend on \(\beta\), but not
on \(n\).  If \(\beta'>\beta\), it follows that
\[
  c_ne^{-\beta'n}
  \le
  C(n+1)^{\kappa_\beta} e^{-(\beta'-\beta)n}.
\]
The series on the right is summable, and hence
\(\chi(\beta')<\infty\).  This proves (i).

For (ii), suppose that \(\chi(\beta)<\infty\), and fix
\(\sigma\in\cS\).  Let \(k\) be the number of components of
\(\sigma\).  The state \(\sigma\) specifies the two endpoints of each
component and determines which boundary vertices belong to the same
component.  Order the components using the fixed order
\(q_0,\ldots,q_{b-1}\) of \(V_0\), and orient each one from the endpoint
with the smaller index to the other endpoint.  For
\(\Gamma\in\cC_m(\sigma)\), write its components as
\(\Gamma^{(1)},\ldots,\Gamma^{(k)}\).  The initial vertex of every
oriented component belongs to \(\partial F_m\).  Applying
Lemma~\ref{lem:boundary-rooting} to all components produces an ordered
\(k\)-tuple of nonempty self-avoiding walks starting at \(O\), with the
same respective lengths.

The map from \(\Gamma\) to this tuple is injective.  Indeed, the initial
boundary vertex of each component is fixed by \(\sigma\), so the inverse
of the corresponding automorphism in
Lemma~\ref{lem:boundary-rooting} recovers that component.  If
\(\ell_i=\#E(\Gamma^{(i)})\), then
\[
  \#E(\Gamma)=\ell_1+\cdots+\ell_k.
\]
Using \eqref{eq:boundary-partition} and dropping the requirements that
the recovered components be mutually disjoint and have state
\(\sigma\), we obtain
\[
  Z_m^\sigma(\beta)
  \le
  \sum_{\ell_1,\ldots,\ell_k\ge1}
  \prod_{i=1}^k c_{\ell_i}e^{-\beta\ell_i}
  =
  \bigl(\chi(\beta)-1\bigr)^k.
\]
The number \(k\) is bounded by \(b\), and there are only finitely many
states.  Hence
\(\sup_m\|\bZ_m(\beta)\|_\infty<\infty\), so
\(\beta\in\cB\), as required.
\end{proof}

\begin{proof}[Proof of Theorem~\ref{thm:critical-identity}\textup{(i)}]
By \eqref{eq:elementary-walk-bound}, \(\chi(\beta)<\infty\) whenever
\(\beta>\log(D-1)\).  Proposition~\ref{prop:critical-comparison}(ii)
therefore shows that \(\cB\ne\varnothing\), so
\(\beta_{\mathrm c}<\infty\).

Define the susceptibility threshold
\[
  \beta_\chi
  =
  \inf\{\beta\ge0:\chi(\beta)<\infty\}.
\]
If \(\beta\in\cB\), Proposition~\ref{prop:critical-comparison}(i) gives
\(\chi(\beta')<\infty\) for every \(\beta'>\beta\).  Hence
\(\beta_\chi\le\beta\), and taking the infimum over
\(\beta\in\cB\) yields
\[
  \beta_\chi\le\beta_{\mathrm c}.
\]
Conversely, if \(\chi(\beta)<\infty\), then
Proposition~\ref{prop:critical-comparison}(ii) gives
\(\beta\in\cB\).  Taking the infimum over such \(\beta\) yields
\[
  \beta_{\mathrm c}\le\beta_\chi.
\]
Thus \(\beta_{\mathrm c}=\beta_\chi\).

Finally, Theorem~\ref{thm:connective} gives
\[
  \lim_{n\to\infty}
  \bigl(c_ne^{-\beta n}\bigr)^{1/n}
  =\mu e^{-\beta}.
\]
The root test applied to \eqref{eq:susceptibility} shows that
\(\chi(\beta)\) converges for \(\beta>\log\mu\) and diverges for
\(\beta<\log\mu\).  Therefore \(\beta_\chi=\log\mu\), and hence
\(\beta_{\mathrm c}=\log\mu\).  Equivalently,
\(\mu=e^{\beta_{\mathrm c}}\), as claimed.
\end{proof}

\subsection{\texorpdfstring{Polynomial
bounds}{Polynomial bounds}}

We now prove Theorem~\ref{thm:critical-identity}\textup{(ii)}, assuming
that the infimum in \eqref{eq:critical-point} is attained.  The upper-bound argument is the
boundary-state form of the scale-localization method used by Hattori and
Kusuoka on the Sierpi\'nski gasket~\cite{HattoriKusuoka1992}.  Their
lower bound used a detailed analysis of the critical two-variable
recursion; the lower bound below instead follows from the finite
convolution estimate \eqref{eq:uniform-convolution}.

\begin{proof}[Proof of Theorem~\ref{thm:critical-identity}\textup{(ii)}]
Assume that \(\beta_{\mathrm c}=\min\cB\).  By
Theorem~\ref{thm:critical-identity}(i),
\(\mu=e^{\beta_{\mathrm c}}\).

\emph{Upper bound.}
Since \(\beta_{\mathrm c}\in\cB\),
Proposition~\ref{prop:marked-renormalization} applies at
\(\beta=\beta_{\mathrm c}\).  Together with
Lemma~\ref{lem:logarithmic-localization}, it gives constants \(K\ge1\)
and \(a\ge0\) such that, for \(n\ge1\),
\[
  c_ne^{-\beta_{\mathrm c}n}
  \le
  \#\cS^\bullet\,
  \|\bA_{m(n)}(\beta_{\mathrm c})\|_\infty
  \le
  K(n+1)^a.
\]
Since \(n+1\le2n\) for \(n\ge1\), increasing \(K\) gives
\[
  c_n\le K\mu^n n^a,
  \qquad n\ge1.
\]

\emph{Lower bound.}
Recall the constants \(C,Q\) from
Proposition~\ref{prop:bounded-piece} and the quantity \(G_n\) from
\eqref{eq:weight-bound}.  Put
\[
  s=Q-1+aQ,
  \qquad
  B=QC^QK^Q.
\]
The upper bound just proved, together with
\eqref{eq:uniform-convolution} and \eqref{eq:weight-bound}, gives
\[
  u_j
  \le
  G_jK^Q\mu^j j^{aQ}
  \le
  B\mu^j j^s,
  \qquad j\ge1.
\]

Fix \(t\ge1\) and set \(n=Qt\).  Since
\(\mu=\inf_{j\ge1}u_j^{1/j}\) by
Theorem~\ref{thm:connective}, we have \(u_n\ge\mu^n\).
By \eqref{eq:uniform-convolution} and
\eqref{eq:weight-bound}, there are \(q\le Q\) and positive integers
\(\ell_1,\ldots,\ell_q\), with
\(\ell_1+\cdots+\ell_q=n\), such that
\[
  \prod_{i=1}^q c_{\ell_i}
  \ge
  \frac{u_n}{G_n}
  \ge
  \frac{\mu^n}{QC^Qn^{Q-1}}.
\]
Let \(k=\max_i\ell_i\).  Then
\(k\ge\lceil n/Q\rceil=t\).  Applying the upper
bound to all factors except one of length \(k\) yields

\[
  \prod_{i\ne i_*}c_{\ell_i}
  \le
  K^{q-1}\mu^{n-k}n^{a(q-1)},
  \qquad \ell_{i_*}=k.
\]
Since \(q\le Q\), the definitions of \(B\) and \(s\) therefore give
\[
  c_k
  \ge
  B^{-1}\mu^k n^{-s}.
\]
On the other hand, cutting a rooted \(k\)-step walk after \(t\) edges
gives
\[
  c_k
  \le
  c_tu_{k-t}
  \le
  Bc_t\mu^{k-t}(k-t+1)^s.
\]
Here the last inequality follows from the preceding bound when \(k>t\),
and from \(u_0=1\) when \(k=t\), after increasing \(B\) if necessary.
Since \(k\le n=Qt\), both \(n\) and \(k-t+1\) are at most \(Qt\).
Combining the last two estimates gives
\[
  c_t
  \ge
  B^{-2}Q^{-2s}\mu^t t^{-2s}.
\]
Taking \(p=2s\) and increasing a single constant \(A\ge1\) to cover both
bounds proves the assertion.
\end{proof}

We conclude this section with a useful criterion for the uniform
boundedness of \(\bZ_m(\beta)\).

\begin{remark}
\label{rem:one-edge-states}
For distinct \(i,j\in\{0,\ldots,b-1\}\), let \(e_{ij}\in\cS\) be the
abstract boundary forest consisting of the single edge
\(\{q_i,q_j\}\), and put
\[
  P_m^{ij}(\beta)=Z_m^{e_{ij}}(\beta).
\]
Then
\[
  \sup_{m\ge0}\|\bZ_m(\beta)\|_\infty<\infty
  \quad\Longleftrightarrow\quad
  \sup_{m\ge0}\max_{i\ne j}P_m^{ij}(\beta)<\infty.
\]

Indeed, put \(P_m(\beta)=\max_{i\ne j}P_m^{ij}(\beta)\).  Fix
\(\sigma\in\cS\) and \(\Gamma\in\cC_m(\sigma)\).  Cutting each
component of \(\Gamma\) at its successive visits to \(\partial F_m\)
produces, for every abstract edge \(\{q_i,q_j\}\) of \(\sigma\), a self-avoiding
path in \(\cC_m(e_{ij})\).  These paths partition the edge set of
\(\Gamma\).  The resulting tuple determines \(\Gamma\), and its weights
multiply.  Hence
\[
  Z_m^\sigma(\beta)
  \le
  \prod_{\{q_i,q_j\}\text{ an edge of }\sigma}
  P_m^{ij}(\beta).
\]
Since \(\sigma\) is a linear forest on at most \(b\) vertices, it has at
most \(b-1\) edges.  Consequently,
\[
  P_m(\beta)
  \le
  \|\bZ_m(\beta)\|_\infty
  \le
  \max\{1,P_m(\beta)^{\,b-1}\}.
\]
\end{remark}

\section{Polygonal \texorpdfstring{\(N\)}{N}-gaskets}

This section proves Theorem~\ref{thm:ngasket}.  We first determine the
contact geometry of the first-level fractal cells and the exact crossing
recursions.  We then compare the crossing critical inverse temperature
with the critical inverse temperature of the boundary-state
renormalization and compute the connective constants for \(N=6\) and
\(N=9\).  The complete dynamics of the recursions is given in
Appendix~\ref{app:ngasket-dynamics}.

\begin{definition}[Regular polygonal \(N\)-gaskets]
\label{def:regular-ngasket}
Let \(N\ge3\) with \(4\nmid N\).  Identify \(\R^2\) with \(\mathbb C\), and put
\[
  q_j=\exp(2\pi\mathrm i j/N),
  \qquad
  V_0^{(N)}=\{q_j:j\in\mathbb Z/N\mathbb Z\}.
\]
Let \(F_0^{(N)}\) be the cycle on \(V_0^{(N)}\).  Write
\[
  k_N=\lfloor N/4\rfloor,
  \qquad
  \ell_N=N-4k_N,
  \qquad
  s_N=k_N+1,
\]
and define
\[
  \rho_N
  =
  \frac{\sin(\pi/N)}
       {\sin(\pi/N)+\sin((2s_N-1)\pi/N)},
  \qquad
  \Psi_j^{(N)}(x)=q_j+\rho_N(x-q_j).
\]
The attractor of the \(N\) homotheties
\(\{\Psi_j^{(N)}:j\in\mathbb Z/N\mathbb Z\}\) is the
\emph{regular polygonal \(N\)-gasket}.  Its finite graphs and its one-sided graph
are denoted by \(F_m^{(N)}\) and \(F_\infty^{(N)}\).  In the construction
\eqref{eq:finite-graphs}, the
indices are relabeled by \(\mathbb Z/N\mathbb Z\), the root is
\(O=q_0\), and \(\Psi_0^{(N)}\) is the map fixing \(O\).  The number
\(\rho_N\) is the \emph{first-contact
ratio}~\cite[Section~7]{TysonWu2006}.
When \(4\mid N\), two consecutive first-level fractal cells in the
first-contact construction share a non-degenerate line segment.  Since
\(V_0^{(N)}\) is finite, this intersection cannot equal the intersection
of their boundary vertex sets.  Thus \textup{(NF4)} fails, so this
construction is not a nested fractal and is excluded from the definition.
\end{definition}

\begin{figure}[ht]
\centering
\begin{minipage}{.62\textwidth}
\centering
\begin{overpic}[width=\linewidth]{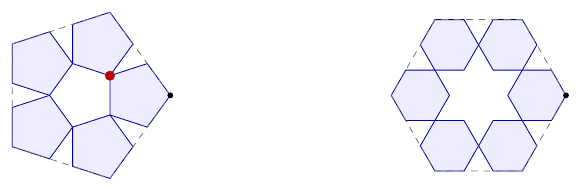}
  \put(19,22.1){\makebox(0,0)[b]{\scriptsize\(\zeta_0\)}}
  \put(30.2,16.5){\makebox(0,0)[l]{\scriptsize\(q_0\)}}
  \put(98.2,16.5){\makebox(0,0)[l]{\scriptsize\(q_0\)}}
\end{overpic}
\par\vspace{2pt}
\makebox[\linewidth]{%
  \makebox[.50\linewidth]{\small (a) the \(5\)-gasket}%
  \makebox[.50\linewidth]{\small (b) the \(6\)-gasket}}
\end{minipage}
\caption{First-level polygonal hulls for two polygonal \(N\)-gaskets.
The dashed polygon has vertex set \(V_0^{(N)}\), and each blue polygon
is the hull of one first-level fractal cell.  In (a), \(\zeta_0\) marks
the contact point between two adjacent first-level fractal cells.}
\label{fig:ngasket-first-level}
\end{figure}
\FloatBarrier

Figure~\ref{fig:ngasket-first-level} illustrates the first-contact
construction.  For a fixed polygonal \(N\)-gasket, put
\[
  a_{j,m}=(\Psi_0^{(N)})^{-m}(q_j),
  \qquad j\in\mathbb Z/N\mathbb Z,\quad m\ge0.
\]
A \emph{crossing} of \(F_m^{(N)}\) is a self-avoiding walk between two
prescribed distinct vertices of \(\partial F_m^{(N)}\).

If \(N\notin\{3,5\}\), define
\begin{equation}
  \begin{aligned}
  X_m^{(N)}(\beta)
    &=
    \sum_{\substack{w:a_{0,m}\to a_{s_N,m}\\
                    w\text{ self-avoiding in }F_m^{(N)}}}
      e^{-\beta L(w)},\\
  Y_m^{(N)}(\beta)
    &=
    \sum_{\substack{w:a_{-s_N,m}\to a_{s_N,m}\\
                    w\text{ self-avoiding in }F_m^{(N)}}}
      e^{-\beta L(w)}.
  \end{aligned}
  \label{eq:ngasket-crossings}
\end{equation}
Thus \(X_m^{(N)}\) joins \(a_{0,m}\) to \(a_{s_N,m}\), whereas
\(Y_m^{(N)}\) joins \(a_{-s_N,m}\) to \(a_{s_N,m}\).

For \(N=3\), let \(x_m(\beta)\) and \(y_m(\beta)\) be the partition
functions of crossings between two fixed boundary vertices which,
respectively, avoid and visit the third boundary vertex.  For \(N=5\),
let \(\phi_{1,m}(\beta)\) and \(\phi_{2,m}(\beta)\) count crossings from
\(a_{0,m}\) to \(a_{2,m}\) which, respectively, avoid and visit
\(a_{-2,m}\); let \(\theta_{1,m}(\beta)\) and
\(\theta_{2,m}(\beta)\) count crossings from \(a_{-2,m}\) to
\(a_{2,m}\) which, respectively, avoid and visit \(a_{0,m}\).
In each case the weight of a walk is \(e^{-\beta L(w)}\).  Define the
\emph{crossing vector}
\[
  \mathbf C_m^{(N)}(\beta)
  =
  \begin{cases}
    (x_m(\beta),y_m(\beta)),&N=3,\\
    (\phi_{1,m}(\beta),\phi_{2,m}(\beta),
     \theta_{1,m}(\beta),\theta_{2,m}(\beta)),&N=5,\\
    (X_m^{(N)}(\beta),Y_m^{(N)}(\beta)),&N\notin\{3,5\},
  \end{cases}
\]
and its \emph{crossing critical inverse temperature} by
\begin{equation}
  \beta_{\mathrm c,N}^{\mathrm{cr}}
  =
  \inf\left\{
    \beta\ge0:
    \sup_{m\ge0}\|\mathbf C_m^{(N)}(\beta)\|_\infty<\infty
  \right\}.
  \label{eq:ngasket-critical-beta}
\end{equation}
The infimum is understood in \([0,\infty]\).
Let \(\mu_N\) and \(\beta_{\mathrm c,N}\) denote, respectively, the
connective constant and the critical inverse temperature of
\(F_\infty^{(N)}\), with the latter defined by
Definition~\ref{def:critical-point}.

Throughout this section, fix \(N\ge3\) with \(4\nmid N\).  We use the
polygonal \(N\)-gasket from Definition~\ref{def:regular-ngasket} and
omit the superscript \((N)\) from \(F_m^{(N)}\) and \(F_\infty^{(N)}\),
and also from \(X_m^{(N)}\) and \(Y_m^{(N)}\).

\subsection{\texorpdfstring{Geometry and crossing
recursions}{Geometry and crossing recursions}}

The following proposition describes the contact pattern that determines the
crossing recursions.

For a first-level fractal cell \(\Psi_j^{(N)}(K)\), the fixed point
\(q_j\) of \(\Psi_j^{(N)}\) is called its \emph{outer vertex}.  More
generally, when a parent fractal cell is decomposed into its immediate
fractal subcells,
the corresponding image of \(q_j\) is called the outer vertex of the
\(j\)-th subcell.

\begin{proposition}
\label{prop:ngasket-geometry}
The following statements hold.
\begin{enumerate}[label=\textup{(\roman*)}]
\item The polygonal \(N\)-gasket satisfies
  \textup{(NF1)}--\textup{(NF4)} and the one-point intersection condition
  \eqref{eq:one-point}.
\item The incidence graph of its first-level fractal cells is the cycle
  graph \(C_N\) on \(\mathbb Z/N\mathbb Z\), with an edge between \(j\) and
  \(j+1\) for each \(j\) modulo \(N\).
\item The outer vertex of a first-level fractal cell is its unique vertex in
  \(V_0^{(N)}\).  Fix such a fractal cell and identify its boundary with
  \(V_0^{(N)}\) by applying the inverse cell map followed by a dihedral
  relabeling that sends its outer vertex to \(q_0\).  Under this
  identification, its contact points with the previous and next
  first-level fractal cells in the cyclic order correspond to
  \(q_{-s_N}\) and \(q_{s_N}\), respectively.
\end{enumerate}
\end{proposition}

\begin{proof}
We first prove \textup{(ii)} and \textup{(iii)}.  The standard
first-contact calculation for two consecutive homothetic
regular \(N\)-gons gives
\[
  \Psi_0^{(N)}(q_{s_N})
  =
  \Psi_1^{(N)}(q_{1-s_N}).
\]
For \(4\nmid N\), this is the unique contact of the two fractal cells.
Rotating the identity by one step in the opposite direction also gives
\[
  \Psi_{-1}^{(N)}(q_{s_N-1})
  =
  \Psi_0^{(N)}(q_{-s_N}).
\]
Thus, under \((\Psi_0^{(N)})^{-1}\), the contact points of the
first-level fractal cell \(\Psi_0^{(N)}(K)\) with
\(\Psi_{-1}^{(N)}(K)\) and \(\Psi_1^{(N)}(K)\) correspond to
\(q_{-s_N}\) and \(q_{s_N}\), respectively.  A dihedral symmetry gives
the same normalized description for every first-level fractal cell.

Rotating the configuration gives one contact for every consecutive pair,
while nonconsecutive polygonal hulls are disjoint at the first-contact
ratio~\cite[Section~7]{TysonWu2006}.  Hence the incidence graph is
\(C_N\).  Moreover, each first-level polygonal hull contains only its
fixed point from \(V_0^{(N)}\), so this fixed point is the unique outer
vertex of the corresponding fractal cell.  This proves \textup{(ii)} and
\textup{(iii)}, as well as the first-level form of \textup{(NF4)} and
the one-point intersection condition.

It remains to prove \textup{(i)}.  The interiors of the first-level
polygonal hulls give \textup{(NF1)}, their cyclic contact pattern gives
\textup{(NF2)}, and the dihedral symmetries of the regular polygon give
\textup{(NF3)}.  The preceding contact description gives
\textup{(NF4)} at the first level, and iteration gives the one-point
intersection condition at every level.
\end{proof}

The symbols \(\cS\), \(\cC_m(\sigma)\), and \(\bZ_m\) refer to the
boundary states, boundary forests, and boundary-state partition vector
of this polygonal \(N\)-gasket.
The resulting recursions are as follows.  Every quantity in
the proposition is evaluated at the same fixed \(\beta\ge0\).

\begin{proposition}
\label{prop:ngasket-recursions}
If \(N=4k_N+\ell_N\notin\{3,5\}\), where
\(\ell_N\in\{1,2,3\}\), then the functions in
\eqref{eq:ngasket-crossings} satisfy
\begin{equation}
  \begin{aligned}
  X_0
    &=e^{-(k_N+1)\beta}
      +e^{-(3k_N+\ell_N-1)\beta},&
  Y_0
    &=e^{-2(k_N+1)\beta}
      +e^{-(2k_N+\ell_N-2)\beta},\\
  X_{m+1}
    &=X_m^2
      \bigl(Y_m^{k_N}+Y_m^{3k_N+\ell_N-2}\bigr),&
  Y_{m+1}
    &=X_m^2
      \bigl(Y_m^{2k_N+\ell_N-3}+Y_m^{2k_N+1}\bigr).
  \end{aligned}
  \label{eq:ngasket-recursion}
\end{equation}

For \(N=3\),
\begin{equation}
  \begin{aligned}
  x_0&=e^{-\beta},&
  y_0&=e^{-2\beta},\\
  x_{m+1}
    &=(x_m+y_m)^2+x_m^2(x_m+2y_m),&
  y_{m+1}
    &=x_my_m(x_m+2y_m).
  \end{aligned}
  \label{eq:sg-crossing-recursion}
\end{equation}

For \(N=5\), put
\[
  \bar\phi_m=\phi_{1,m}+\phi_{2,m},
  \qquad
  \bar\theta_m=\theta_{1,m}+\theta_{2,m}.
\]
Then
\begin{equation}
  \begin{aligned}
  (\phi_{1,0},\phi_{2,0},\theta_{1,0},\theta_{2,0})
    &=(e^{-2\beta},e^{-3\beta},e^{-\beta},e^{-4\beta}),\\
  \phi_{1,m+1}
    &=\bar\phi_m^2\bar\theta_m(1+\theta_{1,m}),\\
  \phi_{2,m+1}
    &=\bar\phi_m^2\bar\theta_m\theta_{2,m},\\
  \theta_{1,m+1}
    &=\bar\phi_m^2+
      \bar\theta_m^2\theta_{1,m}\phi_{1,m}
      (\phi_{1,m}+2\phi_{2,m}),\\
  \theta_{2,m+1}
    &=\bar\theta_m^2\theta_{2,m}\phi_{1,m}
      (\phi_{1,m}+2\phi_{2,m}).
  \end{aligned}
  \label{eq:pentagasket-recursion}
\end{equation}
\end{proposition}

\begin{proof}
At level zero, \(F_0=C_N\), so the two paths between any prescribed
pair of vertices are the two polygonal arcs.  Their lengths give all
the initial conditions in
\eqref{eq:ngasket-recursion}--\eqref{eq:pentagasket-recursion}.

Inside an \((m+1)\)-cell, label the \(m\)-subcells
\(\Delta_0,\ldots,\Delta_{N-1}\) so that \(\Delta_j\) contains the
outer vertex \(a_{j,m+1}\), and put
\[
  \zeta_j=\Delta_j\cap\Delta_{j+1}.
\]
By Proposition~\ref{prop:ngasket-geometry}(i)--(ii) and self-similarity,
\(\zeta_j\) is the unique contact between two consecutive subcells, and
each subcell meets the rest of the parent only at its two contacts
\(\zeta_{j-1}\) and \(\zeta_j\).
By
Proposition~\ref{prop:cell-hierarchy}, each edge belongs to a unique
\(m\)-subcell.
For a crossing, its \emph{subcell itinerary} is the sequence of
\(m\)-subcells containing its successive edges, with consecutive
repetitions deleted.
Once a crossing enters and leaves an internal subcell, both contacts
have been used, so self-avoidance prevents a later return.  An endpoint
subcell cannot be re-entered either, since the other endpoint lies in a
different subcell.  Hence the subcell itinerary is one of the two simple
arcs of \(C_N\).  Dihedral symmetry identifies every local orientation of
a piece with the corresponding partition function.

We first prove \eqref{eq:ngasket-recursion}.  An \(X\)-crossing has two
endpoint pieces of type \(X_m\).  Its short arc has
\(s_N-1=k_N\) interior pieces of type \(Y_m\), and its long arc has
\[
  N-s_N-1=3k_N+\ell_N-2
\]
such pieces.  This gives the recursion for \(X_{m+1}\).
For a \(Y\)-crossing, the two arcs have
\[
  2s_N-1=2k_N+1
  \quad\text{and}\quad
  N-2s_N-1=2k_N+\ell_N-3
\]
interior cells.  This gives the recursion for \(Y_{m+1}\).
Since \(s_N=\lfloor N/4\rfloor+1\), the endpoint cells of an
\(X\)-crossing are adjacent only when \(N=3\).  For a \(Y\)-crossing,
they are adjacent precisely when \(N-2s_N=1\), equivalently when
\(N\in\{3,5\}\).  Thus, in the present case \(N\notin\{3,5\}\), neither
pair of endpoint cells is adjacent, and multiplication of the child
partition functions is exact.

For \(N=3\), consider a crossing from \(a_{0,m+1}\) to
\(a_{1,m+1}\).  The direct itinerary
\((\Delta_0,\Delta_1)\) avoids \(a_{2,m+1}\), and each of its two pieces
may either avoid or visit its contact with \(\Delta_2\).  The two
contacts are distinct, so the total contribution is
\((x_m+y_m)^2\).

The long itinerary is
\((\Delta_0,\Delta_2,\Delta_1)\).  Its two endpoint subcells also meet
at one common contact.  Either endpoint piece may visit that contact, but
not both, since the resulting walk would otherwise repeat the vertex.
The total weight of the allowed endpoint choices is therefore
\[
  (x_m+y_m)^2-y_m^2=x_m(x_m+2y_m).
\]
The middle piece has type \(x_m\) when the parent crossing avoids
\(a_{2,m+1}\), and type \(y_m\) when it visits that vertex.  The long
itinerary consequently contributes \(x_m^2(x_m+2y_m)\) to
\(x_{m+1}\) and \(x_my_m(x_m+2y_m)\) to \(y_{m+1}\).
This proves \eqref{eq:sg-crossing-recursion}, the two-variable recursion
introduced in~\cite{HHK1990}.

It remains to prove \eqref{eq:pentagasket-recursion}.  Consider first a
crossing from \(a_{0,m+1}\) to \(a_{2,m+1}\).  Its short itinerary
\((\Delta_0,\Delta_1,\Delta_2)\) has two endpoint pieces of total weight
\(\bar\phi_m^2\) and one
interior piece of weight \(\bar\theta_m\).  It avoids
\(a_{3,m+1}\), so its contribution to \(\phi_{1,m+1}\) is
\(\bar\phi_m^2\bar\theta_m\).  The long itinerary is
\((\Delta_0,\Delta_4,\Delta_3,\Delta_2)\).
The piece in \(\Delta_4\) has total weight \(\bar\theta_m\), whereas
the piece in \(\Delta_3\) has weight \(\theta_{1,m}\) or
\(\theta_{2,m}\), according as it
avoids or visits its outer vertex \(a_{3,m+1}\).
Hence the long itinerary contributes
\(\bar\phi_m^2\bar\theta_m\theta_{1,m}\) to
\(\phi_{1,m+1}\) and
\(\bar\phi_m^2\bar\theta_m\theta_{2,m}\) to
\(\phi_{2,m+1}\).

Next, consider a crossing from \(a_{-2,m+1}=a_{3,m+1}\) to
\(a_{2,m+1}\).  Its short itinerary
\((\Delta_3,\Delta_2)\) contributes \(\bar\phi_m^2\) and avoids
\(a_{0,m+1}\), so it contributes \(\bar\phi_m^2\) to
\(\theta_{1,m+1}\).  Its long itinerary is
\((\Delta_3,\Delta_4,\Delta_0,\Delta_1,\Delta_2)\).  The pieces in
\(\Delta_4\) and \(\Delta_1\) contribute
\(\bar\theta_m^2\).  The middle piece in \(\Delta_0\) has type
\(\theta_{1,m}\) or \(\theta_{2,m}\), according as the parent crossing
avoids or visits \(a_{0,m+1}\).  The endpoint cells of the long itinerary
share one common contact.  Both
endpoint pieces may visit it, but not simultaneously.  The total weight
of the allowed endpoint choices is
\[
  \bar\phi_m^2-\phi_{2,m}^2
  =
  \phi_{1,m}(\phi_{1,m}+2\phi_{2,m}).
\]
Multiplication of these factors gives the two
\(\theta\)-recursions and completes the proof of
\eqref{eq:pentagasket-recursion}.
\end{proof}

\begin{proposition}
\label{prop:ngasket-completeness}
For every \(\beta\ge0\),
\[
  \sup_{m\ge0}\|\mathbf C_m^{(N)}(\beta)\|_\infty<\infty
  \quad\Longleftrightarrow\quad
  \sup_{m\ge0}\|\bZ_m(\beta)\|_\infty<\infty.
\]
Consequently,
\(\beta_{\mathrm c,N}^{\mathrm{cr}}=\beta_{\mathrm c,N}\).
\end{proposition}

\begin{proof}
Each partition function defining \(\mathbf C_m^{(N)}(\beta)\) is a finite
sum of partition functions in \(\bZ_m(\beta)\): a crossing is a
one-component boundary forest, and its boundary state records the
additional boundary vertices, if any, which the crossing visits.  Thus
\[
  \sup_{m\ge0}\|\bZ_m(\beta)\|_\infty<\infty
  \quad\Longrightarrow\quad
  \sup_{m\ge0}\|\mathbf C_m^{(N)}(\beta)\|_\infty<\infty.
\]

Conversely, suppose that
\(\|\mathbf C_m^{(N)}(\beta)\|_\infty\le K\) for every \(m\).
Define the two total crossing weights by
\[
  \mathsf O_m
  =
  \begin{cases}
    x_m+y_m,&N=3,\\
    \bar\phi_m,&N=5,\\
    X_m,&N\notin\{3,5\},
  \end{cases}
  \qquad
  \mathsf P_m
  =
  \begin{cases}
    x_m+y_m,&N=3,\\
    \bar\theta_m,&N=5,\\
    Y_m,&N\notin\{3,5\}.
  \end{cases}
\]
Both quantities are at most \(2K\).

For distinct \(i,j\in\mathbb Z/N\mathbb Z\), let
\(P_m^{ij}(\beta)\) be the one-edge partition function from
Remark~\ref{rem:one-edge-states}.  The same contact argument used in the proof of
Proposition~\ref{prop:ngasket-recursions} shows that a path counted by
\(P_m^{ij}(\beta)\) has one of the two simple arcs of \(C_N\) as its
subcell itinerary.  Let \(r\in\{1,\ldots,N-1\}\) be determined by
\(j-i\equiv r\pmod N\).  Discarding possible intersections between the
pieces gives, for \(m\ge1\),
\[
  P_m^{ij}(\beta)
  \le
  \mathsf O_{m-1}^{\,2}
  \left(
    \mathsf P_{m-1}^{\,r-1}
    +\mathsf P_{m-1}^{\,N-r-1}
  \right).
\]
Since \(N\) is fixed, the right-hand side is bounded uniformly in
\(m,i,j\); the finitely many level-zero values are bounded as well.
Remark~\ref{rem:one-edge-states} now gives
\(\sup_m\|\bZ_m(\beta)\|_\infty<\infty\), which proves the reverse
implication.
Taking the infima in
\eqref{eq:ngasket-critical-beta} and \eqref{eq:critical-point} proves
the desired equality of critical inverse temperatures.
\end{proof}

\subsection{\texorpdfstring{Critical values and
connective constants}{Critical values and connective constants}}

\begin{proof}[Proof of Theorem~\ref{thm:ngasket}]
Proposition~\ref{prop:ngasket-geometry}(i) places \(F_\infty^{(N)}\) in
the class of nested fractal graphs considered in
Theorems~\ref{thm:connective} and~\ref{thm:critical-identity}.
The
closed polynomial recursions are given in
Proposition~\ref{prop:ngasket-recursions}.  By
Proposition~\ref{prop:ngasket-completeness} and
Theorem~\ref{thm:critical-identity}(i),
\[
  \beta_{\mathrm c,N}^{\mathrm{cr}}
  =
  \beta_{\mathrm c,N}
  =
  \log\mu_N.
\]

It remains to compute this value when \(N=6\) or \(N=9\).  In these
two cases the initial crossing variables in
\eqref{eq:ngasket-recursion} are equal, and the recursion preserves
their equality.  Writing their common value as \(t_m\), we have
\[
  t_0=e^{-d_N\beta}+e^{-2d_N\beta},
  \qquad
  t_{m+1}=F(t_m),
  \qquad
  F(u)=u^{d_N+1}+u^{2d_N+1}.
\]
The fixed points of \(F\) in \([0,\infty)\) are \(0\) and
\[
  t_*=\tau^{1/d_N},
\]
because, for \(u>0\), \(F(u)=u\) is equivalent to
\(u^{d_N}+u^{2d_N}=1\), whose unique solution is \(t_*\) since
\(\tau+\tau^2=1\).  Moreover,
\[
  F(u)-u=u\bigl(u^{d_N}+u^{2d_N}-1\bigr),
  \qquad
  \operatorname{sgn}\bigl(F(u)-u\bigr)
  =\operatorname{sgn}(u-t_*)
  \quad(u>0),
\]
since \(u\mapsto u^{d_N}+u^{2d_N}\) is strictly increasing.
If \(t_0<t_*\), induction gives
\(0<t_{m+1}<t_m<t_*\); its limit is a fixed point below \(t_*\), hence
\(t_m\to0\).  If \(t_0=t_*\), then \(t_m=t_*\) for every \(m\).  If
\(t_0>t_*\), induction gives \(t_{m+1}>t_m>t_*\); if the sequence were
bounded, its limit would be a fixed point above \(t_*\), which is
impossible.  Hence \(t_m\to\infty\).

Consequently, the crossing orbit is bounded exactly when
\[
  e^{-d_N\beta}+e^{-2d_N\beta}\le\tau^{1/d_N}.
\]
Since the left-hand side decreases continuously from \(2\) to \(0\) as
\(\beta\) runs from \(0\) to \(\infty\),
\eqref{eq:ngasket-critical-beta} shows that
\(\beta_{\mathrm c,N}^{\mathrm{cr}}\) is the unique positive solution of
\[
  e^{-d_N\beta_{\mathrm c,N}^{\mathrm{cr}}}
  +e^{-2d_N\beta_{\mathrm c,N}^{\mathrm{cr}}}
  =
  \tau^{1/d_N}.
\]
Setting \(z=e^{-d_N\beta_{\mathrm c,N}^{\mathrm{cr}}}\) gives
\[
  z+z^2=\tau^{1/d_N},
  \qquad
  z=\frac{\sqrt{1+4\tau^{1/d_N}}-1}{2}.
\]
Since
\(\mu_N=e^{\beta_{\mathrm c,N}^{\mathrm{cr}}}\), this is precisely
\eqref{eq:ngasket-algebraic}.
\end{proof}

\section{Ratio limits}

We prove Theorem~\ref{thm:ratio-limit}(i) as follows.  In
Subsection~\ref{subsec:local-surgery}, Lemma~\ref{lem:cell-surgery} gives
the state-preserving cell surgery, and
Proposition~\ref{prop:canonical-density} gives the positive-density estimate
for canonical cell patterns.  In
Subsection~\ref{subsec:marked-switches}, the marked identities of
Lemma~\ref{lem:switch-identities} convert this abundance, under the
flexibility hypothesis, into the regularity estimates of
Proposition~\ref{prop:h-ratio-regularity}, and
Lemma~\ref{lem:ratio-criterion} then gives ratio convergence.
Theorem~\ref{thm:ratio-limit}(i) is a
finite-ramification version of the local replacement mechanism in ratio
theorems such as Kesten's~\cite{Kesten1963}.
Subsection~\ref{subsec:sg} verifies the hypothesis for the
Sierpi\'nski gasket.
Subsection~\ref{subsec:ngasket-flexibility} determines the smallest
flexibility step for every polygonal \(N\)-gasket, thereby proving
Theorem~\ref{thm:ratio-limit}(ii), and
Subsection~\ref{subsec:vicsek} proves part~(iii).

\begin{definition}[Finite-scale flexibility]
\label{def:boundary-flexibility}
For \(r\ge1\), let
\[
  \cS_r=\{\sigma\in\cS:\cC_r(\sigma)\ne\varnothing\}.
\]
Given an integer \(h\ge1\), the graph is
\emph{\(h\)-flexible at scale \(r\)} if, for every
\(\sigma\in\cS_r\), there are forests
\[
  \Gamma_\sigma^-,\Gamma_\sigma^+\in\cC_r(\sigma)
  \quad\text{such that}\quad
  \#E(\Gamma_\sigma^+)=\#E(\Gamma_\sigma^-)+h.
\]
Thus every realizable way of connecting the boundary of an \(r\)-cell
has two internal realizations whose lengths differ by exactly \(h\).
The graph is \emph{\(h\)-flexible} if it is \(h\)-flexible at some
finite scale; such an \(h\) is called a \emph{flexibility step}.
\end{definition}

\subsection{\texorpdfstring{Local surgery and canonical
patterns}{Local surgery and canonical patterns}}
\label{subsec:local-surgery}

Fix a scale \(r\ge1\).  For each \(r\)-cell \(\Delta\), fix one of the
isometries
\[
  S_\Delta:F_r\longrightarrow\Delta
\]
provided by \eqref{eq:m-cell}.  If \(w\) traverses an edge of \(\Delta\),
let \(\Gamma_\Delta(w)\) be the edge subgraph of \(\Delta\) consisting of
all edges of \(w\) that belong to \(\Delta\).

\begin{lemma}
\label{lem:cell-surgery}
Let \(w\) be a self-avoiding walk, and let \(\Delta\) be an \(r\)-cell
which contains neither endpoint of \(w\).  If
\(\Gamma_\Delta(w)\ne\varnothing\), then
\(S_\Delta^{-1}\Gamma_\Delta(w)\) is a boundary forest in \(F_r\).
Moreover, it may be replaced by any other boundary forest with the same
boundary state.  The replacement yields a self-avoiding walk with the same
endpoints as \(w\).  Such replacements may be performed
simultaneously in any finite collection of distinct \(r\)-cells.
\end{lemma}

\begin{proof}
Since \(\Delta\) contains neither endpoint of \(w\), every degree-one
vertex of \(\Gamma_\Delta(w)\) lies in \(\partial\Delta\), so its pullback
is a boundary forest.  By Proposition~\ref{prop:cell-hierarchy}, the
unchanged part of \(w\) meets \(\Delta\) only at boundary vertices, whose
degrees and connections are recorded by the boundary state.  A same-state
replacement therefore reconnects the unchanged pieces in the same order
without a new intersection.  Since distinct \(r\)-cells have disjoint
edges and share only boundary vertices, the replacements may be performed
simultaneously.
\end{proof}

The next proposition gives the positive-density estimate for canonical cell
patterns.

\begin{proposition}
\label{prop:canonical-density}
Fix \(r\ge1\).  For each \(\sigma\in\cS_r\), fix a forest
\(\Gamma_\sigma\in\cC_r(\sigma)\).  For a walk \(w\) and an \(r\)-cell
\(\Delta\) that contains neither endpoint of \(w\) and satisfies
\(\Gamma_\Delta(w)\ne\varnothing\), put
\[
  \sigma_\Delta(w)
  =\operatorname{tr}_r\bigl(S_\Delta^{-1}\Gamma_\Delta(w)\bigr).
\]
We call \(\Delta\) \emph{canonical for \(w\)} if
\[
  S_\Delta^{-1}\Gamma_\Delta(w)=\Gamma_{\sigma_\Delta(w)}.
\]
Let \(\mathsf m_r(w)\) be the number of canonical \(r\)-cells of \(w\).
There is a constant \(a>0\) such that
\begin{equation}
  \limsup_{n\to\infty}
  \#\{w\in\cW_{O,n}:\mathsf m_r(w)\le\lfloor an\rfloor\}^{1/n}
  <\mu.
  \label{eq:canonical-density}
\end{equation}
\end{proposition}

\begin{proof}
Put \(e_r=\#E(F_r)\), and let \(J_r\) be the number of nonempty edge
subgraphs of \(F_r\); it is enough to use the bound \(J_r\le2^{e_r}\).
Every edge of a walk belongs to one \(r\)-cell by
Proposition~\ref{prop:cell-hierarchy}, and one \(r\)-cell contains at
most \(e_r\) of its edges.  Hence an \(n\)-step walk traverses edges in
at least \(\lceil n/e_r\rceil\) distinct \(r\)-cells.

A vertex belongs to at most \(D\) distinct \(r\)-cells.  Indeed, if it
belongs to more than one, it is a boundary vertex of each of them, and
each cell contributes a different incident edge; the degree is at most
\(D\).  After discarding the cells which contain one of the two endpoints
of the walk, at least
\[
  \left\lceil\frac{n}{e_r}\right\rceil-2D
\]
cells remain.  Thus, with
\(\varpi_r=(2e_r)^{-1}\), every sufficiently long walk has at least
\(\lceil\varpi_r n\rceil\) endpoint-free cells with a nonempty restriction.

Let \(\mathcal B_n(a)\) denote the cardinality of the set in
\eqref{eq:canonical-density}.  Choose \(a,\delta>0\) with
\(a+\delta<\varpi_r\), and put \(k=\lfloor\delta n\rfloor\).  From
each walk counted by \(\mathcal B_n(a)\),
choose \(k\) of its endpoint-free cells whose restrictions are not
canonical.  There are at least
\[
  \binom{\lfloor(\varpi_r-a)n\rfloor}{k}
\]
choices.  In each chosen cell replace the restriction by the canonical
forest having the same boundary state.  By Lemma~\ref{lem:cell-surgery},
all \(k\) replacements can be made simultaneously and the output is a
self-avoiding walk with the same endpoints.  Each replacement changes the
length by at most \(e_r\), so the output length lies between
\(n-e_rk\) and \(n+e_rk\).

The chosen cells were noncanonical before replacement, and no unchosen
restriction is changed.  Hence the output has exactly \(k\) more canonical
cells than the input, and therefore at most \(k+\lfloor an\rfloor\)
canonical cells.  Given an output, the modified cells must be chosen among
these canonical cells; there are at most
\[
  \binom{k+\lfloor an\rfloor}{k}
\]
ways.  In each selected cell there are at most \(J_r\) possible original
restrictions.  Keeping choices which do not reconstruct a valid walk
only enlarges the bound.  We have proved
\begin{equation}
  \mathcal B_n(a)\binom{\lfloor(\varpi_r-a)n\rfloor}{k}
  \le
  J_r^k\binom{k+\lfloor an\rfloor}{k}
  \sum_{\substack{j\ge0\\|j-n|\le e_rk}}c_j.
  \label{eq:cell-surgery-count}
\end{equation}

Let \(\mathcal H(t)=-t\log t-(1-t)\log(1-t)\), with
\(\mathcal H(0)=\mathcal H(1)=0\).  For every \(\lambda>\mu\),
Theorem~\ref{thm:connective} gives a constant \(C_\lambda<\infty\) such that
\[
  c_j\le C_\lambda\lambda^j,\qquad j\ge0.
\]
Apply
\(\binom{q}{p}\le\exp\bigl(q\mathcal H(p/q)\bigr)\)
to the binomial coefficient
\[
  \binom{k+\lfloor an\rfloor}{k}
\]
on the right-hand side of \eqref{eq:cell-surgery-count}.  Apply the
Stirling lower bound
\(\binom{q}{p}\ge\exp\bigl(q\mathcal H(p/q)-o(q)\bigr)\)
to the binomial coefficient
\[
  \binom{\lfloor(\varpi_r-a)n\rfloor}{k}
\]
on the left-hand side of \eqref{eq:cell-surgery-count}.
Using \(c_j\le C_\lambda\lambda^j\) and
\(\sum_{|j-n|\le e_rk}c_j\le C'_{\lambda,r}\lambda^{n+e_rk}\),
divide \eqref{eq:cell-surgery-count} by its left-hand binomial
coefficient and take \(n\)-th roots.  Since
\[
  \frac{k}{n}\longrightarrow\delta,
  \qquad
  \frac{k+\lfloor an\rfloor}{n}\longrightarrow\delta+a,
  \qquad
  \frac{\lfloor(\varpi_r-a)n\rfloor}{n}
  \longrightarrow\varpi_r-a,
\]
the factor \(J_r^k\) contributes \(\delta\log J_r\) to the logarithm
of the resulting upper bound.  The two binomial coefficients contribute,
respectively,
\[
  (\delta+a)\mathcal H\left(\frac{\delta}{\delta+a}\right)
  \quad\text{and}\quad
  -(\varpi_r-a)\mathcal H
     \left(\frac{\delta}{\varpi_r-a}\right),
\]
while the sum over \(j\) contributes at most
\((1+e_r\delta)\log\lambda\).  It follows that
\[
\begin{aligned}
  \limsup_{n\to\infty}\mathcal B_n(a)^{1/n}
 \le \lambda^{1+e_r\delta}\exp\bigg\{&
   \delta\log J_r
  +(\delta+a)\mathcal H\left(\frac{\delta}{\delta+a}\right)
  -(\varpi_r-a)\mathcal H\left(\frac{\delta}{\varpi_r-a}\right)
 \bigg\}.
\end{aligned}
\]

Set \(a=\delta^2\).  Then
\[
  (\delta+\delta^2)
  \mathcal H\left(\frac{\delta}{\delta+\delta^2}\right)
  =
  O\left(\delta^2\log(1/\delta)\right).
\]
Moreover, \(\delta\log J_r=O(\delta)\), whereas
\[
  (\varpi_r-a)\mathcal H
  \left(\frac{\delta}{\varpi_r-a}\right)
  =\delta\log(1/\delta)+O(\delta).
\]
Thus the expression in braces is
\(-\delta\log(1/\delta)+O(\delta)\).  Since
\(e_r\delta\log\mu=O(\delta)\), for all sufficiently small \(\delta\)
the expression in braces is smaller than \(-e_r\delta\log\mu\).
Fix such a \(\delta\).  The continuous expression on the right of the
preceding estimate has value strictly below \(\mu\) at \(\lambda=\mu\).
Hence it remains below \(\mu\) for some \(\lambda>\mu\) sufficiently
close to \(\mu\), for which the estimate is valid.
This proves \eqref{eq:canonical-density}.
\end{proof}

\subsection{\texorpdfstring{Ratio
estimates}{Ratio estimates}}
\label{subsec:marked-switches}

Suppose from now on that the graph is \(h\)-flexible at scale \(r\).
For each \(\sigma\in\cS_r\), choose and fix forests
\(\Gamma_\sigma^-,\Gamma_\sigma^+\in\cC_r(\sigma)\), as in
Definition~\ref{def:boundary-flexibility}, with
\(\Gamma_\sigma^+\) having exactly \(h\) more edges than
\(\Gamma_\sigma^-\).  For a walk \(w\) and an
endpoint-free \(r\)-cell \(\Delta\) with
\(\Gamma_\Delta(w)\ne\varnothing\), write
\[
  \sigma_\Delta(w)
  =\operatorname{tr}_r\bigl(S_\Delta^{-1}\Gamma_\Delta(w)\bigr).
\]
Let \(m_-(w)\) and \(m_+(w)\) be the numbers of such cells for which,
respectively,
\[
  S_\Delta^{-1}\Gamma_\Delta(w)
  =\Gamma_{\sigma_\Delta(w)}^-
  \quad\text{or}\quad
  S_\Delta^{-1}\Gamma_\Delta(w)
  =\Gamma_{\sigma_\Delta(w)}^+.
\]
Cells satisfying the first and second conditions are called
\emph{minus occurrences} and \emph{plus occurrences}, respectively.
Replacing a minus occurrence by
the corresponding plus occurrence is a \emph{switch}: it preserves the
boundary state and increases the walk length by exactly \(h\).
Apply Proposition~\ref{prop:canonical-density} first with
\(\Gamma_\sigma=\Gamma_\sigma^-\) and then with
\(\Gamma_\sigma=\Gamma_\sigma^+\).  Together with
Theorem~\ref{thm:connective}, this gives constants
\(\alpha>0\), \(\vartheta\in(0,1)\), and \(n_0\) such that,
for each \(\varepsilon\in\{-,+\}\),
\begin{equation}
  \frac{\#\{w\in\cW_{O,n}:
    m_\varepsilon(w)<\lfloor\alpha n\rfloor\}}{c_n}
  \le\vartheta^n,
  \qquad n\ge n_0.
  \label{eq:switch-density}
\end{equation}
Indeed, take \(\alpha\) smaller than both constants supplied by
Proposition~\ref{prop:canonical-density}.  For all sufficiently large
\(n\), each set in \eqref{eq:switch-density} is then contained in the
corresponding exceptional set from that proposition.  The strict inequality
between the exceptional exponential growth rates and \(\mu\) gives a common
\(\vartheta\).  A walk in which at least one switch type occurs fewer
than \(\lfloor\alpha n\rfloor\) times belongs to one of the two exceptional
sets in \eqref{eq:switch-density}.  Hence
\[
  \frac{\#\{w\in\cW_{O,n}:
    m_-(w)<\lfloor\alpha n\rfloor\ \text{or}\
    m_+(w)<\lfloor\alpha n\rfloor\}}{c_n}
  \le 2\vartheta^n,
  \qquad n\ge n_0.
\]
Thus, except for an exponentially small proportion of walks, both
switch types occur at least \(\lfloor\alpha n\rfloor\) times.

The next lemma counts marked switches in two ways.

\begin{lemma}
\label{lem:switch-identities}
For \(j=0,1\), let
\[
  \mathcal E_s^{(j)}
  =
  \#\{w\in\cW_{O,s}:m_+(w)\le j\}.
\]
Then, for every \(n\ge1\),
\begin{equation}
  c_{n+h}
  =
  \mathcal E_{n+h}^{(0)}
  +
  \sum_{w\in\cW_{O,n}}
  \frac{m_-(w)}{m_+(w)+1},
  \label{eq:one-cell-switch}
\end{equation}
and
\begin{equation}
  c_{n+2h}
  =
  \mathcal E_{n+2h}^{(1)}
  +
  \sum_{w\in\cW_{O,n}}
  \frac{m_-(w)(m_-(w)-1)}
       {(m_+(w)+1)(m_+(w)+2)}.
  \label{eq:two-cell-switches}
\end{equation}
\end{lemma}

\begin{proof}
For \eqref{eq:one-cell-switch}, mark one minus occurrence \(\Delta\) in an
\(n\)-step walk and replace \(\Gamma_{\sigma_\Delta(w)}^-\) by
\(\Gamma_{\sigma_\Delta(w)}^+\).  By Lemma~\ref{lem:cell-surgery}, this
produces a self-avoiding walk of length \(n+h\), and
\[
  m_-(w')=m_-(w)-1,\qquad m_+(w')=m_+(w)+1.
\]
Replacing the marked plus occurrence in \(w'\) by the corresponding
minus occurrence recovers \(w\).  Thus the operation is reversible.

Assign the marked walk \(w\) the weight \(1/(m_+(w)+1)\).  For a fixed
\(w\), summing over its \(m_-(w)\) possible marks gives
\[
  \frac{m_-(w)}{m_+(w)+1}.
\]
Conversely, a walk \(w'\) with \(m_+(w')\ge1\) has \(m_+(w')\) possible
marked plus occurrences at which the switch can be reversed.  Since
\(m_+(w')=m_+(w)+1\), each such choice has weight \(1/m_+(w')\), and
their total weight is one.  Therefore every \((n+h)\)-step walk with at
least one plus occurrence contributes one to the weighted sum.  The
remaining \((n+h)\)-step walks are counted by
\(\mathcal E_{n+h}^{(0)}\), which proves
\eqref{eq:one-cell-switch}.

For \eqref{eq:two-cell-switches}, mark an ordered pair of distinct minus
occurrences and switch both cells.  By Lemma~\ref{lem:cell-surgery}, this
produces a self-avoiding walk \(w'\) of length \(n+2h\), with
\(m_+(w')=m_+(w)+2\).  Replacing the two marked plus occurrences by the
corresponding minus occurrences recovers \(w\).

Assign each marked walk the weight
\[
  \frac{1}{(m_+(w)+1)(m_+(w)+2)}.
\]
For a fixed \(w\), summing over its
\(m_-(w)(m_-(w)-1)\) ordered pairs gives the summand in
\eqref{eq:two-cell-switches}.  Conversely, a walk \(w'\) with
\(m_+(w')\ge2\) has \(m_+(w')(m_+(w')-1)\) ordered pairs at which both
switches can be reversed.  Each pair has weight
\[
  \frac{1}{(m_+(w')-1)m_+(w')},
\]
so their total weight is one.  Therefore every \((n+2h)\)-step walk with
at least two plus occurrences contributes one to the weighted sum.
The remaining walks are counted by \(\mathcal E_{n+2h}^{(1)}\), which
proves \eqref{eq:two-cell-switches}.
\end{proof}

\begin{proposition}
\label{prop:h-ratio-regularity}
Put
\[
  R_n=\frac{c_{n+h}}{c_n}.
\]
There are constants \(b>0\), \(B<\infty\), and
\(n_1\) such that
\begin{equation}
  R_n\ge b,
  \qquad
  R_{n+h}-R_n\ge-\frac{B}{n},
  \qquad n\ge n_1.
  \label{eq:h-ratio-regularity}
\end{equation}
\end{proposition}

\begin{proof}
We first estimate the two error terms in
\eqref{eq:one-cell-switch} and \eqref{eq:two-cell-switches}.  An
\((n+j)\)-step walk has an \(n\)-step initial part.  After this part is
fixed, there are at most \(D\) choices at each of the remaining \(j\)
steps.  Therefore
\[
  c_{n+j}\le D^j c_n,\qquad j\ge0.
\]
For sufficiently large \(n\),
\(0<\lfloor\alpha(n+h)\rfloor\) and
\(1<\lfloor\alpha(n+2h)\rfloor\).  A walk
counted by \(\mathcal E_{n+h}^{(0)}\) has
\(m_+(w)=0<\lfloor\alpha(n+h)\rfloor\), while a walk counted by
\(\mathcal E_{n+2h}^{(1)}\) has
\(m_+(w)\le1<\lfloor\alpha(n+2h)\rfloor\).  Applying
\eqref{eq:switch-density} at these two lengths gives
\[
\begin{aligned}
 \mathcal E_{n+h}^{(0)}
   &\le \vartheta^{n+h}c_{n+h}
    \le D^h\vartheta^{n+h}c_n,\\
 \mathcal E_{n+2h}^{(1)}
   &\le \vartheta^{n+2h}c_{n+2h}
    \le D^{2h}\vartheta^{n+2h}c_n.
\end{aligned}
\]
Since \(h\) is fixed, there are \(C_0<\infty\) and \(n_2\) such that
\begin{equation}
  \frac{\mathcal E_{n+h}^{(0)}+\mathcal E_{n+2h}^{(1)}}{c_n}
  \le C_0\vartheta^n,
  \qquad n\ge n_2.
  \label{eq:switch-error-bound}
\end{equation}

We next prove the lower bound for \(R_n\).  Since distinct \(r\)-cells
have disjoint edge sets,
\[
  m_-(w)\le n,\qquad m_+(w)\le n.
\]
For all sufficiently large \(n\), if
\(m_-(w)\ge \lfloor\alpha n\rfloor\), then
\[
  \frac{m_-(w)}{m_+(w)+1}
  \ge\frac{\lfloor\alpha n\rfloor}{n+1}
  \ge\frac{\alpha}{2}.
\]
By \eqref{eq:switch-density}, such walks form a proportion at least
\(1-\vartheta^n\) of \(\cW_{O,n}\).  Since
\(\mathcal E_{n+h}^{(0)}\ge0\), equation
\eqref{eq:one-cell-switch} therefore gives, for all sufficiently large
\(n\),
\[
  R_n
  \ge\frac{\alpha}{2}(1-\vartheta^n)
  \ge\frac{\alpha}{3}.
\]
This proves the first inequality in \eqref{eq:h-ratio-regularity} with
\(b=\alpha/3\).

It remains to estimate \(R_{n+h}-R_n\).  Put
\[
\begin{aligned}
  \mathcal M_{1,n}&=\frac1{c_n}\sum_{w\in\cW_{O,n}}
       \frac{m_-(w)}{m_+(w)+1},\\
  \mathcal M_{2,n}&=\frac1{c_n}\sum_{w\in\cW_{O,n}}
       \frac{m_-(w)(m_-(w)-1)}
       {(m_+(w)+1)(m_+(w)+2)}.
\end{aligned}
\]
By Lemma~\ref{lem:switch-identities}, equations
\eqref{eq:one-cell-switch} and \eqref{eq:two-cell-switches} give
\[
  R_n=\mathcal M_{1,n}+\varepsilon_{1,n},\qquad
  R_nR_{n+h}=\mathcal M_{2,n}+\varepsilon_{2,n},
\]
where
\[
  \varepsilon_{1,n}=\frac{\mathcal E_{n+h}^{(0)}}{c_n},\qquad
  \varepsilon_{2,n}=\frac{\mathcal E_{n+2h}^{(1)}}{c_n}.
\]
Thus \(\varepsilon_{1,n},\varepsilon_{2,n}\ge0\), and
\eqref{eq:switch-error-bound} gives
\(\varepsilon_{1,n}+\varepsilon_{2,n}\le C_0\vartheta^n\).
Also \(0\le\mathcal M_{1,n}\le n\).  Choose
\(\vartheta_1\in(\vartheta,1)\).  Since
\(n\vartheta^n=O(\vartheta_1^n)\), there is \(C_1<\infty\) such that
\[
  2\mathcal M_{1,n}\varepsilon_{1,n}+\varepsilon_{1,n}^2
  \le C_1\vartheta_1^n.
\]
Using \(\varepsilon_{2,n}\ge0\), we obtain
\[
  R_nR_{n+h}-R_n^2
  \ge
  \mathcal M_{2,n}-\mathcal M_{1,n}^2-C_1\vartheta_1^n.
\]

Write \(\omega(w)=m_-(w)/(m_+(w)+1)\).  The Cauchy--Schwarz inequality gives
\[
  \mathcal M_{1,n}^2
  \le
  \frac1{c_n}\sum_{w\in\cW_{O,n}}\omega(w)^2.
\]
For each \(w\), a direct calculation gives
\[
\begin{aligned}
 &\frac{m_-(w)(m_-(w)-1)}
 {(m_+(w)+1)(m_+(w)+2)}
  -\omega(w)^2\\
 &\qquad =
 -\frac{m_-(w)(m_-(w)+m_+(w)+1)}
 {(m_+(w)+1)^2(m_+(w)+2)}.
\end{aligned}
\]
It follows that
\[
\begin{aligned}
 R_nR_{n+h}-R_n^2
 \ge
 -\frac1{c_n}\sum_{w\in\cW_{O,n}}
 \frac{m_-(w)(m_-(w)+m_+(w)+1)}
 {(m_+(w)+1)^2(m_+(w)+2)}
 -C_1\vartheta_1^n.
\end{aligned}
\]
For all sufficiently large \(n\),
\(\lfloor\alpha n\rfloor\ge\alpha n/2\).  Hence, for a walk satisfying
\(m_+(w)\ge \lfloor\alpha n\rfloor\), the fraction
inside the sum is at most
\[
  \frac{n(2n+1)}{\lfloor\alpha n\rfloor^3}
  \le\frac{24}{\alpha^3n}.
\]
For every walk, the same fraction is at most \(3n^2\).  By
\eqref{eq:switch-density}, walks with
\(m_+(w)<\lfloor\alpha n\rfloor\) form a proportion at most \(\vartheta^n\).
Consequently,
\[
 \frac1{c_n}\sum_{w\in\cW_{O,n}}
 \frac{m_-(w)(m_-(w)+m_+(w)+1)}
 {(m_+(w)+1)^2(m_+(w)+2)}
 \le
 \frac{24}{\alpha^3n}+3n^2\vartheta^n.
\]
Since \(n^2\vartheta^n=O(1/n)\) and
\(\vartheta_1^n=O(1/n)\), the preceding estimates give a constant
\(C_2<\infty\) such that
\[
  R_nR_{n+h}-R_n^2\ge-\frac{C_2}{n}
\]
for all sufficiently large \(n\).  Since the left-hand side equals
\(R_n(R_{n+h}-R_n)\) and \(R_n\ge b\), division by \(R_n\)
proves the second inequality in
\eqref{eq:h-ratio-regularity}, with
\(B=C_2/b\).
\end{proof}

The following lemma converts the estimates in
Proposition~\ref{prop:h-ratio-regularity} into convergence of the ratios.

\begin{lemma}
\label{lem:ratio-criterion}
Let \(\mathsf a_n>0\), and suppose that
\[
  \lim_{n\to\infty}\mathsf a_n^{1/n}=\lambda>0.
\]
Fix \(h\ge1\), and put
\(\mathcal R_n=\mathsf a_{n+h}/\mathsf a_n\).  If there are
\(b>0\), \(K<\infty\), and \(n_0\) such that
\[
  \mathcal R_n\ge b,\qquad
  \mathcal R_{n+h}-\mathcal R_n\ge-\frac{K}{n},
  \qquad n\ge n_0,
\]
then
\[
  \lim_{n\to\infty}\mathcal R_n=\lambda^h.
\]
\end{lemma}

\begin{proof}
Fix \(s\in\{0,\ldots,h-1\}\), and set
\[
  \mathsf a_k^{(s)}=\mathsf a_{s+kh},\qquad
  \varrho_k=\frac{\mathsf a_{k+1}^{(s)}}{\mathsf a_k^{(s)}}.
\]
Put \(L=\lambda^h\), \(C=K/h\), and
\(\kappa=C/b\).  Then
\[
  \bigl(\mathsf a_k^{(s)}\bigr)^{1/k}\longrightarrow L,\qquad
  \varrho_k=\mathcal R_{s+kh}\ge b,
  \qquad
  \varrho_{k+1}-\varrho_k\ge-\frac{C}{k}
\]
for all sufficiently large \(k\); since \(s\) is arbitrary, it is enough
to prove that \(\varrho_k\to L\) for this fixed \(s\).

For fixed \(\mathfrak m\ge1\), set
\[
  Q_{\mathfrak m,q}
  =
  \left(
    \frac{\mathsf a_{(\mathfrak m+1)q+1}^{(s)}}
         {\mathsf a_{\mathfrak m q}^{(s)}}
  \right)^{1/(q+1)}.
\]
The root limit gives \(Q_{\mathfrak m,q}\to L\), since
\(\log\mathsf a_k^{(s)}=k\log L+o(k)\) and the two indices in the
quotient differ by \(q+1\).
For sufficiently large \(q\) and \(0\le k\le q\),
\[
\begin{aligned}
  \varrho_{\mathfrak m q+k}
  &\le\varrho_{(\mathfrak m+1)q}
      +C\sum_{\ell=\mathfrak m q+k}^{(\mathfrak m+1)q-1}\frac1\ell\\
  &\le\varrho_{(\mathfrak m+1)q}
      \left(1+\frac{\kappa}{\mathfrak m}\right).
\end{aligned}
\]
Multiplying over \(k=0,\ldots,q\) yields
\[
  \liminf_{q\to\infty}\varrho_{(\mathfrak m+1)q}
  \ge\frac{L}{1+\kappa/\mathfrak m}.
\]
If \(j=(\mathfrak m+1)q\le k<j+\mathfrak m+1\), then
\[
  \varrho_k\ge\varrho_j-\frac{C\mathfrak m}{j}.
\]
Hence the same lower bound holds for
\(\liminf_{k\to\infty}\varrho_k\).  Letting
\(\mathfrak m\to\infty\) gives
\[
  \liminf_{k\to\infty}\varrho_k\ge L.
\]

For the reverse inequality, fix an integer
\(\mathfrak m>\kappa\).  For sufficiently large \(q\) and
\(0\le k\le q\),
\[
\begin{aligned}
  \varrho_{\mathfrak m q+k}
  &\ge\varrho_{\mathfrak m q}
      -C\sum_{\ell=\mathfrak m q}^{\mathfrak m q+k-1}\frac1\ell\\
  &\ge\varrho_{\mathfrak m q}
      \left(1-\frac{\kappa}{\mathfrak m}\right).
\end{aligned}
\]
Multiplication and the limit of \(Q_{\mathfrak m,q}\) give
\[
  \limsup_{q\to\infty}\varrho_{\mathfrak m q}
  \le\frac{L}{1-\kappa/\mathfrak m}.
\]
For any large \(k\), let \(j\) be the smallest multiple of
\(\mathfrak m\) with \(j\ge k\).  Then
\[
  \varrho_k\le\varrho_j+\frac{C\mathfrak m}{k}.
\]
Thus the same upper bound holds for
\(\limsup_{k\to\infty}\varrho_k\).  Letting
\(\mathfrak m\to\infty\) gives
\[
  \limsup_{k\to\infty}\varrho_k\le L.
\]
Together with the lower bound, this proves
\(\varrho_k\to L\) and completes the proof.
\end{proof}

\begin{proof}[Proof of Theorem~\ref{thm:ratio-limit}\textup{(i)}]
Theorem~\ref{thm:connective} gives
\(\lim_n c_n^{1/n}=\mu\).  Proposition~\ref{prop:h-ratio-regularity},
through \eqref{eq:h-ratio-regularity}, verifies the two hypotheses of
Lemma~\ref{lem:ratio-criterion} for \(\mathsf a_n=c_n\), which gives
\eqref{eq:h-step-ratio}.
\end{proof}

\subsection{\texorpdfstring{The Sierpi\'nski
gasket}{The Sierpinski gasket}}
\label{subsec:sg}

The flexibility condition can be checked directly on the
Sierpi\'nski gasket.

\begin{figure}[ht]
\centering
\begin{tabular}{@{}c@{\hspace{1.2cm}}c@{}}
  \begin{overpic}[width=.14\textwidth]{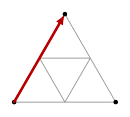}
    \put(9.5,2){\makebox(0,0)[r]{\scriptsize \(O\)}}
    \put(50,81){\makebox(0,0)[b]{\scriptsize \(A\)}}
    \put(90.5,2){\makebox(0,0)[l]{\scriptsize \(B\)}}
  \end{overpic}
  &
  \begin{overpic}[width=.14\textwidth]{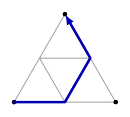}
    \put(9.5,2){\makebox(0,0)[r]{\scriptsize \(O\)}}
    \put(50,81){\makebox(0,0)[b]{\scriptsize \(A\)}}
    \put(90.5,2){\makebox(0,0)[l]{\scriptsize \(B\)}}
  \end{overpic}
  \\[-1pt]
  \scriptsize \(2\): avoids \(B\)
  & \scriptsize \(3\): avoids \(B\)
  \\[12pt]
  \begin{overpic}[width=.14\textwidth]{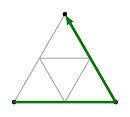}
    \put(9.5,2){\makebox(0,0)[r]{\scriptsize \(O\)}}
    \put(50,81){\makebox(0,0)[b]{\scriptsize \(A\)}}
    \put(90.5,2){\makebox(0,0)[l]{\scriptsize \(B\)}}
  \end{overpic}
  &
  \begin{overpic}[width=.14\textwidth]{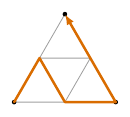}
    \put(9.5,2){\makebox(0,0)[r]{\scriptsize \(O\)}}
    \put(50,81){\makebox(0,0)[b]{\scriptsize \(A\)}}
    \put(90.5,2){\makebox(0,0)[l]{\scriptsize \(B\)}}
  \end{overpic}
  \\[-1pt]
  \scriptsize \(4\): visits \(B\)
  & \scriptsize \(5\): visits \(B\)
\end{tabular}
\caption{Representative pairs of boundary forests in \(F_1\) of the
Sierpi\'nski gasket.  The two forests in each pair have the same boundary
state and lengths differing by one.  Their symmetry images cover all six
states in \(\cS_1\).}
\label{fig:sg-boundary-pairs}
\end{figure}
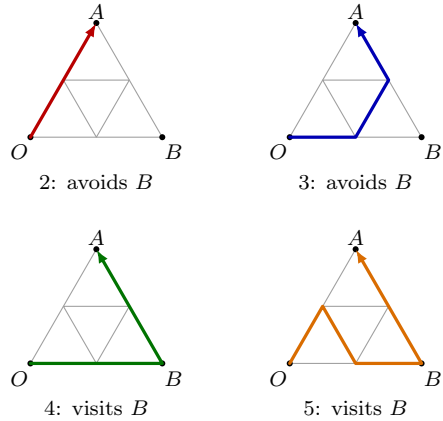
\FloatBarrier

\begin{proposition}
\label{prop:sg-flexible}
The Sierpi\'nski gasket graph is \(1\)-flexible at scale \(1\).
\end{proposition}

\begin{proof}
We work in \(F_1\).  There are six states in \(\cS_1\).  Every boundary
forest in \(F_1\) has one component: two nonempty components would require
at least four degree-one vertices in \(\partial F_1\), whereas
\(\#\partial F_1=3\).  Thus a state either joins two boundary vertices
without using the third, or visits all three boundary vertices with one of
them in the middle.  The first kind has realizations in \(F_1\) of lengths
\(2\) and \(3\), and the second has realizations in \(F_1\) of lengths
\(4\) and \(5\).  Figure~\ref{fig:sg-boundary-pairs} illustrates one pair
for each of these two kinds.  The symmetries of \(F_1\) transport these
pairs to every choice of endpoints and middle vertex.  Hence every state
in \(\cS_1\) has two representatives in \(F_1\) whose lengths differ by
one.  Therefore, the Sierpi\'nski gasket graph is \(1\)-flexible at
scale \(1\).
\end{proof}

Combining Proposition~\ref{prop:sg-flexible} with
Theorem~\ref{thm:ratio-limit}(i) gives
\(\lim_{n\to\infty}c_{n+1}/c_n=\mu\) for the Sierpi\'nski gasket.

\subsection{\texorpdfstring{Flexibility of polygonal
\(N\)-gaskets}{Flexibility of polygonal N-gaskets}}
\label{subsec:ngasket-flexibility}

We now determine the smallest flexibility step for every polygonal
\(N\)-gasket.  By self-similarity,
Proposition~\ref{prop:ngasket-geometry} applies within every
parent cell.  Thus each immediate subcell has an outer vertex on the
boundary of its parent and two contacts with its neighboring subcells.
A nonempty restriction of a boundary forest to this subcell is a union
of paths whose endpoints belong to these three vertices.  It has only
one component, since two path components would require four endpoints.
Its local behavior therefore has six possibilities: it either joins two
of the three vertices and avoids the third, or visits all three with one
of them in the middle.

Keep the notation of Definition~\ref{def:regular-ngasket}, and abbreviate
\[
  k=k_N,\qquad \ell=\ell_N,\qquad s=s_N=k+1,\qquad
  g=N-2s=2k+\ell-2.
\]
Set
\[
  d=2s-g=4-\ell,
  \qquad
  h_N=\gcd(g,d).
\]
Inside \(F_1^{(N)}\), let
\[
  P_N=\{a_{-s,1},a_{0,1},a_{s,1}\}.
\]
For a path in \(F_1^{(N)}\) with endpoints in \(P_N\), record only the
vertices of \(P_N\) in the order in which the path visits them.  A
sequence and its reversal are regarded as the same
\emph{\(P_N\)-pattern}.  Thus there are six patterns: three choices of
two visited vertices, and three choices of the middle vertex when all
three are visited.

\begin{lemma}
\label{lem:ngasket-three-vertex-switch}
For each of the six \(P_N\)-patterns, there are two self-avoiding paths in
\(F_1^{(N)}\) which have that pattern and whose lengths differ by \(h_N\).
\end{lemma}

\begin{proof}
By the graph construction and
Proposition~\ref{prop:ngasket-geometry}(ii)--(iii), \(F_1^{(N)}\) is the
union of \(N\) copies of \(F_0^{(N)}\), arranged in a cycle.
For this proof, write \(a_i:=a_{i,1}\), and let \(\zeta_i\) be the
contact between the \(i\)-th and \((i+1)\)-st copies; all indices are
understood modulo \(N\).  In the \(i\)-th copy, the arcs from \(a_i\)
to either contact, avoiding the other contact, have length \(s\).  The
arc between the two contacts, avoiding \(a_i\), has length \(g\).
We may therefore represent this copy by the weighted triangle
\[
  a_i\zeta_{i-1}\zeta_i,\qquad i\in\mathbb Z/N\mathbb Z,
\]
whose three edge weights are \(s,s,g\).  Each weighted edge represents
the corresponding arc in the original copy of \(F_0^{(N)}\).

Inside one weighted triangle there are three useful local changes:
\[
  \begin{aligned}
  \zeta_{i-1}\zeta_i
    &\longleftrightarrow
      \zeta_{i-1}a_i\zeta_i,\\
  a_i\zeta_i
    &\longleftrightarrow
      a_i\zeta_{i-1}\zeta_i,
  \qquad
  a_i\zeta_{i-1}
    \longleftrightarrow
      a_i\zeta_i\zeta_{i-1}.
  \end{aligned}
\]
The first change has length difference \(2s-g=d\).  Each of the two
changes in the second line has length difference \(g\); they are the
two symmetric forms of the same operation.  We call these operations
a \emph{\(d\)-switch} and a \emph{\(g\)-switch}, respectively.

Now \(P_N=\{a_{-s},a_0,a_s\}\).  The three copies carrying these
vertices divide the cyclic chain into three sections of lengths
\(s,s,g\), measured by the number of steps from one copy to the next.
The interiors of the three sections are disjoint, and neighboring
sections share an endpoint copy.

Fix a \(P_N\)-pattern and choose one of its two orientations.  Between
each consecutive pair in this ordered list, follow the section that
avoids the remaining vertex of \(P_N\).  In the first and last copies
of the section, use the length-\(s\) arc leading from the prescribed
outer vertex to the contact belonging to that section.  In every copy
strictly between them, use the length-\(g\) arc between the two contacts.
If the pattern
visits all three vertices, the two chosen sections have disjoint
intermediate copies and share only the copy carrying the middle vertex.
Inside that copy, the incoming and outgoing length-\(s\) arcs meet only
at the middle vertex, so the two paths can be concatenated.  In either
case, the result is a simple path \(\gamma\) with the prescribed pattern.

We shall use at most one \(g\)-switch, at one endpoint of \(\gamma\).
If that endpoint is \(a_i\), the path uses one contact of the \(i\)-th
copy.  The switch first goes from \(a_i\) to the other contact and then
to the contact already used by \(\gamma\).  The other contact lies on
the unused section of the chain, so the new path is still simple and
has the same \(P_N\)-pattern.

A \(d\)-switch may be made once in any intermediate copy of the section
or sections used by \(\gamma\).  In the \(i\)-th such copy, it replaces
the length-\(g\) arc between the two contacts by the two length-\(s\)
arcs through \(a_i\).  Since this copy is not an endpoint copy,
\(a_i\notin P_N\), so the pattern is unchanged.  Any set of distinct
intermediate copies may be switched simultaneously: their interiors are
disjoint, and every switch keeps the two contacts through which
\(\gamma\) enters and leaves its copy.  Thus one \(g\)-switch adds
\(g\) to the length, while \(r\) \(d\)-switches add \(rd\).

It remains to choose the switches so that the length difference is
exactly \(h_N\).  If \(\ell=3\) and \(N>3\), then \(d=h_N=1\), and
every path \(\gamma\) constructed above contains at least one
intermediate copy, so one \(d\)-switch suffices.  If \(N=3\), then
\(g=h_N=1\), so one \(g\)-switch suffices.
If \(\ell=2\), then \(g=2k\) and \(d=h_N=2\); again one
\(d\)-switch suffices.

Suppose finally that \(\ell=1\).  Then \(g=2k-1\) and \(d=3\).
Thus \(h_N=3\) when \(k\equiv2\pmod3\), and \(h_N=1\) otherwise.
If \(k=3j\), compare a path with one \(g\)-switch to a path with
\(2j\) distinct \(d\)-switches; their added lengths differ by
\[
  2j\cdot3-(2k-1)=6j-(6j-1)=1.
\]
If \(k=3j+1\), the same two choices have length difference
\[
  (2k-1)-2j\cdot3=(6j+1)-6j=1.
\]
If \(k=3j+2\), then both \(g=6j+3\) and \(d=3\) are divisible by
\(3\), and one \(d\)-switch gives a difference of \(3\).

The required \(d\)-switches are always available.  The path \(\gamma\)
either uses a section of length \(s\), which has \(s-1=k\) intermediate
copies, or uses only the section of length \(g\), which has \(g-1\)
intermediate copies.  For \(k=3j\) or \(k=3j+1\), both \(2j\le k\)
and \(2j\le g-1\).  For \(k=3j+2\), only one \(d\)-switch is
required, and both \(k\) and \(g-1\) are positive.
Thus the construction applies to all six patterns.
\end{proof}

\begin{proof}[Proof of Theorem~\ref{thm:ratio-limit}\textup{(ii)}]
We first prove that the polygonal \(N\)-gasket is \(h_N\)-flexible at scale
\(2\).  Fix \(\sigma\in\cS_2\) and
\(\Gamma\in\cC_2(\sigma)\).  Choose a copy \(\Delta\) of
\(F_1^{(N)}\) in \(F_2^{(N)}\) which contains an edge of \(\Gamma\).
The copy \(\Delta\) has one outer vertex on \(\partial F_2^{(N)}\) and
two contacts with its neighbors.  Every endpoint of
\(\Gamma\cap\Delta\) is one of these three vertices.  It cannot have
two components, because that would require four distinct endpoints.
Thus \(\Gamma\cap\Delta\) is a single path.  Self-similarity and a
dihedral relabeling identify \(\Delta\) with \(F_1^{(N)}\) and these
three vertices with \(P_N\).  The path
\(\Gamma\cap\Delta\) consequently has one of the six patterns in
Lemma~\ref{lem:ngasket-three-vertex-switch}.

Replace this path, in turn, by the two paths supplied by that lemma, and
leave \(\Gamma\) unchanged outside \(\Delta\).  Because the two
replacements visit the same vertices of \(P_N\) in the same order,
both resulting graphs are boundary forests with state \(\sigma\).
Their edge counts differ by \(h_N\).  Since \(\sigma\) was arbitrary,
the graph is \(h_N\)-flexible at scale \(2\).

We next show that no smaller step is possible.  If \(\ell=2\), then
\(N=4k+2\).  The cycle \(F_0^{(N)}\) is bipartite, and its two contacts
\(q_{-s}\) and \(q_s\) have the same color because their cyclic distance
\(2s\) is even.  The copies can therefore be colored consistently at
every scale, so each \(F_r^{(N)}\) is bipartite.  In a bipartite graph,
the parity of the length of a path is determined by the colors of its
endpoints.  A boundary state fixes the endpoints of every component and
their pairing, so the edge counts of all forests with that state have
the same parity.  Every flexibility step is therefore even.  Since
\[
  h_N=\gcd(2k,2)=2,
\]
the step constructed above is minimal.

Now suppose that \(\ell=1\) and \(k\equiv2\pmod3\), equivalently
\(N\equiv9\pmod{12}\).  Then
\[
  3\mid s=k+1,
  \qquad
  3\mid g=2k-1.
\]
Decompose a boundary forest into its pieces inside the copies of
\(F_0^{(N)}\).  Each piece is a concatenation of arcs of lengths
\(s,s,g\), and hence has length divisible by \(3\).  Since different
copies have disjoint edge sets, the whole forest also has length
divisible by \(3\).  Thus every flexibility step is divisible by \(3\).
Here \(h_N=\gcd(g,3)=3\), so the constructed step is minimal.

In all remaining cases \(h_N=1\), which is automatically the smallest
positive step.  This proves the stated classification and minimality.
Finally, the ratio limit follows from
Theorem~\ref{thm:ratio-limit}(i).
\end{proof}

\subsection{\texorpdfstring{The Vicsek
counterexample}{The Vicsek counterexample}}
\label{subsec:vicsek}

The Vicsek graph is not \(h\)-flexible for any \(h\ge1\).  At every
scale, all self-avoiding paths
between a fixed pair of
diagonally opposite boundary vertices have the same
length.  By Theorem~\ref{thm:connective}, its connective constant
exists, but its successive ratios have two different subsequential
limits.  Let
\[
  p_0=(0,0),\quad p_1=(1,0),\quad p_2=(1,1),\quad
  p_3=(0,1),\quad p_4=(1/2,1/2),
\]
and define
\[
  \Psi_i(x)=p_i+\frac13(x-p_i),\qquad 0\le i\le4.
\]
The attractor of this iterated function system is the standard Vicsek
nested fractal.  Its essential fixed points are \(p_0,p_1,p_2,p_3\), so
\(F_0\) is the four-cycle \(p_0p_1p_2p_3p_0\), and we take \(O=p_0\).
Since \(\Psi_0\) fixes \(O\), the graphs \(F_m\) are defined
by \eqref{eq:finite-graphs}.  The four corner cells meet the central cell
at four distinct vertices and do not meet one another, so the one-point
intersection condition \eqref{eq:one-point} holds.

For a polynomial \(G(z)\), write \([z^n]G(z)\) for the coefficient of
\(z^n\) in \(G(z)\).

\begin{proposition}
\label{prop:vicsek-recursions}
Let
\[
  H_m(z)=
  \sum_{\substack{w\text{ is a self-avoiding walk in }F_m\\w(0)=O}}
  z^{L(w)}.
\]
The number of \(n\)-step self-avoiding walks in \(F_m\) starting at
\(O\) is \([z^n]H_m(z)\).  Let \(P_m(z)\) be the sum of \(z^{L(w)}\)
over all self-avoiding paths joining a fixed pair of
diagonally opposite boundary vertices, and define
\(Q_m(z)\) in the same way for a fixed adjacent pair.  Dihedral symmetry
makes these polynomials independent of the
chosen pair.  They satisfy
\begin{equation}
  \begin{aligned}
  P_{m+1}&=P_m^3,\\
  Q_{m+1}&=P_m^2Q_m,\\
  H_{m+1}&=H_m+P_m(1+2Q_m+P_m)(H_m-1),
  \end{aligned}
  \label{eq:vicsek-recursions}
\end{equation}
with
\[
  H_0=1+2z+2z^2+2z^3,\qquad
  P_0=2z^2,\qquad Q_0=z+z^3.
\]
\end{proposition}

\begin{figure}[ht]
\centering
\begin{overpic}[width=.25\textwidth]{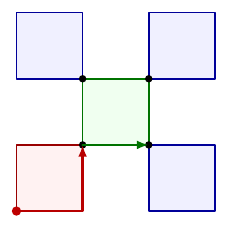}
  \put(22,22){\makebox(0,0){\small\(\Delta_0\)}}
  \put(80,22){\makebox(0,0){\small\(\Delta_1\)}}
  \put(51,58){\makebox(0,0){\small\(\Delta_{\mathrm c}\)}}
  \put(22,80){\makebox(0,0){\small\(\Delta_2\)}}
  \put(80,80){\makebox(0,0){\small\(\Delta_3\)}}
  \put(5.5,5){\makebox(0,0)[r]{\small\color{red!65!black}\(O\)}}
  \put(29,31){\makebox(0,0){\scriptsize\color{red!65!black}\(P_m\)}}
  \put(51,31){\makebox(0,0){\scriptsize\color{green!40!black}\(Q_m\)}}
\end{overpic}
\caption{The five-copy decomposition of the Vicsek graph
\(F_{m+1}\).  Each \(\Delta_j\) is a copy of \(F_m\);
\(\Delta_0\) is the rooted copy and \(\Delta_{\mathrm c}\) is the central
copy.  A walk which leaves \(\Delta_0\) can enter at most one further
corner copy.}
\label{fig:vicsek-decomposition}
\end{figure}
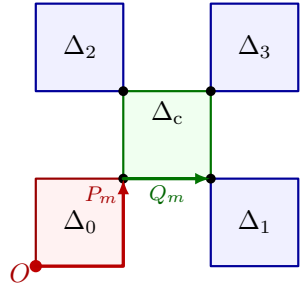
\FloatBarrier

\begin{proof}
The graph \(F_{m+1}\) consists of five copies of \(F_m\): four corner
copies and one central copy.  Each corner copy meets the central copy at
one vertex, and no two corner copies meet.
Figure~\ref{fig:vicsek-decomposition} illustrates this decomposition.

A path joining diagonally opposite boundary vertices crosses three copies
in order: the initial corner, the central copy, and the terminal corner.
Within each copy it again joins diagonally opposite boundary vertices, so
\(P_{m+1}=P_m^3\).  For adjacent boundary vertices, the path joins
diagonally opposite boundary vertices in each of the two corner copies
and adjacent boundary vertices in the central copy.  Hence
\(Q_{m+1}=P_m^2Q_m\).

Now consider a walk counted by \(H_{m+1}\).  If it stays in the corner
copy containing \(O\), its contribution is \(H_m\).  Otherwise it first
crosses that copy to the central copy, contributing \(P_m\).  From the
contact vertex it either makes a nonempty walk inside the central copy
and stops there, or crosses the central copy into one of the other three
corner copies and makes a nonempty terminal walk in that copy.  The three
exits consist of two adjacent crossings and
one diagonally opposite crossing.
Therefore the contribution of walks leaving the rooted corner is
\[
  P_m\bigl((H_m-1)+(2Q_m+P_m)(H_m-1)\bigr).
\]
Adding the walks that stay in the rooted corner gives the third recursion
in \eqref{eq:vicsek-recursions}.  The
initial polynomials follow by listing the self-avoiding paths in
\(F_0\).
\end{proof}

\begin{proof}[Proof of Theorem~\ref{thm:ratio-limit}\textup{(iii)}]
The first two recursions in Proposition~\ref{prop:vicsek-recursions} give
\[
  P_m(z)=2^{3^m}z^{2\cdot3^m},\qquad
  Q_m(z)
  =
  2^{3^m-1}
  \left(z^{2\cdot3^m-1}+z^{2\cdot3^m+1}\right).
\]
Fix \(r\ge1\), choose a path counted by \(P_r\), and let \(\sigma\) be
its boundary state.  Every forest in \(\cC_r(\sigma)\) consists of one
self-avoiding path joining the same pair of
diagonally opposite boundary vertices,
and is therefore counted by \(P_r\).  The formula for \(P_r\) shows that
every such forest has exactly \(2\cdot3^r\) edges.  Thus
\(\cC_r(\sigma)\) contains no two forests whose numbers of edges differ
by a positive integer.  Since \(r\) was arbitrary, the Vicsek graph is
not \(h\)-flexible at any finite scale for any \(h\ge1\).

The third recursion in \eqref{eq:vicsek-recursions}, together with
\(\deg H_0=3\), gives by induction
\[
  \deg H_m=2\cdot3^m+m+1.
\]
Moreover, every term in \(H_{m+1}-H_m\) has
degree at least \(2\cdot3^m+1\).  Hence \([z^n]H_m\) no longer changes
once \(2\cdot3^m\ge n\), and its stable value is \(c_n\).

At level \(1\), \eqref{eq:vicsek-recursions} gives
\[
  H_1
  =
  1+2z+2z^2+6z^3+12z^4+20z^5+24z^6+16z^7+8z^8.
\]
Moreover,
\[
  P_1=8z^6,\qquad Q_1=4z^5+4z^7.
\]
The term \(P_1(2Q_1+P_1)(H_1-1)\) starts in degree \(12\).  Therefore,
in degrees \(7,8,9\), only \(H_1+P_1(H_1-1)\) contributes to \(H_2\).
Using \(P_1=8z^6\), we obtain
\[
  [z^7]H_2=16+8\cdot2=32,\qquad
  [z^8]H_2=8+8\cdot2=24,\qquad
  [z^9]H_2=0+8\cdot6=48.
\]
Now let \(k\ge3\) and \(j\in\{-2,-1,0\}\).  Since
\(\deg H_{k-1}<3^k-2\),
\([z^{3^k+j}]H_{k-1}=0\).  Moreover, every term in
\[
  P_{k-1}(2Q_{k-1}+P_{k-1})(H_{k-1}-1)
\]
has degree at least \(4\cdot3^{k-1}>3^k\).  Thus only
\(P_{k-1}(H_{k-1}-1)\) contributes to the coefficient of
\(z^{3^k+j}\) in \(H_k\), and
\[
  [z^{3^k+j}]H_k
  =
  2^{3^{k-1}}
  [z^{3^{k-1}+j}]H_{k-1}.
\]
Starting with the three coefficients of \(H_2\) above and applying this
identity repeatedly proves \eqref{eq:vicsek-special-counts}.

The first formula in \eqref{eq:vicsek-special-counts} gives
\[
  \mu
  =
  \lim_{k\to\infty}
  c_{3^k-2}^{1/(3^k-2)}
  =\sqrt2.
\]
The other two formulas in \eqref{eq:vicsek-special-counts} give
\[
  \frac{c_{3^k-1}}{c_{3^k-2}}=\frac34,\qquad
  \frac{c_{3^k}}{c_{3^k-1}}=2.
\]
Since these two subsequential limits are different,
\(c_{n+1}/c_n\) does not converge.
\end{proof}

\appendix
\renewcommand{\thesection}{\Alph{section}}
\setcounter{section}{0}
\refstepcounter{section}
\section*{Appendix \thesection\quad Dynamics of the polygonal crossing recursions}
\label{app:ngasket-dynamics}
\addcontentsline{toc}{section}{Appendix A. Dynamics of the polygonal crossing recursions}

The two-variable recursion for the Sierpi\'nski gasket and its critical
behavior were introduced by Hattori, Hattori and
Kusuoka~\cite{HHK1990}.  The same renormalization viewpoint motivates
the analysis below, which combines coordinatewise monotonicity with
compactness of the critical orbit.

\begin{theorem}
\label{thm:ngasket-dynamics}
Let \(N\ge3\) and \(4\nmid N\).  Every entry of
\(\mathbf C_m^{(N)}(\beta)\) tends to zero when
\(\beta>\beta_{\mathrm c,N}\), and every entry tends to infinity when
\(\beta<\beta_{\mathrm c,N}\).  At
\(\beta=\beta_{\mathrm c,N}\), the crossing vector has the following
limits.  Put \(\tau=(\sqrt5-1)/2\).
\begin{enumerate}[label=\textup{(\roman*)}]
\item If \(N=3\), then
\[
  (x_m,y_m)\longrightarrow(\tau,0).
\]
\item If \(N=5\), then
\[
  (\phi_{1,m},\phi_{2,m},\theta_{1,m},\theta_{2,m})
  \longrightarrow
  (\phi_*,0,\theta_*,0),
\]
where
\[
  \theta_*^5+2\theta_*^4=1,
  \qquad
  \phi_*=\frac1{\theta_*(1+\theta_*)}.
\]
\item If \(N\in\{6,9\}\), then
\[
  (X_m,Y_m)\longrightarrow(t_*,t_*),
  \qquad
  t_*=\tau^{1/d_N},
  \qquad
  d_N=\frac N3.
\]
\item If \(N=4k+\ell\notin\{3,5,6,9\}\), where
\(\ell\in\{1,2,3\}\), then
\[
  (X_m,Y_m)\longrightarrow(X_{*,N},Y_{*,N}),
\]
where \(Y_{*,N}\in(0,1)\) is the unique solution of
\[
  Y_{*,N}^{\,N}+2Y_{*,N}^{\,2k+2}=1
\]
and
\[
  X_{*,N}
  =
  \bigl(Y_{*,N}^{\,k}+Y_{*,N}^{\,3k+\ell-2}\bigr)^{-1}.
\]
\end{enumerate}
\end{theorem}

We prove Theorem~\ref{thm:ngasket-dynamics} by treating the generic
two-variable recursion first.  The recursions for \(N=3\) and \(N=5\)
require separate arguments; the cases \(N=6\) and \(N=9\)
then follow from the scalar recursion used in the proof of
Theorem~\ref{thm:ngasket}.

\begin{proposition}
\label{prop:generic-ngasket-phase}
Let \(N=4k+\ell\notin\{3,5,6,9\}\), where
\(\ell\in\{1,2,3\}\).  There is a unique value
\(\beta_{*,N}>0\) such that
\[
  \begin{array}{ll}
  \beta>\beta_{*,N}:&(X_m,Y_m)\longrightarrow(0,0),\\[1mm]
  \beta<\beta_{*,N}:&(X_m,Y_m)\longrightarrow(+\infty,+\infty).
  \end{array}
\]
At \(\beta=\beta_{*,N}\),
\[
  (X_m,Y_m)\longrightarrow(X_{*,N},Y_{*,N}),
\]
where \(Y_{*,N}\in(0,1)\) is the unique solution of
\[
  Y_{*,N}^{\,N}+2Y_{*,N}^{\,2k+2}=1
\]
and
\[
  X_{*,N}
  =
  \bigl(Y_{*,N}^{\,k}+Y_{*,N}^{\,3k+\ell-2}\bigr)^{-1}.
\]
Consequently,
\(\beta_{*,N}
=\beta_{\mathrm c,N}^{\mathrm{cr}}
=\beta_{\mathrm c,N}\).
\end{proposition}

\begin{proof}
Define
\[
  A(y)=y^k+y^{3k+\ell-2},
  \qquad
  B(y)=y^{2k+\ell-3}+y^{2k+1},
\]
so that the recursion can be written as
\[
  \mathcal R(x,y)=\bigl(x^2A(y),x^2B(y)\bigr).
\]
By Proposition~\ref{prop:ngasket-recursions},
\((X_{m+1},Y_{m+1})=\mathcal R(X_m,Y_m)\).
Set
\[
  \alpha=k+1,
  \qquad
  \delta=k+\ell-3.
\]
In the present case \(k\ge1\), \(\delta\ge1\), and
\(\alpha>\delta\).  For every \(x,y>0\),
\begin{equation}
  \frac{\mathcal R_2(x,y)}{\mathcal R_1(x,y)}
  =
  f(y)
  :=
  \frac{y^\delta+y^\alpha}{1+y^{\alpha+\delta}}.
  \label{eq:generic-projective}
\end{equation}
Since
\[
  1-f(y)
  =
  \frac{(1-y^\delta)(1-y^\alpha)}
       {1+y^{\alpha+\delta}},
\]
we have \(0<f(y)\le1\) for every \(y>0\), with equality only at
\(y=1\).
The ratio in \eqref{eq:generic-projective} therefore gives
\begin{equation}
  Y_m\le X_m,
  \qquad m\ge1.
  \label{eq:generic-coordinate-order}
\end{equation}
Moreover,
\[
  f'(y)
  =
  \frac{
    \delta y^{\delta-1}(1-y^{2\alpha})
    +\alpha y^{\alpha-1}(1-y^{2\delta})
  }{(1+y^{\alpha+\delta})^2}
  >0,
  \qquad 0<y<1.
\]
Thus \(f\) is strictly increasing on \((0,1)\).

We first determine the positive fixed points of \(\mathcal R\).  The
fixed-point equations are
\[
  x=A(y)^{-1},
  \qquad
  yA(y)^2=B(y).
\]
Substituting
\[
  A(y)=y^k(1+y^{\alpha+\delta}),
  \qquad
  B(y)=y^{k+\delta}(1+y^{\alpha-\delta}),
\]
into the second equation gives
\begin{equation}
  y^N+2y^{2k+2}=1.
  \label{eq:generic-positive-fixed-point}
\end{equation}
The left-hand side of \eqref{eq:generic-positive-fixed-point} is
strictly increasing from \(0\) to \(3\) on \((0,1)\).  Hence
\((X_{*,N},Y_{*,N})\) is the unique positive fixed point.

We next locate the parameter at which the orbit changes from
divergence to decay.  Let
\[
  \mathcal E
  =
  \left\{
    \beta\ge0:
    Y_m(\beta)>1\ \text{for some }m\ge1
  \right\}.
\]
If \(Y_m>1\), then \eqref{eq:generic-coordinate-order} gives
\(X_m\ge Y_m>1\).  The recursion then gives
\(X_{m+1}>X_m^2\) and
\(Y_{m+1}>Y_m^{\,2k+\ell-1}\), so both coordinates tend to
infinity.

The map \(\mathcal R\) is strictly
increasing in each variable, while both initial values are strictly
decreasing functions of \(\beta\).  Consequently, if
\(\beta\in\mathcal E\), then every smaller nonnegative parameter also
belongs to \(\mathcal E\).  Thus \(\mathcal E\) is downward closed in
\([0,\infty)\).
At \(\beta=0\), the orbit enters the region \(X_m,Y_m>1\), so
\(0\in\mathcal E\).  Put \(r=k+2>1\).  For
\(z\in[0,\infty)^2\) with \(\|z\|_\infty\le1\),
\[
  \|\mathcal R(z)\|_\infty
  \le 2\|z\|_\infty^r.
\]
Hence a sufficiently small square about the origin is invariant and
every orbit starting there tends to the origin.  Since the initial
point tends to the origin as \(\beta\to\infty\), sufficiently large
\(\beta\) does not belong to \(\mathcal E\).
Finally, every finite iterate depends continuously on \(\beta\), and
hence \(\mathcal E\) is open.  There is therefore a unique
\(\beta_{*,N}>0\) such that
\begin{equation}
  \mathcal E=[0,\beta_{*,N}).
  \label{eq:generic-escape-interval}
\end{equation}

We now study the orbit at the endpoint \(\beta_{*,N}\).  By
\eqref{eq:generic-escape-interval}, it cannot have \(Y_m>1\).
It cannot have \(Y_m=1\) either: in that case
\eqref{eq:generic-coordinate-order} gives \(X_m\ge1\), while
\eqref{eq:ngasket-recursion} gives
\(Y_{m+1}=2X_m^2\ge2\).
Hence
\[
  0<Y_m<1,\qquad m\ge1.
\]

To bound the other coordinate, put
\[
  a_m=\log X_m,
  \qquad
  b_m=-\log Y_m,
  \qquad
  c=2k+\ell-3.
\]
For \(m\ge1\), the recursion gives
\[
  \begin{aligned}
  a_{m+1}&=2a_m-kb_m+\varepsilon_m,\\
  b_{m+1}&=-2a_m+cb_m-\eta_m,
  \end{aligned}
\]
where the error terms are
\[
  \varepsilon_m
  =
  \log\bigl(1+Y_m^{\,2k+\ell-2}\bigr),
  \qquad
  \eta_m
  =
  \log\bigl(1+Y_m^{\,4-\ell}\bigr).
\]
Since \(0<Y_m<1\) and both exponents are positive,
\[
  0\le\varepsilon_m,\eta_m\le\log2.
\]

The polynomial
\[
  P(t)=t^2-(c+2)t+2\delta
\]
satisfies \(P(0)>0\) and \(P(1)=\ell-4<0\).  Let
\(\lambda\in(0,1)\) be its smaller zero.  Then
\((2,2-\lambda)\) is a positive left eigenvector of
\[
  \begin{pmatrix}2&-k\\-2&c\end{pmatrix}
\]
with eigenvalue \(\lambda\).  Hence
\[
  W_m=2a_m+(2-\lambda)b_m
\]
satisfies
\[
  \begin{aligned}
  W_{m+1}
  &=
  \lambda W_m+2\varepsilon_m-(2-\lambda)\eta_m\\
  &\le
  \lambda W_m+2\log2.
  \end{aligned}
\]
Iterating this scalar inequality shows that \(W_m\) is bounded above.
Since \(b_m\ge0\) and \(2-\lambda>0\), we have \(2a_m\le W_m\).
Thus \(a_m\), and therefore \(X_m\), is bounded above.

The critical orbit cannot enter the small invariant square
constructed above.  Otherwise, continuity of a finite iterate would
put an orbit with \(\beta<\beta_{*,N}\) in the same square,
contradicting \eqref{eq:generic-escape-interval}.
Suppose that \(Y_m\to0\) along a subsequence.  Since \(X_m\) is
bounded, both coordinates of the next iterate tend to zero, which is
impossible by the preceding paragraph.  If \(X_m\to0\) along a
subsequence, then \(Y_m\le X_m\) by
\eqref{eq:generic-coordinate-order}, and the same contradiction
follows.
Finally, suppose that \(Y_{m_j}\to1\).  Then
\(X_{m_j}\ge Y_{m_j}\), and therefore
\[
  Y_{m_j+1}
  =
  X_{m_j}^2B(Y_{m_j})
  \ge
  Y_{m_j}^2B(Y_{m_j})
  \longrightarrow 2,
\]
contrary to \(Y_{m_j+1}<1\).
The critical orbit consequently stays in a compact subset of
\((0,\infty)\times(0,1)\).

Let \(L\) be its \emph{omega-limit set}, that is, the set of limits of
convergent subsequences of the critical orbit.  Since the orbit remains
in a compact set and \(\mathcal R\) is continuous,
\(\mathcal R(L)\subseteq L\).  Conversely, if
\(z_{m_j}\to q\in L\), compactness gives a convergent subsequence of
\((z_{m_j-1})_j\); its limit \(p\in L\) satisfies
\(\mathcal R(p)=q\).  Hence \(\mathcal R(L)=L\).

Every monomial in \(\mathcal R\) has total
degree at least \(r=k+2>1\), and hence
\begin{equation}
  \mathcal R(tz)\ge t^r\mathcal R(z),
  \qquad z\in(0,\infty)^2,\quad t\ge1.
  \label{eq:generic-scaling}
\end{equation}
No two distinct points of \(L\) are comparable in the coordinatewise
order.  Indeed, if distinct \(p,q\in L\) satisfy
\(p\le q\), strict monotonicity and \(\mathcal R(L)=L\) allow us,
after replacing them by their images, to choose \(t>1\) such that
\(q\ge tp\).  Monotonicity and \eqref{eq:generic-scaling} then give
by induction
\[
  \mathcal R^m(q)
  \ge
  t^{r^m}\mathcal R^m(p),
  \qquad m\ge0.
\]
This is impossible because the compact set
\(L\subset(0,\infty)\times(0,1)\) is bounded away from the coordinate
axes.

The function \(f\) in \eqref{eq:generic-projective} is strictly
increasing on \((0,1)\).  It follows that \(\mathcal R\) is injective
on \((0,\infty)\times(0,1)\): equality of two images first gives
equality of their second input coordinates from the ratio in
\eqref{eq:generic-projective}, and then equality of their first input
coordinates.  Order the points of \(L\) by their first coordinate.
If \(p_1<q_1\), the preceding incomparability gives \(p_2>q_2\).
If one also had
\(\mathcal R_1(p)\ge\mathcal R_1(q)\), then the strict increase of
\(f\), together with \eqref{eq:generic-projective}, would imply
\(\mathcal R_2(p)>\mathcal R_2(q)\).  The two images would then be
comparable, a contradiction.  Hence
\(\mathcal R_1(p)<\mathcal R_1(q)\): the restriction of
\(\mathcal R\) to \(L\) preserves the order of the first coordinates.

Since \(\mathcal R(L)=L\), the restriction of \(\mathcal R\) to \(L\)
is an order-preserving bijection.  It therefore fixes the points of
\(L\) with the smallest and largest first coordinates.  Both are
positive fixed points of \(\mathcal R\), so the uniqueness proved in
\eqref{eq:generic-positive-fixed-point} forces them to coincide.
Thus
\[
  L=\{(X_{*,N},Y_{*,N})\},
\]
and the critical orbit converges to this point.

If \(\beta>\beta_{*,N}\), choose \(0<t<1\) such that the initial point
at \(\beta\) is bounded above by \(t\) times the critical initial
point.  Since every monomial has degree at least \(r\), the
counterpart of \eqref{eq:generic-scaling} for \(0<t<1\), followed by
induction, gives
\[
  (X_m(\beta),Y_m(\beta))
  \le
  t^{r^m}
  (X_m(\beta_{*,N}),Y_m(\beta_{*,N})),
\]
so the orbit tends to the origin.  If \(\beta<\beta_{*,N}\),
\eqref{eq:generic-escape-interval} gives
\(Y_m>1\) for some \(m\).  Then
\eqref{eq:generic-coordinate-order} and
\eqref{eq:ngasket-recursion} imply that both coordinates tend to
infinity.
Since
\(\mathbf C_m^{(N)}(\beta)=(X_m(\beta),Y_m(\beta))\) in the present
case, we have proved
\[
  \left\{
    \beta\ge0:
    \sup_{m\ge0}\|\mathbf C_m^{(N)}(\beta)\|_\infty<\infty
  \right\}
  =
  [\beta_{*,N},\infty).
\]
The definition \eqref{eq:ngasket-critical-beta} and
Proposition~\ref{prop:ngasket-completeness} therefore give
\(\beta_{*,N}
=\beta_{\mathrm c,N}^{\mathrm{cr}}
=\beta_{\mathrm c,N}\).
\end{proof}

For completeness, we also treat the case of the Sierpi\'nski
gasket studied in~\cite{HHK1990}.

\begin{proposition}
\label{prop:sg-phase}
For \(N=3\), the recursion \eqref{eq:sg-crossing-recursion} has a unique
critical value \(\beta_{*,3}>0\).  Its orbit tends to \((0,0)\) for
\(\beta>\beta_{*,3}\), both coordinates tend to infinity for
\(\beta<\beta_{*,3}\), and
\[
  (x_m,y_m)\longrightarrow(\tau,0),
  \qquad
  \tau=\frac{\sqrt5-1}{2},
\]
at \(\beta=\beta_{*,3}\).  Consequently,
\(\beta_{*,3}
=\beta_{\mathrm c,3}^{\mathrm{cr}}
=\beta_{\mathrm c,3}\).
\end{proposition}

\begin{proof}
Put \(r_m=y_m/x_m\).  The recursion gives
\begin{equation}
  r_{m+1}
  =
  r_m
  \frac{x_m(1+2r_m)}
       {(1+r_m)^2+x_m(1+2r_m)}.
  \label{eq:sg-ratio}
\end{equation}
Thus \(0<r_{m+1}<r_m\).  If the orbit is bounded, then \(x_m\) is
bounded and \(0<r_m\le r_0\le1\).  The fraction in
\eqref{eq:sg-ratio} is consequently bounded above by a constant
strictly smaller than one.  Hence \(r_m\to0\).

If \(x_m>1\), then
\[
  x_{m+1}\ge x_m^2(1+x_m)>x_m,
\]
and \(x_m\to\infty\).  Since
\[
  \frac{y_{m+1}}{y_m}=x_m(x_m+2y_m),
\]
the second coordinate then tends to infinity as well.

Define
\[
  \mathcal E_3
  =
  \{\beta\ge0:x_m(\beta)>1\text{ for some }m\ge0\}.
\]
The recursion is strictly increasing in both variables, while
\((x_0,y_0)=(e^{-\beta},e^{-2\beta})\) is strictly decreasing in
\(\beta\).  Consequently, if \(\beta\in\mathcal E_3\), then every
smaller nonnegative parameter also belongs to \(\mathcal E_3\); thus
\(\mathcal E_3\) is downward closed in \([0,\infty)\).
At \(\beta=0\), one has \(x_0=y_0=1\) and \(x_1=7>1\), so
\(0\in\mathcal E_3\).  For all sufficiently large \(\beta\), the
initial point lies in a small invariant neighborhood of the origin
whose orbits tend to the origin, so such a parameter does not belong
to \(\mathcal E_3\).
Since each finite iterate depends continuously
on \(\beta\), the set \(\mathcal E_3\) is open.  It follows that
\(\mathcal E_3=[0,\beta_{*,3})\) for a unique
\(\beta_{*,3}>0\).

At \(\beta=\beta_{*,3}\), one has \(x_m<1\) for
every \(m\), since equality would imply \(x_{m+1}>1\).  Also
\(y_m=r_mx_m\le x_m\), so the critical orbit is bounded.
It cannot converge to the origin.  Otherwise, a finite critical iterate would
lie in a small invariant neighborhood of the origin; continuity
would then put an orbit with \(\beta<\beta_{*,3}\) in the same
neighborhood, contrary to the definition of \(\mathcal E_3\).

Since \(r_m\to0\), the first recursion may be written as
\[
  x_{m+1}
  =
  x_m^2
  \bigl((1+r_m)^2+x_m(1+2r_m)\bigr).
\]
The limiting map is \(g(x)=x^2(1+x)\).  Its only nonnegative
fixed points are \(0\) and \(\tau\); moreover,
\[
  g(x)<x\quad(0<x<\tau),
  \qquad
  g(x)>x\quad(x>\tau).
\]
The critical orbit cannot have \(x_m>\tau\).  Indeed, the recursion
gives \(x_{m+1}>g(x_m)>x_m\) in that case.  The subsequent
values would then increase and, since \(x_m<1\), converge to a limit
in \((\tau,1]\).  Passing to the limit and using \(r_m\to0\) would
give \(x=g(x)\), which is impossible in this interval.

Now fix \(\varepsilon\in(0,\tau)\).  At \(r=0\), the quotient
\(x_{m+1}/x_m\) is \(x(1+x)\).  Since
\(\tau(1+\tau)=1\),
\[
  x(1+x)
  \le
  (\tau-\varepsilon)(1+\tau-\varepsilon)
  <1,
  \qquad 0<x\le\tau-\varepsilon.
\]
By continuity, there are \(q<1\) and \(r_0>0\) such that
\[
  x\bigl((1+r)^2+x(1+2r)\bigr)\le q
\]
whenever \(0<x\le\tau-\varepsilon\) and \(0\le r\le r_0\).
Since \(r_m\to0\), this applies to every sufficiently late iterate.
If such an iterate satisfied \(x_m\le\tau-\varepsilon\), then
\(x_{m+1}\le qx_m\le\tau-\varepsilon\), and induction would give
\(x_{m+j}\le q^j x_m\) for every \(j\ge0\).  The full orbit would
then converge to the origin, contradicting the preceding paragraph.
Therefore
\[
  \tau-\varepsilon<x_m\le\tau
\]
for all sufficiently large \(m\).  Since \(\varepsilon\) is arbitrary,
\(x_m\to\tau\), and \(y_m=r_mx_m\to0\).

The scaling comparison used in
\eqref{eq:generic-scaling}, now with minimum degree \(2\), shows that
every initial point corresponding to \(\beta>\beta_{*,3}\) tends to
zero.  Every initial point corresponding to
\(\beta<\beta_{*,3}\) belongs to \(\mathcal E_3\).
Since
\(\mathbf C_m^{(3)}(\beta)=(x_m(\beta),y_m(\beta))\), we have proved
\[
  \left\{
    \beta\ge0:
    \sup_{m\ge0}\|\mathbf C_m^{(3)}(\beta)\|_\infty<\infty
  \right\}
  =
  [\beta_{*,3},\infty).
\]
The definition \eqref{eq:ngasket-critical-beta} and
Proposition~\ref{prop:ngasket-completeness} therefore give
\(\beta_{*,3}
=\beta_{\mathrm c,3}^{\mathrm{cr}}
=\beta_{\mathrm c,3}\), completing the proof.
\end{proof}

The pentagasket recursion differs from those in
Propositions~\ref{prop:generic-ngasket-phase} and~\ref{prop:sg-phase}:
it has four variables, and the coordinate quotient of its limiting
planar map is not monotone on the whole positive half-line.  We first
isolate the planar argument needed to identify the critical limit.

\begin{lemma}
\label{lem:pentagasket-reduced}
Consider
\[
  \mathcal G(x,y)
  =
  \bigl(y^2(1+x^3),y^2x(1+x)\bigr),
  \qquad x,y>0.
\]
Its unique positive fixed point is \((\theta_*,\phi_*)\), where
\[
  \theta_*^5+2\theta_*^4=1,
  \qquad
  \phi_*=\frac1{\theta_*(1+\theta_*)}.
\]
If a nonempty compact set \(L\subset(0,\infty)^2\) satisfies
\(\mathcal G(L)=L\), then
\[
  L=\{(\theta_*,\phi_*)\}.
\]
\end{lemma}

\begin{proof}
At a positive fixed point, the second equation gives
\(y=[x(1+x)]^{-1}\), and the first then becomes
\(x^5+2x^4=1\).  Its left-hand side is strictly increasing, so the
fixed point is unique and \(\theta_*\in(1/2,1)\).

We next describe the set separating decay from escape.  Let
\[
  \mathcal B_0
  =
  \{z\in(0,\infty)^2:\mathcal G^m(z)\to(0,0)\}.
\]
For sufficiently small \(\delta>0\), the square
\((0,\delta)^2\) is invariant and every orbit in it tends to the
origin, since
\(\|\mathcal G(z)\|_\infty\le2\|z\|_\infty^2\) there.
Thus \(\mathcal B_0\) is open and downward closed in the
coordinatewise order.  For every fixed \(x>0\), the point \((x,y)\)
belongs to \(\mathcal B_0\) when \(y\) is small enough, whereas it
escapes to infinity when \(y\) is large enough.  Hence
\[
  p(x)=\sup\{y>0:(x,y)\in\mathcal B_0\}
\]
is finite and positive, and \((x,p(x))\in\partial\mathcal B_0\).
Here large \(y\) gives escape because, once both coordinates exceed
one, their minimum grows at least by its fourth power.  The escaping
set is therefore open.

We need one elementary bound for an orbit
\((x_n,y_n)=\mathcal G^n(z)\) starting on
\(\partial\mathcal B_0\).  Such an orbit neither tends to the origin
nor escapes.  Since
\[
  \frac{y_{n+1}}{x_{n+1}}
  =
  \varphi(x_n),
  \qquad
  \varphi(x)=\frac{x(1+x)}{1+x^3}
             =\frac{x}{x^2-x+1}
  \le1,
\]
we have \(y_n\le x_n\) for \(n\ge1\).  Put \(q_n=x_ny_n\).
If \(x_n>1\) and \(q_n>1\), then
\[
  x_{n+1}=q_n^2(x_n+x_n^{-2}),
  \qquad
  y_{n+1}=q_n^2(1+x_n^{-1}),
\]
and both coordinates increase until the orbit escapes.  Thus
\(x_n>1\) implies \(q_n\le1\).  If also \(x_{n+1}>1\), applying the
same implication at time \(n+1\) gives
\[
  1\ge x_{n+1}y_{n+1}\ge q_n^4x_n,
  \qquad
  x_{n+1}\le x_n^{1/2}+x_n^{-5/2}.
\]
If \(x_n\le1\), then \(y_n\le x_n\), and hence
\[
  x_{n+1}=y_n^2(1+x_n^3)
  \le x_n^2(1+x_n^3)\le2.
\]
Together with the case \(x_{n+1}\le1\), this estimate shows
that \((x_n)\), and hence \((y_n)\), is bounded above.  Neither
coordinate can approach zero along a subsequence: for \(x_n\) this
would put a later iterate in \((0,\delta)^2\), while for \(y_n\) it
follows from \(y_n=x_n\varphi(x_{n-1})\).  Every boundary orbit
therefore has a compact tail in \((0,\infty)^2\).

The scaling inequality
\[
  \mathcal G(tz)\ge t^2\mathcal G(z),
  \qquad t\ge1,
\]
now shows that no two distinct boundary points are comparable.  In
fact, strict monotonicity would give boundary points \(u<v\) and
\(t>1\) such that \(\mathcal G(v)\ge t\mathcal G(u)\); iteration
would then give
\(\mathcal G^{m+1}(v)\ge t^{2^m}\mathcal G^{m+1}(u)\), contradicting
the compact-tail bound.  It follows that
\[
  \partial\mathcal B_0
  =
  \{(x,p(x)):x>0\},
\]
where \(p\) is strictly decreasing.  It is also continuous:
monotonicity makes \(p\) nonincreasing, while any jump would produce
two comparable boundary points on the same vertical line.  Moreover,
if \(y>p(x)\), comparison with the boundary orbit through
\((x,p(x))\) and the same scaling inequality show that \((x,y)\)
escapes.  Thus every point lies either below this curve, on it, or
above it, according as its orbit tends to zero, stays on the curve,
or escapes.  In particular,
\(\mathcal G(\partial\mathcal B_0)\subseteq\partial\mathcal B_0\):
an image below the curve would put the original point in
\(\mathcal B_0\), whereas an image above it would make the original
point escape.

Put \(h(x)=p(x)/x\), and let \(T(x)\) be the first coordinate of
\(\mathcal G(x,p(x))\).  The image again lies on the boundary curve,
and hence
\[
  h(T(x))=\varphi(x).
\]
Since \(h\) is continuous and strictly decreasing, \(T\) is
continuous.  Since
\[
  \varphi'(x)=\frac{1-x^2}{(x^2-x+1)^2},
  \qquad
  \varphi(x)=\varphi(1/x),
\]
the map \(T\) is decreasing on \((0,1)\), increasing on
\((1,\infty)\), and satisfies \(T(x)=T(1/x)\).

To locate its fixed point, set
\[
  Q(x)=\bigl(x,x\varphi(x)\bigr).
\]
A direct calculation gives
\[
  \mathcal G(Q(x))=R(x)Q(x),
  \qquad
  R(x)=\frac{x^3(1+x)^2}{1+x^3}.
\]
The function \(R\) is strictly increasing because
\[
  \frac{R'(x)}{R(x)}
  =
  \frac{3}{x(1+x^3)}+\frac2{1+x}>0,
\]
and \(R(x)=1\) precisely when \(x=\theta_*\).  It follows by
monotone iteration that the orbit of \(Q(x)\) tends to the origin for
\(x<\theta_*\) and escapes for \(x>\theta_*\).  Comparing \(Q(x)\)
with the boundary curve gives
\[
  T(x)>x\quad(x<\theta_*),
  \qquad
  T(x)<x\quad(x>\theta_*).
\]

The interval \([1/2,1]\) attracts every orbit of \(T\).  Indeed,
\[
  \mathcal G(1/2,1/2)=(9/32,3/16)<(1/2,1/2),
\]
and its orbit decreases monotonically to the origin, so
\(p(1/2)>1/2\).  Also
\(\mathcal G(1,2/3)=(8/9,8/9)\), whose next iterate has both
coordinates larger than one, so \(p(1)<2/3\).  Therefore
\[
  h(1)<\frac23\le\varphi(x)\le1<h(1/2),
  \qquad \frac12\le x\le1,
\]
and the decrease of \(h\) gives
\(T([1/2,1])\subset(1/2,1)\).  Below \(1/2\), the iterates increase
while they remain below \(1/2\); they cannot remain there forever,
since their limit would be a fixed point below \(1/2\).  If \(x>1\),
then \(T(x)=T(1/x)>1/2\): this follows from the invariance of
\([1/2,1]\) when \(1/x\ge1/2\), and from the decrease of \(T\) on
\((0,1)\) otherwise.  The iterates decrease while they remain above
one and must therefore enter \([1/2,1]\), since \(T\) has no fixed
point above one.

Finally, \(\mathcal G\) has no nontrivial positive two-cycle.  If
\((x,y)\) and \((u,v)\) formed one, the coordinate quotients and the
first-coordinate equations would give
\[
  \frac vu=\varphi(x),
  \qquad
  \frac yx=\varphi(u),
\]
and then
\[
  J(x)=J(u),
  \qquad
  J(x)=x(x^2-x+1)^2(1+x^3).
\]
But
\[
  \frac{J'(x)}{J(x)}
  =
  \frac{5x^2-3x+1}{x(x^2-x+1)}
  +\frac{3x^2}{1+x^3}
  >0,
\]
so \(x=u\) and then \(y=v\).  Since \(T\) is decreasing on
\([1/2,1]\), its even and odd iterates converge to a possible
two-cycle.  The preceding calculation excludes a nontrivial one, and
therefore every orbit on the boundary curve converges to
\((\theta_*,\phi_*)\).

Now let \(L\subset(0,\infty)^2\) be nonempty and compact, with
\(\mathcal G(L)=L\).  Compactness excludes both decay to the origin
and escape, so \(L\) lies on the boundary curve.  The eventual-entry
property and compactness give a common entry time for all first
coordinates: the sets of points entering the forward-invariant open
interval \((1/2,1)\) by time \(n\) form an increasing open cover of
the compact projection of \(L\).  Invariance of \(L\) then shows that
the first coordinates already lie in \((1/2,1)\).  If \(I\) is their
compact set of first
coordinates, then the first-coordinate action of
\(\mathcal G(L)=L\) gives \(T(I)=I\).  Since \(T\) is decreasing,
\[
  T(\min I)=\max I,
  \qquad
  T(\max I)=\min I.
\]
The absence of a nontrivial two-cycle gives
\(\min I=\max I=\theta_*\), and the conclusion follows.
\end{proof}

\begin{proposition}
\label{prop:pentagasket-phase}
For \(N=5\), the recursion \eqref{eq:pentagasket-recursion} has a unique
critical value \(\beta_{*,5}>0\).  All four variables tend to zero for
\(\beta>\beta_{*,5}\), and all four tend to infinity for
\(\beta<\beta_{*,5}\).  At the critical value,
\[
  (\phi_{1,m},\phi_{2,m},\theta_{1,m},\theta_{2,m})
  \longrightarrow
  (\phi_*,0,\theta_*,0),
\]
where \((\theta_*,\phi_*)\) is the fixed point in
Lemma~\ref{lem:pentagasket-reduced}.  Consequently,
\(\beta_{*,5}
=\beta_{\mathrm c,5}^{\mathrm{cr}}
=\beta_{\mathrm c,5}\).
\end{proposition}

\begin{proof}
Put
\[
  a_m=\phi_{1,m},
  \quad b_m=\phi_{2,m},
  \quad \xi_m=\theta_{1,m},
  \quad \widehat\xi_m=\theta_{2,m},
\]
and introduce the two ratios
\[
  r_m=\frac{b_m}{a_m},
  \qquad
  s_m=\frac{\widehat\xi_m}{\xi_m}.
\]
The recursion gives
\begin{equation}
  \begin{aligned}
  r_{m+1}
    &=s_m\frac{\xi_m}{1+\xi_m},\\
  s_{m+1}
    &=s_m\frac{H_m}{1+H_m},\\
  H_m
    &=\xi_m^3
      \frac{(1+s_m)^2(1+2r_m)}
           {(1+r_m)^2}.
  \end{aligned}
  \label{eq:pentagasket-ratios}
\end{equation}
Since \(r_0=e^{-\beta}\le1\) and \(s_0=e^{-3\beta}\le1\),
\eqref{eq:pentagasket-ratios} implies
\[
  0<r_m\le1,
  \qquad
  0<s_m\le1.
\]
The main pair satisfies
\begin{equation}
  \begin{aligned}
  \xi_{m+1}
    &=
    a_m^2
    \left[
      (1+r_m)^2+
      \xi_m^3(1+s_m)^2(1+2r_m)
    \right],\\
  a_{m+1}
    &=
    a_m^2\xi_m(1+r_m)^2(1+s_m)(1+\xi_m).
  \end{aligned}
  \label{eq:pentagasket-dominant-pair}
\end{equation}
When \(r_m=s_m=0\), the right-hand side is
\(\mathcal G(\xi_m,a_m)\), with \(\mathcal G\) as in
Lemma~\ref{lem:pentagasket-reduced}.

We first separate escaping from bounded orbits.  If
\(a_m>1\) and \(\xi_m>1\), then
\eqref{eq:pentagasket-dominant-pair} keeps both variables above one
and makes them tend to infinity.  In this case
\(\xi_{m+1}\ge\xi_m^3\) and
\[
  H_m\ge\frac14\xi_m^3,
  \qquad
  \frac{s_{m+1}}{s_m}=1-\frac1{1+H_m}.
\]
Thus \(\sum_m(1+H_m)^{-1}<\infty\), so \(s_m\), and then \(r_m\),
converges to a positive limit.  Consequently,
\(b_m=r_ma_m\) and
\(\widehat\xi_m=s_m\xi_m\) also tend to infinity.  We call such an
orbit \emph{escaping}.

Every orbit from the prescribed initial conditions that does not escape
is uniformly bounded.  Indeed, it must satisfy
\[
  a_m\xi_m\le1,
  \qquad
  a_m^2\xi_m\le1.
\]
To see this, if \(\xi_m>1\) and \(a_m\xi_m>1\), or if
\(\xi_m\le1<a_m\) and \(a_m^2\xi_m>1\), the next main pair is already
in \((1,\infty)^2\).  The remaining cases give the two inequalities
directly.  Hence the second equation in
\eqref{eq:pentagasket-dominant-pair} yields
\[
  a_{m+1}
  =
  (1+r_m)^2(1+s_m)
  \bigl(a_m^2\xi_m+(a_m\xi_m)^2\bigr)
  \le16.
\]
Since \(a_0=e^{-2\beta}\le1\), it follows that \(a_m\le16\) for
every \(m\).  If \(\xi_m\le1\), the first equation in
\eqref{eq:pentagasket-dominant-pair} gives
\[
  \xi_{m+1}
  \le16^2(4+12)=4096.
\]
If \(\xi_m>1\) and
\(\xi_{m+1}>1\), put \(q_m=a_m\xi_m\).  Because the orbit does not
escape, the inequalities at time \(m+1\) give
\[
  1\ge a_{m+1}\xi_{m+1}\ge q_m^4\xi_m,
\]
and therefore
\[
  \xi_{m+1}
  \le
  4\xi_m^{-5/2}+12\xi_m^{1/2}
  \le16\xi_m^{1/2}.
\]
In the last case, \(\xi_m\le4096\) implies
\(\xi_{m+1}\le16\sqrt{4096}=1024\).  Together with the cases
\(\xi_m\le1\) and \(\xi_{m+1}\le1\), induction from
\(\xi_0=e^{-\beta}\le1\) therefore gives
\[
  a_m\le16,
  \qquad
  \xi_m\le4096,
\]
and \(r_m,s_m\le1\) then bound all four variables.

Write \(t=e^{-\beta}\), so the initial point is
\[
  (a_0,b_0,\xi_0,\widehat\xi_0)=(t^2,t^3,t,t^4).
\]
The recursion is strictly increasing in each variable, and the initial
point is strictly increasing in \(t\).  Let \(\mathcal E_5\) be the set of
\(t\in(0,1]\) for which the orbit escapes.  The preceding argument
shows that escape occurs exactly when
\(a_m>1\) and \(\xi_m>1\) for some \(m\).  Thus
\(\mathcal E_5\) is open and upward closed.  It contains \(1\),
because \(a_1=\xi_1=16\) when \(t=1\), and it excludes all
sufficiently small \(t\), whose initial points lie in an invariant
neighborhood of the origin.  Therefore
\[
  \mathcal E_5=(t_{*,5},1]
\]
for a unique \(t_{*,5}\in(0,1)\).  Set
\(\beta_{*,5}=-\log t_{*,5}\).

The critical orbit does not escape and is therefore bounded.  It cannot tend
to the origin, since continuity of a finite iterate would then put an
orbit with \(t>t_{*,5}\) in the same invariant neighborhood, contrary
to the definition of \(t_{*,5}\).  Since \(\xi_m\) is bounded,
\(H_m\) in \eqref{eq:pentagasket-ratios} is bounded.  Consequently,
if \(H_*=\sup_m H_m\) and \(q=H_*/(1+H_*)<1\), then
\[
  s_m\le s_0q^m\le q^m,
  \qquad
  r_{m+1}\le s_m.
\]
Thus the two ratios decay geometrically.  By
\eqref{eq:pentagasket-dominant-pair},
\[
  (\xi_{m+1},a_{m+1})
  =
  \mathcal G(\xi_m,a_m)+e_m,
  \qquad
  \lVert e_m\rVert\le C(r_m+s_m)
\]
for some \(C<\infty\), and hence \(\sum_m\lVert e_m\rVert<\infty\).

The critical main pair stays away from both coordinate axes.  If
\(a_{m_j}\to0\), boundedness of \(\xi_m,r_m,s_m\) in
\eqref{eq:pentagasket-dominant-pair} gives
\((\xi_{m_j+1},a_{m_j+1})\to(0,0)\).  If
\(\xi_{m_j}\to0\), the second equation first gives
\(a_{m_j+1}\to0\), while \(\xi_{m_j+1}\) remains bounded; applying
the preceding argument at time \(m_j+1\) gives
\((\xi_{m_j+2},a_{m_j+2})\to(0,0)\).  Since
\(b_m\le a_m\) and \(\widehat\xi_m\le\xi_m\), either alternative
would put a full iterate in the invariant neighborhood of the
origin, which is impossible at \(t_{*,5}\).

Let \(L\) be the omega-limit set of \((\xi_m,a_m)\).  It is a
nonempty compact subset of \((0,\infty)^2\).  Since \(e_m\to0\),
limits of shifted subsequences give \(\mathcal G(L)\subseteq L\).
Conversely, taking a convergent subsequence of the corresponding
predecessors gives \(\mathcal G(L)\supseteq L\).  Hence
\(\mathcal G(L)=L\), and Lemma~\ref{lem:pentagasket-reduced} gives
\[
  L=\{(\theta_*,\phi_*)\}.
\]
Therefore
\[
  (\xi_m,a_m)\longrightarrow(\theta_*,\phi_*),
\]
while \(b_m=r_ma_m\to0\) and
\(\widehat\xi_m=s_m\xi_m\to0\).  This proves the critical limit.

If \(t<t_{*,5}\), choose \(\lambda\in(0,1)\) so that its initial point
is at most \(\lambda\) times the critical initial point.  Every
monomial in \eqref{eq:pentagasket-recursion} has degree at least two,
so induction bounds the corresponding orbit by
\(\lambda^{2^m}\) times the bounded critical orbit; it therefore
tends to zero.  If \(t>t_{*,5}\), all four variables tend to infinity
by the definition of \(\mathcal E_5\).  Since
\(t=e^{-\beta}\),
\[
  \left\{
    \beta\ge0:
    \sup_{m\ge0}\|\mathbf C_m^{(5)}(\beta)\|_\infty<\infty
  \right\}
  =
  [\beta_{*,5},\infty).
\]
The definition \eqref{eq:ngasket-critical-beta} and
Proposition~\ref{prop:ngasket-completeness} now give
\(\beta_{*,5}
=\beta_{\mathrm c,5}^{\mathrm{cr}}
=\beta_{\mathrm c,5}\).
\end{proof}

\begin{proof}[Proof of Theorem~\ref{thm:ngasket-dynamics}]
For \(N\notin\{3,5,6,9\}\), the conclusion follows from
Proposition~\ref{prop:generic-ngasket-phase}; for \(N=3\) and \(N=5\),
it follows from Propositions~\ref{prop:sg-phase} and~
\ref{prop:pentagasket-phase}, respectively.
In each case the critical value appearing there is
\(\beta_{\mathrm c,N}\).

It remains to consider \(N=6\) and \(N=9\).  The proof of
Theorem~\ref{thm:ngasket} gives \(X_m=Y_m=t_m\), where
\[
  t_{m+1}=F(t_m),
  \qquad
  F(u)=u^{d_N+1}+u^{2d_N+1},
\]
whose unique positive fixed point is \(t_*=\tau^{1/d_N}\).
The computation in that proof and
Proposition~\ref{prop:ngasket-completeness} show that the initial value
equals \(t_*\) precisely when \(\beta=\beta_{\mathrm c,N}\), and the
orbit is then constant.
If
\(\beta>\beta_{\mathrm c,N}\), the initial value is smaller than
\(t_*\); since \(0<F(u)<u\) for \(0<u<t_*\), the orbit decreases to
zero.  If \(\beta<\beta_{\mathrm c,N}\), the initial value
is larger than \(t_*\); since \(F(u)>u\) for \(u>t_*\), the orbit
increases to infinity.  This proves all four cases.
\end{proof}

\refstepcounter{section}
\section*{Appendix \thesection\quad Off-critical rates}
\label{app:ngasket-rates}
\addcontentsline{toc}{section}{Appendix B. Off-critical rates}

Away from criticality, every crossing variable tends either
to zero or to infinity.  Taking logarithms turns each polynomial
recursion into a linear recursion with a bounded error.

The following lemma gives the asymptotic estimate used later in
Theorem~\ref{thm:ngasket-rates}.

\begin{lemma}
\label{lem:pf-asymptotics}
Let \(\mathsf M\) be a primitive nonnegative \(2\times2\) matrix with
Perron--Frobenius eigenvalue \(\Lambda>1\), and let
\[
  v_{m+1}=\mathsf Mv_m+e_m,
\]
where \((e_m)\) is bounded.  If both entries of \(v_m\) tend to
infinity, then there exist a positive right eigenvector \(v_+\) of
\(\mathsf M\) and a constant \(c>0\) such that
\[
  \Lambda^{-m}v_m\longrightarrow cv_+.
\]
\end{lemma}

\begin{proof}
Let \(v_+\) and \(w_+\) be positive right and left eigenvectors,
normalized by \(w_+^{\mathsf T}v_+=1\), and put
\(s_m=w_+^{\mathsf T}v_m\).  Multiplying the recursion by
\(w_+^{\mathsf T}\) gives the scalar relation
\[
  s_{m+1}=\Lambda s_m+w_+^{\mathsf T}e_m.
\]
Since \((e_m)\) is bounded and \(\Lambda>1\), the sequence
\[
  \Lambda^{-m}s_m
  =
  \Lambda^{-m_0}s_{m_0}
  +
  \sum_{j=m_0}^{m-1}
    \Lambda^{-j-1}w_+^{\mathsf T}e_j
\]
has a finite limit \(c\).

The second eigenvalue \(\Lambda_2\) of \(\mathsf M\) satisfies
\(|\Lambda_2|<\Lambda\).  Let
\[
  \Pi_+=v_+w_+^{\mathsf T},
  \qquad \Pi_\perp=I-\Pi_+.
\]
Then \(\Pi_+\) is the projection onto the Perron--Frobenius direction and
\(\mathsf M\Pi_\perp=\Lambda_2\Pi_\perp\).  Applying
\(\Pi_\perp\) and iterating gives
\[
  \Pi_\perp v_m
  =
  \Lambda_2^{m-m_0}\Pi_\perp v_{m_0}
  +
  \sum_{j=m_0}^{m-1}\Lambda_2^{m-1-j}\Pi_\perp e_j.
\]
Because \((e_j)\) is bounded and \(|\Lambda_2|<\Lambda\), division by
\(\Lambda^m\) makes the right-hand side tend to zero.  Combining this
with the limit of \(\Lambda^{-m}s_m\) gives
\[
  \Lambda^{-m}v_m\longrightarrow cv_+.
\]

It remains to show that \(c>0\).  Since both entries of \(v_m\) tend
to infinity and \(w_+>0\), one has \(s_m\to\infty\), while the
preceding limit gives \(c\ge0\).  If \(c=0\), the scalar recursion
could be solved backwards as
\[
  s_m
  =
  -\sum_{j=m}^{\infty}
    \Lambda^{m-1-j}w_+^{\mathsf T}e_j,
\]
so \((s_m)\) is bounded.  This contradicts the positivity of \(w_+\)
and the assumption that both entries of \(v_m\) tend to infinity.
Therefore \(c>0\).
\end{proof}

\begin{theorem}
\label{thm:ngasket-rates}
For \(N\ge3\) with \(4\nmid N\), define
\begin{equation}
  \begin{aligned}
  \Lambda_N^+
  &=
  \begin{cases}
  3,&N=3,\\[1mm]
  \dfrac{5+\sqrt{17}}2,&N=5,\\[2mm]
  k_N+\dfrac32+
  \sqrt{\left(k_N+\dfrac32\right)^2
        +2(k_N+\ell_N-3)},&N\notin\{3,5\},
  \end{cases}\\[2mm]
  \Lambda_N^-
  &=
  \begin{cases}
  2,&N=3,\\[1mm]
  1+\sqrt3,&N=5,\\[1mm]
  k_N+\dfrac{\ell_N-1}{2}+
  \sqrt{\left(k_N+\dfrac{\ell_N-1}{2}\right)^2
        -2(k_N+\ell_N-3)},&N\notin\{3,5\}.
  \end{cases}
  \end{aligned}
  \label{eq:ngasket-rate-exponents}
\end{equation}
If \(\beta<\beta_{\mathrm c,N}\), then for every entry
\(\mathsf c_m(\beta)\) of \(\mathbf C_m^{(N)}(\beta)\),
\[
  0<
  \lim_{m\to\infty}
  (\Lambda_N^+)^{-m}\log\mathsf c_m(\beta)
  <\infty.
\]
If \(\beta>\beta_{\mathrm c,N}\), then
\[
  0<
  \lim_{m\to\infty}
  (\Lambda_N^-)^{-m}\bigl(-\log\mathsf c_m(\beta)\bigr)
  <\infty.
\]
\end{theorem}

\begin{proof}
The phase diagram in Theorem~\ref{thm:ngasket-dynamics} ensures that all
crossing variables are eventually larger than one when
\(\beta<\beta_{\mathrm c,N}\), and eventually smaller
than one when \(\beta>\beta_{\mathrm c,N}\).  Thus all logarithms below have the
stated sign for sufficiently large \(m\).

Suppose first that \(N=4k+\ell\notin\{3,5\}\).  On the divergent side,
the larger power of \(Y_m\) in each sum in
\eqref{eq:ngasket-recursion} gives
\[
  \begin{pmatrix}\log X_{m+1}\\ \log Y_{m+1}\end{pmatrix}
  =
  \mathsf M_N^+
  \begin{pmatrix}\log X_m\\ \log Y_m\end{pmatrix}
  +e_m^+,
  \qquad
  \mathsf M_N^+
  =
  \begin{pmatrix}
  2&3k+\ell-2\\
  2&2k+1
  \end{pmatrix},
\]
where \(e_m^+\) is bounded; the error consists only of logarithms of
factors of the form \(1+Y_m^{-a}\), with \(a>0\).  On the side
converging to zero, the smaller power of \(Y_m\) gives
\[
  \begin{pmatrix}-\log X_{m+1}\\ -\log Y_{m+1}\end{pmatrix}
  =
  \mathsf M_N^-
  \begin{pmatrix}-\log X_m\\ -\log Y_m\end{pmatrix}
  +e_m^-,
  \qquad
  \mathsf M_N^-
  =
  \begin{pmatrix}
  2&k\\
  2&2k+\ell-3
  \end{pmatrix},
\]
where \(e_m^-\) is bounded for the same reason, now with factors of
the form \(1+Y_m^a\).  Both matrices are primitive.  Direct
calculation of their characteristic polynomials gives the two
Perron--Frobenius eigenvalues
\(\Lambda_N^\pm\) in
\eqref{eq:ngasket-rate-exponents}.  In particular,
\(\det\mathsf M_N^-=2(k+\ell-3)\), which produces the minus sign under
the square root in the formula for \(\Lambda_N^-\).
Lemma~\ref{lem:pf-asymptotics} now gives the required limit for both
crossing variables.

For \(N=3\), put \(r_m=y_m/x_m\) as in
\eqref{eq:sg-ratio}.  On the escaping side, \(r_m\) decreases to a
positive limit.  Indeed, once \(x_m>1\),
\[
  0<
  1-\frac{r_{m+1}}{r_m}
  =
  \frac{(1+r_m)^2}
       {(1+r_m)^2+x_m(1+2r_m)}
  \le\frac4{x_m},
\]
whereas \(x_{m+1}\ge x_m^3\).  Therefore
\(\sum_m(1-r_{m+1}/r_m)<\infty\), and the product
\(\prod_m(r_{m+1}/r_m)\) is positive.  Since \(r_m\) is decreasing,
it follows that \(r_m\) converges to a positive limit.  The logarithmic
recursion is consequently
\[
  \begin{pmatrix}\log x_{m+1}\\ \log y_{m+1}\end{pmatrix}
  =
  \begin{pmatrix}3&0\\2&1\end{pmatrix}
  \begin{pmatrix}\log x_m\\ \log y_m\end{pmatrix}
  +O(1).
\]
On the side converging to zero, \(r_m\to0\) and
\[
  \begin{pmatrix}-\log x_{m+1}\\ -\log y_{m+1}\end{pmatrix}
  =
  \begin{pmatrix}2&0\\2&1\end{pmatrix}
  \begin{pmatrix}-\log x_m\\ -\log y_m\end{pmatrix}
  +O(1).
\]
In both recursions, the first component has the form
\(\upsilon_{m+1}=\Lambda\upsilon_m+O(1)\), where
\(\upsilon_m\to\infty\) and
\(\Lambda=3\) or \(2\).  The elementary one-dimensional version of
Lemma~\ref{lem:pf-asymptotics} gives a finite positive limit for
\(\Lambda^{-m}\upsilon_m\); denote it by \(c_\upsilon>0\).  If
\(\omega_m\) denotes the second component, then
\[
  \omega_{m+1}=\omega_m+2\upsilon_m+O(1),
\]
and hence
\[
  \omega_m
  =
  \omega_{m_0}
  +2\sum_{j=m_0}^{m-1}\upsilon_j+O(m).
\]
It follows that
\[
  \Lambda^{-m}\omega_m
  \longrightarrow
  \frac{2c_\upsilon}{\Lambda-1}>0.
\]
This proves
\(\Lambda_3^+=3\) and \(\Lambda_3^-=2\) in
\eqref{eq:ngasket-rate-exponents}.

Finally, let \(N=5\), and use the notation
\(a_m,b_m,\xi_m,\widehat\xi_m,r_m,s_m\) from
Proposition~\ref{prop:pentagasket-phase}.  On the escaping side,
the proof of Proposition~\ref{prop:pentagasket-phase} shows that
\(s_m\) has a positive limit and that \(r_m\) has the same limit.
Taking logarithms in
\eqref{eq:pentagasket-dominant-pair} therefore gives
\[
  \begin{pmatrix}\log a_{m+1}\\ \log \xi_{m+1}\end{pmatrix}
  =
  \begin{pmatrix}2&2\\2&3\end{pmatrix}
  \begin{pmatrix}\log a_m\\ \log \xi_m\end{pmatrix}
  +O(1).
\]
On the side converging to zero, \(\xi_m\to0\).  Since
\(r_m,s_m\le1\), \eqref{eq:pentagasket-ratios} gives
\[
  H_m\le12\xi_m^3\longrightarrow0,
  \qquad
  s_{m+1}\le H_m\longrightarrow0,
  \qquad
  r_{m+1}\le s_m\longrightarrow0.
\]
Thus \(r_m,s_m\to0\), and
\eqref{eq:pentagasket-dominant-pair} gives
\[
  \begin{pmatrix}-\log a_{m+1}\\ -\log \xi_{m+1}\end{pmatrix}
  =
  \begin{pmatrix}2&1\\2&0\end{pmatrix}
  \begin{pmatrix}-\log a_m\\ -\log \xi_m\end{pmatrix}
  +O(1).
\]
The two Perron--Frobenius eigenvalues are
\((5+\sqrt{17})/2\) and \(1+\sqrt3\), respectively.
Lemma~\ref{lem:pf-asymptotics} gives the stated limits for
\(a_m\) and \(\xi_m\).  On the escaping side, the positive limits of
\(r_m\) and \(s_m\) give the same limits for \(b_m\) and
\(\widehat\xi_m\).
On the side converging to zero,
\eqref{eq:pentagasket-ratios} gives
\[
  -\log s_{m+1}
  =
  -\log s_m+3(-\log \xi_m)+O(1),
  \qquad
  -\log r_{m+1}
  =
  -\log s_m-\log \xi_m+O(1).
\]
Write
\[
  (\Lambda_5^-)^{-m}(-\log \xi_m)\longrightarrow\gamma_\xi>0.
\]
The first relation is a scalar inhomogeneous recursion.  Summing it
and using the limit for \(-\log \xi_m\) gives
\[
  (\Lambda_5^-)^{-m}(-\log s_m)
  \longrightarrow
  \frac{3\gamma_\xi}{\Lambda_5^--1}>0.
\]
Substitution in the second relation then gives a finite positive limit
for
\((\Lambda_5^-)^{-m}(-\log r_m)\).  Since
\[
  -\log b_m=-\log a_m-\log r_m,
  \qquad
  -\log \widehat\xi_m=-\log \xi_m-\log s_m,
\]
the two remaining crossing variables have finite positive normalized
limits as well.  This completes the proof.
\end{proof}

\smallskip
\noindent
\textsc{School of Mathematics, Nanjing University, Nanjing, 210093,
P. R. China.}\\
\textit{Email address:} \texttt{huaqiu@nju.edu.cn}

\smallskip
\noindent
\textsc{School of Mathematics, Nanjing University, Nanjing, 210093,
P. R. China.}\\
\textit{Email address:} \texttt{602025210014@smail.nju.edu.cn}

\end{document}